\documentclass[reqno]{amsart}
\usepackage{tkz-euclide}
\usepackage{hyperref}
\usepackage{amsmath}  
\usepackage{amsfonts}
\usepackage{amssymb} 
\usepackage{mathrsfs} 
\usepackage{graphicx}  
\usepackage{bm}
\usepackage{tikz} 
\usetikzlibrary{svg.path} 

\usetikzlibrary{patterns}
\usetikzlibrary{calc} 
\usepackage{xcolor}  
\usetikzlibrary{arrows.meta} 
\usepackage{amsthm} 
\usepackage{float}
\usepackage{enumitem}
\usepackage[english]{babel}
\usepackage[font=footnotesize, labelfont=bf]{ caption}
\usepackage{ esint }
\usepackage{stackengine,scalerel,graphicx}
\stackMath

\definecolor{trueblue}{rgb}{0.0, 0.45, 0.81}
\definecolor{truegreen}{rgb}{0.13, 0.55, 0.13}

\newcommand{\NNN}{\color{black}} 
\newcommand{\MMM}{\color{black}} 
\newcommand{\BBB}{\color{black}} 
 
\newcommand{\RRR}{\color{black}} 
\newcommand{\EEE}{\color{black}} 

\newcommand{\OOO}{\color{black}} 
\newcommand{\AAA}{\color{black}}
\newcommand{\KKK}{\color{black}}
\newcommand{\QQQ}{\color{black}}
\newcommand{\ZZZ}{\color{black}}

\newcommand{\eps}{\varepsilon} 
\newcommand{\dx}{\, {\rm d}x}

\newcommand{\ggamma}{\bm{\gamma}}

\theoremstyle{plain}
\begingroup
\newtheorem{theorem}{Theorem}[section]
\newtheorem{lemma}[theorem]{Lemma}
\newtheorem{example}[theorem]{Example}
\newtheorem{remark}[theorem]{Remark}
\newtheorem{proposition}[theorem]{Proposition}
\newtheorem{corollary}[theorem]{Corollary}
\endgroup

\newenvironment{step}[1]{\underline{Step #1}.}{}

\theoremstyle{definition}
\begingroup
\newtheorem{definition}[theorem]{Definition}
\endgroup

\renewcommand{\tilde}{\widetilde}

\renewcommand{\d}{ \mathrm{d}}

\DeclareMathOperator{\dist}{dist}

\numberwithin{equation}{section}
\newcommand{\N}{\mathbb{N}}

\newcommand{\R}{\mathbb{R}}

\renewcommand{\L}{\mathcal{L}}

\newcommand{\x}{{\times}}

\newcommand{\M}{\mathfrak{M}}

\usepackage[normalem]{ulem}

\begin{document}

\title[Geometric rigidity in variable domains and derivation of linearized models]{Geometric rigidity in variable domains and derivation of linearized models for elastic materials with free surfaces}

\author[M. Friedrich]{Manuel Friedrich} 
\address[Manuel Friedrich]{Department of Mathematics, Friedrich-Alexander Universit\"at Erlangen-N\"urnberg. Cauerstr.~11,
D-91058 Erlangen, Germany, \& Mathematics M\"{u}nster,  
University of M\"{u}nster, Einsteinstr.~62, D-48149 M\"{u}nster, Germany}
\email{manuel.friedrich@fau.de}

\author[L. Kreutz]{Leonard Kreutz}
\address[Leonard Kreutz]{Department of Mathematical Sciences, Carnegie Mellon University, Pittsburgh, PA 15213, USA}
\email{lkreutz@andrew.cmu.edu}

\author{Konstantinos Zemas}
\address[Konstantinos Zemas]{Applied Mathematics M\"unster, University of M\"unster\\
Einsteinstrasse 62, 48149 M\"unster, Germany}
\email{konstantinos.zemas@uni-muenster.de}


\begin{abstract}
We present a quantitative geometric rigidity estimate \BBB in dimensions $d=2,3$ \EEE generalizing the celebrated result by
{\sc
Friesecke, James, and M\"uller} \cite{friesecke2002theorem} to the setting of variable domains. Loosely speaking, we show that for each $y \in H^1(U;\R^d)$ and for each connected component of an open, bounded set $U\subset \mathbb{R}^d$,  the $L^2$-distance of $\nabla y$  from a single rotation can be  controlled up to a constant  \EEE by its   \EEE  $L^2$-distance from the group $SO(d)$, with the  \EEE  constant not depending on the precise shape of  $U$, but only on \BBB an integral curvature functional related to  \EEE$\partial U$. We further show that for linear strains the estimate can be refined, leading to a uniform control independent of the set $U$. The estimate can be used to establish compactness in \BBB the space of generalized special functions of bounded deformation \EEE ($GSBD$) for sequences of  displacements related to deformations  with uniformly bounded elastic energy. As an application, we rigorously  derive linearized models for nonlinearly elastic materials with free surfaces \BBB  by \EEE means of $\Gamma$-convergence. In particular, we study energies  related to  epitaxially strained crystalline films and \BBB to \EEE the  formation of material voids inside elastically stressed solids.
\end{abstract}


\maketitle

\section{Introduction}\label{section:introduction}


\EEE Rigidity estimates   have a long history dating back to {\sc
Liouville}'s fundamental result  which states that \BBB smooth mappings are necessarily  affine if their gradient is a rotation everywhere. After various 
 \EEE generalizations of this classical   theorem   over the last decades \cite{John:1961,   kohn-rig, Reshetnyak:1961},  \OOO a fundamental  \EEE breakthrough was achieved by {\sc
Friesecke, James, and M\"uller} \cite{friesecke2002theorem} with their celebrated quantitative geometric rigidity result in  nonlinear elasticity theory. In its basic form, the estimate states that in any dimension \AAA $d\geq 2$\EEE, for a mapping $y \in H^1(\Omega;\R^d)$ there exists  a  \OOO  corresponding  \EEE rotation $R \in SO(d)$ such that
\begin{align}\label{eq: rigidity}
\int_\Omega |\nabla y - R |^2 \, {\rm d}x \le C \int_\Omega\dist^2\big(\nabla y,SO(d)\big) \, {\rm d} x
\end{align}
for a constant $C>0$ only depending on the (sufficiently regular) \AAA bounded \EEE domain $\Omega$.  This result is fundamental in the analysis of variational models in nonlinear elasticity, as it provides compactness for sequences of deformations and corresponding  displacements with uniformly bounded elastic energy in a sharp quantitative fashion. In fact, it has proved to be the cornerstone for rigorous derivations of lower dimensional theories for plates, shells, and rods in various scaling regimes  \cite{Mora3,  friesecke2002theorem, hierarchy,  Mora4, Mora, Mora2}, and for providing  relations between geometrically nonlinear and linear models in elasticity \cite{DalMasoNegriPercivale:02}. The estimate  \eqref{eq: rigidity} was generalized in various directions to analyze variational models for materials with elastic  and plastic behavior. Among others, we mention results for mixed growth conditions \cite{Conti-Dolzmann-Muller:14}, incompatible fields \cite{conti.garroni, lauteri.luckhaus,  Muller-Scardia-Zeppieri:14}, and settings involving multiple energy wells \cite{Chaudhuri, Chermisi-Conti,  conti.schweizer, davoli.friedrich,  De Lellis, Jerrard-Lorent, Lorent}.

\textbf{Background and motivation:} In this paper, we are interested in rigidity estimates for nonlinearly elastic energies  involving free surfaces.  Our motivation lies in studying models in the framework of \emph{stress driven rearrangement instabilities} (SDRI),  i.e.,  morphological instabilities of interfaces between elastic phases generated by the
competition between elastic \BBB bulk \EEE and surface energies, including many different  phenomena such as brittle fracture, formation of material voids inside elastically stressed solids, or  hetero-epitaxial growth of elastic thin films. We refer to  \cite{Bourdin-Francfort-Marigo:2008, GaoNix99, Grin86, Grin93, KhoPio19, SieMikVoo04, Spe99} for an overview of some mathematical and physical literature. From a variational viewpoint, the  common feature of functionals describing SDRI is  the presence of both stored elastic \BBB energies \OOO in the \EEE bulk  and surface energies. \BBB This \EEE  can be formulated in the language of \emph{free discontinuity problems} \cite{DeGiorgi-Ambrosio:1988}, where the set of discontinuities  is not preassigned, but determined from an energy minimization principle.

In this context, a major challenge  \OOO in obtaining \EEE rigidity results lies in the fact that the functional setting goes beyond Sobolev spaces and requires functions allowing for jump discontinuities, more precisely  \emph{\OOO(special) \EEE functions of bounded variation} ($SBV$), see \cite[Section 4]{Ambrosio-Fusco-Pallara:2000},  or  \emph{\OOO(special) \EEE functions of bounded deformation} ($SBD$), see \cite{Ambrosio-Coscia-Dal Maso:1997, DalMaso:13}. Moreover, the formulation is genuinely more involved compared to \eqref{eq: rigidity},  as the domain may be disconnected by the  jump set into various components, and therefore at most \emph{piecewise rigidity} results can be expected, i.e., on each \OOO  connected  \EEE component of the domain without the jump set \EEE the deformation is close to a possibly different rigid motion. 

The last years have witnessed a tremendous progress for rigidity results in the linearly elastic setting \cite{FinalKorn,Chambolle-Conti-Francfort:2014, Iu3, Conti-Focardi-Iurlano:15, Friedrich:15-3, Friedrich:15-4}, generalizing suitably the classical Korn's inequality to $SBD$, and controlling also the surface contributions of the energy. The situation in the geometrically nonlinear setting, however, is by far less well understood. A first key step in this direction was achieved by {\sc
Chambolle, Giacomini, and Ponsiglione} \cite{Chambolle-Giacomini-Ponsiglione:2007} showing a Liouville-type result for brittle materials  storing no elastic energy.   To the best of our knowledge, to date counterparts of the quantitative estimate \eqref{eq: rigidity} are limited to dimension two \cite{Friedrich-Schmidt:15} or, in general dimensions, \EEE to a model for \emph{nonsimple materials} \cite{higherordergriffith} where the elastic energy  depends additionally on the second gradient of the deformation, cf.\  \cite{Toupin:62}. The latter results have been employed successfully to identify  linearized models in the small-strain limit \cite{Friedrich:15-2, higherordergriffith}, and to perform dimension reduction \cite{Schmidt17}.

In this paper, we prove a novel \emph{quantitative geometric rigidity result for variable domains} in dimensions $d=2,3$, see Theorem~\ref{prop:rigidity}. \OOO While our proof strategy \BBB in principle allows \EEE to establish the result also in higher dimensions, \BBB there is a single missing point, namely a specific geometric \EEE estimate of \BBB possible independent interest, \EEE see Remark~\ref{obstacle_higher_dimensions_generalization}. \OOO In the physically relevant dimensions $d=2,3$, \EEE we believe that  \OOO our result \EEE may be applicable in a variety of different contexts, in particular to study problems on dimension reduction. In the present paper, as a first application, we employ the estimate to rigorously derive linearized models for elastic materials with free surfaces.

\textbf{The rigidity estimate:} Loosely speaking, given a fixed open, bounded set $\Omega \subset \R^d$, $d=2,3$, our main result states the following: for every  regular \OOO open  \EEE set $E\OOO\subset \Omega\EEE$,  we can \EEE find a \emph{thickened set}   \OOO $E\subset E^*\subset \Omega$  \EEE such that 
\begin{align}\label{eq: setsetset}
{\rm (i)} \ \ \mathcal{L}^d(E^* \setminus E) \ll 1, \quad \quad \quad {\rm (ii)} \ \  \big|\mathcal{H}^{d-1}( \partial E^* \cap \Omega) -  \mathcal{H}^{d-1}( \partial E \cap \Omega)   \big| \ll 1\,,
\end{align}
where $\mathcal{L}^d$ and $\mathcal{H}^{d-1}$ denote the $d$-dimensional Lebesgue and $(d-1)$-dimensional Hausdorff measure, respectively,  and for each   $y \in H^1(\OOO\Omega\setminus \overline{E}\EEE;\R^d)$ with \textit{elastic energy} $\eps:= \int_{\OOO\Omega\setminus \overline{E}\EEE}\dist^2(\nabla y,SO(d)) \, {\rm d} x  $  \OOO there exists \EEE  a proper rotation $R \in SO(d)$ such that
\begin{align}\label{eq: rigidity-new}
\begin{split}
{\rm (i)} & \ \ \int_{\Omega \setminus \overline{E^*}} \Big|\OOO\mathrm{sym}(R^T\nabla y)  - {\rm Id} \Big|^2 \, {\rm d}x \le C \big(1+ C_{\partial E}^{\rm curv}\eps\big)\OOO\varepsilon\,,\\
{\rm (ii)} & \ \ 
\int_{\Omega \setminus \overline{E^*}} |\nabla y - R |^2 \, {\rm d}x \le C_{\partial E}^{\rm curv} \eps\,,
\end{split}
\end{align}
where \OOO $\mathrm{sym}(F)\EEE:=\frac{1}{2}(F+F^T)$ for $F\in \mathbb{R}^{d\times d}$, \EEE ${\rm Id} \in \R^{d\times d}$ denotes the identity matrix and \EEE  $C>0$ is a constant depending on $\Omega$ \OOO but not on $E$. Eventually,  \EEE  $C_{\partial E}^{\rm curv}>0$ is a  constant depending on  \OOO a suitable integral curvature functional \EEE of $\partial E$ and \EEE  can \BBB possibly \EEE become large \EEE as the \OOO  curvature of $\partial E$ becomes large. More precisely, if $\Omega \setminus \overline{E^*}$ consists of different connected components, the rotation $R$ may be different for each connected component,   \BBB cf.\ also  \EEE the piecewise estimate \cite[\BBB Theorem 1.1]{Chambolle-Giacomini-Ponsiglione:2007}.

Here, the role played by the unknown (i.e., variable) set $E$ depends on the  application, e.g., it may model  material voids inside an elastic material with reference domain $\Omega$. As $E$ is regular, an estimate of the form \eqref{eq: rigidity-new} would in general follow directly from \eqref{eq: rigidity} for a \emph{constant depending on $E$}. \EEE We therefore emphasize that the essential point \OOO of our estimate  \EEE is that the constant $C$ is \emph{independent} of $E$ and $C_{\partial E}^{\rm curv}$ does not depend  on the precise shape of $E$, but   only on 
\begin{align}\label{eq: curvatureeee}
\int_{\partial E\cap \Omega} |\bm{A}|^q\, {\rm d}\mathcal{H}^{d-1}
\end{align} 
for some fixed $q \geq d-1$, where  $\bm{A}$ denotes the second fundamental form of $\partial E$. (The choice $q \ge d-1$ is essential for the proof, see Lemma \ref{lemma:slicing} and Example \ref{ex: q!!!}.) 

Given a uniform control on the \OOO above  \EEE curvature \OOO term, \EEE  \eqref{eq: rigidity-new}(ii) yields the exact counterpart of the estimate \eqref{eq: rigidity}, generalized to \OOO the setting of  \EEE variable domains. Moreover, \eqref{eq: rigidity-new}(i), \OOO say for simplicity \EEE for $R={\rm Id}$, shows that the $L^2$-norm of the \AAA symmetric part of $\nabla y-\rm{Id}$ \EEE can be controlled by the nonlinear elastic energy independently of $C_{\partial E}^{\rm curv}$, provided that $\eps$ is small compared to the inverse of $C_{\partial E}^{\rm curv}$. The latter property will allow us to obtain a uniform control on linear strains $e(u):= \frac{1}{2} (\nabla u + \nabla u^{T})$ for displacements $u = y-{\rm id}$, where ${\rm id}$ denotes the identity mapping. \EEE This naturally leads to effective descriptions in the realm of $SBD$  functions \cite{DalMaso:13}, for which only symmetrized gradients are controlled.  

\textbf{Proof strategy and discussion:} \BBB The core of the proof consists in \EEE a geometric construction to  \OOO modify  \EEE the set $E$, along with the proof strategy \OOO  for \eqref{eq: rigidity} \EEE  devised in \cite{friesecke2002theorem}. More specifically, we find a thickened set $E^* \supset E$ consisting essentially of a union of cubes of a specific sidelength $\rho>0$, which depends only on the size of the  curvature term in \eqref{eq: curvatureeee}. As already observed  in \cite{friesecke2002theorem}, the rigidity constant \OOO of $\Omega\setminus \overline{E^*}$  \EEE only depends on $\Omega$ and $\rho$, which implies \OOO\eqref{eq: rigidity-new}\EEE(ii). To derive \OOO\eqref{eq: rigidity-new}\EEE(i), we use  \OOO\eqref{eq: rigidity-new}\EEE(ii) and the fact that the tangent space of the smooth manifold $SO(d)$ at the identity \OOO matrix  \EEE is given by the linear space of all skew-symmetric matrices, which in particular implies  that  
$$\big|(F^T + F)/2- {\rm Id}\big| =   {\rm dist}(F, SO(d)) + {\rm O}(|F- {\rm Id}|^2)\,.$$
Here, as in \cite{friesecke2002theorem}\AAA, \EEE we also reduce the problem to harmonic  \OOO mappings \EEE in order to control higher order terms  through an $L^2-L^{\infty}$ estimate obtained  by the mean value property.  After controlling the symmetric \OOO part of the  \EEE gradient, the last step in the proof of  \eqref{eq: rigidity} in    \cite{friesecke2002theorem} consists  in    applying Korn's inequality to obtain \eqref{eq: rigidity}. This, however, is not possible in our setting as the constant in Korn's inequality again depends on the shape of the domain $\Omega \setminus \overline{E^*}$  \BBB which would only give back an estimate of the form \eqref{eq: rigidity-new}(ii). \EEE  In conclusion,  \OOO even in the regime where the elastic energy is sufficiently small with respect to the curvature energy term in \eqref{eq: curvatureeee}, \EEE uniform bounds independent of $E$ can only be obtained for symmetrized gradients but not for full gradients. Simple examples show that  estimate \eqref{eq: rigidity-new}(ii) is indeed sharp, see Example~\ref{ex: sharp}. 

Whereas \eqref{eq: rigidity-new} can be derived by adapting the original strategy devised in \cite{friesecke2002theorem}, the real novelty of our work lies in the construction of the \emph{thickened set} $E^* \supset E$. In the application to variational models for SDRI presented below, estimate \eqref{eq: setsetset} is essential to ensure that  the thickening of the set does \OOO not affect \EEE  asymptotically $E$ in volume and surface measure.  In a first auxiliary step, \OOO in order \EEE to ensure that \BBB$\Omega \setminus \overline{E^*}$ is essentially a union of cubes with equal sidelength, we tessellate $\R^d$ with cubes of sidelength $\rho\OOO >0\EEE$ and add to $E$ all cubes intersecting $\partial E$, \OOO the  \EEE so-called \emph{boundary cubes}. In order to verify \eqref{eq: setsetset}(i), one needs to control the number of boundary cubes. This is highly nontrivial as the boundary $\partial E$ might become extremely complex, exhibiting  \emph{thin spikes} or  microscopically small components  with small surface measure on different length scales, see Figure~\ref{fig:thinspikes}. The key ingredient is Lemma~\ref{lemma:slicing} which, in rough terms, states that for a specific choice of the sidelength $\rho$, in each boundary cube $Q_\rho$ we get that $\mathcal{H}^{d-1}(\partial E \cap Q_\rho)$ or $\int_{\partial E \cap Q_\rho} |\bm{A}|^q\, {\rm d}\mathcal{H}^{d-1}$ is \BBB at least of order  \EEE $\rho^{d-1}$.
\begin{figure}
\begin{tikzpicture}

\draw(2.25,2.25) node{$\Omega$};
\draw(1,1.5) node{$E$};

\draw[clip] plot [smooth cycle] coordinates {(0,0)(2,0)(2,2)(0,2)};

\begin{scope}[shift={(1,1.75)},scale=.25]
\draw[fill=gray, domain=-.5:.5, smooth, variable=\x, xscale=.5,yscale=1] plot ({\x}, {40*\x*\x*\x*\x});

\end{scope}

\begin{scope}[shift={(.6,1.7)},scale=.2]
\draw[fill=gray, domain=-.8:.8, smooth, variable=\x, xscale=.5,yscale=1] plot ({\x}, {40*\x*\x*\x*\x});

\end{scope}

\begin{scope}[shift={(1.6,1.4)},scale=.1]
\draw[fill=gray, domain=-1.8:1.8, smooth, variable=\x, xscale=.5,yscale=1] plot ({\x}, {40*\x*\x*\x*\x});

\end{scope}

\begin{scope}[shift={(1.3,1.6)},scale=.1]
\draw[fill=gray, domain=-1.8:1.8, smooth, variable=\x, xscale=.5,yscale=1] plot ({\x}, {40*\x*\x*\x*\x});

\end{scope}

\begin{scope}[shift={(1.45,1.6)},scale=.14]
\draw[fill=gray, domain=-1.8:1.8, smooth, variable=\x, xscale=.5,yscale=1] plot ({\x}, {40*\x*\x*\x*\x});

\end{scope}

\draw plot [smooth cycle] coordinates {(0,0)(2,0)(2,2)(0,2)};

\draw[fill=gray](1,.5) circle(.05);
\draw[fill=gray](1.2,.3) circle(.04);
\draw[fill=gray](1.1,.45) circle(.035);
\draw[fill=gray](1.3,.55) circle(.03);
\draw[fill=gray](1.25,.43) circle(.02);
\draw[fill=gray](1.15,.63) circle(.02);
\draw[fill=gray](1.12,.53) circle(.01);
\draw[fill=gray](1.2,.58) circle(.01);
\draw[fill=gray](1.22,.48) circle(.015);

\end{tikzpicture}
\caption{A possible void set $E$, depicted in gray, that contains thin spikes or small components that may prevent rigidity for deformations defined on the set $\Omega \setminus \overline{E}$.}
\label{fig:thinspikes}
\end{figure}
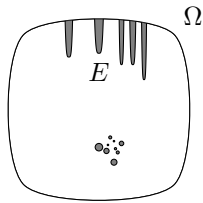 

Loosely speaking, this means that spikes or microscopic components of $\partial E$ accumulating on scales smaller than $\rho$ induce too high curvature \OOO energy, \EEE and can therefore be excluded. \OOO Let us emphasize here that establishing the higher dimensional version of the last assertion for closed hypersurfaces is exactly the missing ingredient to generalize our \BBB result \OOO  to any space dimension. \EEE

Subsequently, the construction of $E^*$ needs to be refined in order to satisfy also  \eqref{eq: setsetset}(ii). To this end, we use the property that under a specific \OOO area and  \EEE curvature bound in a boundary cube  $Q_\rho$, the surface  $\partial E \cap Q_\rho$  \BBB inside a smaller cube \EEE  is essentially a finite union of graphs of Lipschitz functions \OOO with appropriate a priori estimates. \EEE Based on this,  a direct geometric construction can be performed to thicken the sets. Whereas this \OOO  local  \EEE graphical approximation of $\partial E$ is elementary in dimension  \OOO $d=2$  \EEE (see Lemma~\ref{lemma: curve graph}), in dimension \OOO $d=3$ and for $q=2$, \EEE it is a deep \OOO $\varepsilon$-regularity  \EEE result  in   \OOO geometric analysis  due to {\sc
Simon} \cite{Simon1993Willmore}, see Lemma~\ref{Simons_lemma} \BBB and also Remark~\ref{obstacle_higher_dimensions_generalization}. \EEE 

 \EEE We note that the passage to a thickened set $E^*$ is not due to our specific proof strategy, but is indeed necessary for a uniform rigidity estimate. Simple examples, where $\Omega \setminus \overline{E}$ is connected but only through a \emph{thin tunnel}, show that \eqref{eq: rigidity} \OOO (with a uniform constant) \EEE  can be violated for deformations concentrating elastic energy in the tunnel,  see Example~\ref{ex: tunnel}.

Our result appears to address \EEE an immediate situation between the result in the Sobolev setting \cite{friesecke2002theorem} and the abovementioned results \cite{FinalKorn, Chambolle-Giacomini-Ponsiglione:2007,  Friedrich:15-4,  Friedrich-Schmidt:15} in \OOO the  \EEE function spaces $SBV$ and $SBD$, where  additional difficulties are present due to \OOO the \EEE lack of regularity of deformations. Indeed, in our setting\AAA, \EEE deformations are still Sobolev, yet defined on sets with free boundary. By approximation results in $SBV$ and $SBD$ \cite{Crismale2, Cortesani-Toader:1999} however, jump sets can be regularized and can be \EEE covered by regular sets $E$.  In this sense,  our estimate is in  spirit  closer to results in $SBV$ and $SBD$, and along the proof we  encounter many intricacies present in these function spaces concerning the topology and geometry of jump sets.

As a final comment on the rigidity result, let us emphasize that the idea of deriving uniform estimates for variable domains under certain assumptions on the sets $E$ (or assumptions on the geometry of the jump set) is not new but has been used in a variety of free discontinuity problems, see e.g.\ \cite{Lazzaroni, NegriToader:2013, Rondi}. These models, however, are based on considering very specific classes of discontinuity sets with certain geometric features such as well-separateness. Our approach instead, readily relies on a curvature control of the form \eqref{eq: curvatureeee} which  can be implemented easily in a variational model. Indeed,  curvature regularizations are widely used in the mathematical and physical literature \BBB of \EEE SDRI models, including the description of (the evolution of) elastically stressed thin films or  material voids, see \cite{AnGurt, burger, Carlo,  FonFusLeoMor14, FonFusLeoMor15, GurtJabb, Herr,   Pio14, voigt,  SieMikVoo04}.

\textbf{Applications to linearization of variational SDRI models:}  We employ the rigidity result to derive a rigorous connection between models for hyperelastic materials in nonlinear (finite) elasticity and their linear (infinitesimal) counterparts. Although being a classical topic in elasticity theory, this relation has been derived rigorously via \EEE $\Gamma$-convergence \cite{Braides:02, DalMaso:93} only comparatively recently by  {\sc 
Dal Maso, Negri, and Percivale} \cite{DalMasoNegriPercivale:02}. The authors performed a   nonlinear-to-linear analysis in terms of suitably rescaled  displacement fields and proved the convergence of minimizers for corresponding boundary value problems. Their study has been extended into various directions, ranging from  models for \EEE incompressible materials \cite{Jesenko-Schmidt:20, edo}, \BBB from \EEE atomistic models \cite{Braides-Solci-Vitali:07, Schmidt:2009}, to multiwell energies \cite{alicandro.dalmaso.lazzaroni.palombaro, Schmidt:08}, plasticity \cite{Ulisse},  viscoelasticity \cite{MFMK}, \BBB or fracture \cite{Friedrich:15-2, higherordergriffith}. \EEE In all of these results, the rigidity estimate  \eqref{eq: rigidity} or one of its variants plays a key role  \OOO in  \EEE establishing    compactness.    

Despite the huge body of literature on variational  SDRI models, in particular on epitaxially   strained  elastic thin films  (see e.g.\ \cite{BonCha02, Bon12, ChaSol07, Crismale, DavPio17, FonFusLeoMor07}) and material voids  \cite{BraChaSol07, Crismale, FonFusLeoMil11,  santilli}, results on rigorous relations between nonlinear and linear theories are scarce. To the best of our knowledge, the only available result is the recent work  \cite{KrePio19} on two-dimensional elastic thin films. In this setting, one can resort to the Hausdorff topology for sets,  which in turn allows to apply the rigidity estimate \eqref{eq: rigidity}. Yet, the situation in higher dimensions and in the case of \OOO a \EEE possibly unbounded number of surface components (as in \OOO the \EEE case of material voids) is much more intricate, and a more general rigidity result of the form \eqref{eq: rigidity-new} is indispensable.

We consider functionals defined on pairs of function-set featuring nonlinear elastic \BBB bulk \EEE  and surface contributions of the form  
\begin{equation*}
F_{\OOO \delta}(y,E)\EEE:= \frac{1}{\delta^2} \int_{\Omega \setminus \OOO \overline{E}\EEE}  W(\nabla y) \, {\rm d}x + \int_{\partial E\cap\Omega}  \varphi(\nu_{E}) \, \d\mathcal{H}^{d-1} 
+ \gamma_\delta\int_{\partial E \cap \Omega} |\bm{A}|^q \, {\rm d}\mathcal{H}^{d-1}\,, 
\end{equation*}
where $E \subset \Omega$ is \BBB open and  regular, \EEE  $q \ge d-1$ 
, $y \in H^1(\Omega \setminus \OOO \overline{E}\EEE;\R^d)$, and $\gamma_\delta \to 0$ as $\delta \to 0$.  The first part of the functional represents the elastic energy, where  $W$ is a  frame-indifferent stored energy density and $\delta >0$ represents the scaling of the strain.   The surface energy consists of a perimeter term  depending on a (possibly anisotropic) density $\varphi$   evaluated at the outer unit normal $\nu_{E}$ to $\OOO\partial \EEE E$, and a curvature regularization \OOO term. \EEE \OOO In the case $d=3$, $q=2$, \EEE we will also discuss variants where  $|\bm{A}|^{\OOO 2\EEE}$ is replaced by a mean curvature regularization corresponding to the \emph{Willmore energy}.    The setting is complemented with prescribed Dirichlet boundary conditions which \BBB induce \EEE a stress in the solid. 

This energy and its relaxation were studied in   \cite{BraChaSol07, ChaSol07}  without the \OOO curvature \EEE regularization term, where, depending on  the application, $E$  describes material voids in elastically stressed solids or the complement of an elastic thin film.  In this paper,  we are interested in deriving an  effective description in the small-strain limit $\delta \to 0$, in terms of displacement fields $u = \frac{1}{\delta}(y - {\rm id})$. We prove that the $\Gamma$-limit \OOO of the functionals $(F_{\delta})_{\delta >0}$ is of the form 
\begin{equation*}
\mathcal{F}_0(u,E) :=   \frac{1}{2}\int_{\Omega\setminus E} \mathcal{Q}(e(u)) \,\mathrm{d}x  + \int_{\partial^* E\cap\Omega}  \varphi  (\nu_{E}) \,{\rm d}\mathcal{H}^{d-1}  + \int_{J_u \setminus \partial^* E}2\, \varphi(\nu_u)  \,{\rm d}\mathcal{H}^{d-1},
\end{equation*}
i.e., coincides with the relaxation of the models studied in \cite{Crismale}. Here, the map $u$ lies in  $ GSBD^2(\Omega)$ (see Appendix \ref{sec: GSBD}), where  $e(u)$ denotes the approximate symmetrized gradient and $J_u$ is the jump set   with corresponding  measure-theoretical unit normal $\nu_u$. Moreover, $E$ is a set of finite perimeter with essential boundary $\partial^* E$ \OOO and outward pointing measure-theoretical unit normal $\nu_{E}$. \EEE  The elastic energy depends on the linear strain $e(u)$ in terms of the \EEE quadratic form \BBB $\mathcal{Q}=D^2W({\rm Id})$. \EEE Besides the linearization of the elastic term, a further relaxation occurs in the surface energy: parts of the set $E$ may collapse into a discontinuity $J_u$ of the displacement $u$, and are  counted twice in the energy.
 Eventually, our assumption  $\gamma_\delta \to 0$ as $\delta\to 0$ implies that  the curvature regularization of the nonlinear energy does not affect the linearized limit.

%
%
%
%

\textbf{Organization of the paper and notation:}  
The paper is organized as follows: Section \ref{sec: rigidi} is devoted to the rigidity estimate. We give an exact statement of our result along with several extensions in Subsection \ref{sec: state-rig}. The proof is contained in Subsections \ref{sec: main proof}--\ref{sec: thickening2}. In Section \ref{sec: applications} we present our applications to the linearization of SDRI models. Subsections \ref{sec: results1}--\ref{sec: results2} address the case of material voids in elastically stressed solids and epitaxially strained thin  films, respectively. The proofs are given in Subsections \ref{sec: compre}--\ref{sec: compre2}. Finally, in Appendix \ref{sec: appi} we prove some elementary lemmata used in the proofs of our main results, \EEE and collect \BBB basic \EEE properties of the space $GSBD^2$.

We close the Introduction with some basic notation.   Given   $\Omega \subset \R^d$ open, $d=2,3$, we denote by $\mathfrak{M}(\Omega)$ the \AAA collection of all \EEE measurable subsets of $\Omega$. By $\mathcal{A}_{\rm reg}(\Omega)$ we indicate the \AAA collection of all \EEE open subsets $E \subset \Omega$ such that $\partial E \cap \Omega $ is a  $(d-1)$-dimensional \BBB $C^2$-submanifold \EEE of $\mathbb{R}^d$. \BBB Manifolds and functions of $C^2$-regularity will be called \emph{regular} in the following. Given $A \in \mathfrak{M}(\Omega)$, we denote by ${\rm int}(A)$  its   interior and by $A^c= \R^d \setminus A$   its  complement. The diameter of $A$  is denoted by ${\rm diam}(A)$.  Moreover, for $r>0$ we let
\begin{align}\label{eq: thick-def}
(A)_r := \lbrace x\in \R^d\colon \, \dist(x,A) < r\rbrace.
\end{align}  
Given $A, B \in \mathfrak{M}(\Omega)$, we write $A \subset \subset B$ if $\overline{A} \subset B$. The Hausdorff distance of $A$ and $B$ is denoted by ${\rm dist}_\mathcal{H}(A,B)$ and we write $A\triangle B = (A\setminus B) \cup (B\setminus A)$ for the symmetric difference.    By  ${\rm id}$ we denote the identity mapping on $\R^d$ and by ${\rm Id} \in \R^{d\times d}$   the identity matrix.  For each $F \in \R^{d \times d}$ we let ${\rm sym}(F) = \frac{1}{2}\left(F+F^T\right)$, and we define  $SO(d) := \lbrace F\in\R^{d \times d}\colon F^TF = {\rm Id}, \, \det F = 1\rbrace$. Moreover, we denote by $\R^{d \times d}_{\rm sym}$ and $\R^{d \times d}_{\rm skew}$ the set of symmetric and skew\AAA-\EEE symmetric matrices, respectively. We further write $\mathbb{S}^{d-1} \AAA:\EEE= \lbrace \nu \in \R^d\colon \, |\nu|=1\rbrace$. \EEE

By $Q_r(x)$ we denote the half open cube $Q_r(x) := x + r[-\frac{1}{2},\frac{1}{2})^d$  of sidelength $r>0$ centered at  $x \in \R^d$. We introduce a tessellation of $\R^d$ by 
\begin{align}\label{eq: tess}
\mathcal{Q}_r := \{Q_r(x) \colon x \in r\mathbb{Z}^d\}\,.
\end{align}
In the following, we often  omit the center  $(x)$ and simply write $Q_r \in \mathcal{Q}_r$ if no confusion arises. In a similar fashion,  by $Q_{\mu r}$ we indicate the cube with the same center, but sidelength $\mu r$ for $\mu >0$. \EEE   We will use the following elementary fact several times: for each  $Q_r \in \mathcal{Q}_r$  and each $k \in \N$  it holds that 
\begin{align}\label{ineq:Qintersectbound}
\#\{Q'_r \in \mathcal{Q}_r \colon Q_{kr} \cap Q_{kr}' \neq \emptyset \} \leq  (2k-1)^d\,, \EEE
\end{align} 
 where $\#$ indicates the cardinality of a set. Finally, by $B_\rho\subset \R^d$  we denote the  open  ball with radius $\rho$ centered in $0$.

\section{A geometric rigidity result in variable domains}\label{sec: rigidi}

In this section we present a geometric rigidity result generalizing \EEE  the celebrated result in \cite[Theorem~3.1]{friesecke2002theorem} to the setting of variable domains with $C^2$-boundary\AAA. \EEE Here, with \emph{variable domains} we intend sets of the form $\Omega \setminus \overline{E}$, where $\Omega \subset \R^d$, $d=2,3$, is a fixed bounded, open set \EEE and $E \in \mathcal{A}_{\rm reg}(\Omega)$ is arbitrary. The main feature of the result lies in the fact that the rigidity constant is \emph{independent} of the choice of $E$, provided that  a certain curvature regularization  for $\partial E$   is assumed. In Subsection \ref{sec: state-rig} we state our main  result and present  the proof in  Subsections \ref{sec: main proof}--\ref{sec: thickening2}.

\subsection{Statement of the rigidity result}\label{sec: state-rig}

Given $E \in \mathcal{A}_{\rm reg}(\Omega)$, we denote by $\bm{A}$ the second fundamental form  of $\partial E\AAA\cap\Omega\EEE$. In particular, for $d=3$, we have $|\bm{A}| = \sqrt{\kappa_1^2 + \kappa_2^2}$, where $\kappa_1$ and $\kappa_2$  are the principal curvatures of $\partial E\AAA\cap\Omega\EEE$. For $d=2$, we simply have $|\bm{A}| = \kappa$, where $\kappa$ denotes the curvature of the boundary, which is one-dimensional in this case. Given $q\in [d-1,+\infty)\EEE$ and $\gamma \in (0,1)$, we will assume a  curvature regularization for $\partial E$ of the form $\gamma\int_{\partial E \cap \Omega} |\bm{A}|^q \, {\rm d}\mathcal{H}^{d-1} $. Given \EEE also a norm  $\varphi$ on $\R^d$, we introduce the \emph{\AAA local \EEE surface energy}, consisting of a  perimeter  term  with respect to $\varphi$ and the curvature regularization, \AAA defined for every $K\in \mathfrak{M}(\Omega)$, \EEE by
\begin{align}\label{eq: surface energy}
\mathcal{F}_{\rm surf}^{\varphi,\gamma,q}(E\AAA;K)\EEE :=  \int_{\partial E \cap \AAA K\EEE} \varphi(\nu_E) \, {\rm d}\mathcal{H}^{d-1} + \gamma\int_{\partial E \cap \AAA K\EEE} |\bm{A}|^q \, {\rm d}\mathcal{H}^{d-1}\,,
\end{align}
where $\nu_E$ denotes the unit outer normal  to  $\partial E\AAA\cap\Omega\EEE$. \AAA When $K=\Omega$, we omit the dependence of the surface energy on the second argument. \EEE We now formulate the main result of this paper. 
\EEE


\begin{theorem}[Geometric rigidity in variable domains]\label{prop:rigidity}  
Let $d=2,3$, $q\in [d-1,+\infty)$, $\gamma \in (0,1)$, and  $\varphi$ be a norm on $\R^d$. Let $\Omega \subset \R^d$ be open and bounded and let $\tilde \Omega \subset \subset \Omega$ be an open subset. Then, there exist constants $C_0=C_0(\varphi)>0$, $\eta_0 = \eta_0(\QQQ\mathrm{dist}(\partial\Omega,\tilde \Omega)\EEE,\varphi)  \in (0,1)$ and for each $\eta\in (0,\eta_0]$ there exists  $C_\eta = C_\eta(\eta,\Omega,\tilde\Omega)>0$ such that the following holds:\\
\noindent For every $E \in \mathcal{A}_{\rm reg}(\Omega)$ there exists an open set $E_{\eta,\gamma}$ such that $E  \subset  E_{\eta,\gamma}  \subset \Omega$, $\partial E_{\eta,\gamma}\cap \Omega$ is a union of finitely many regular submanifolds, and  
\begin{align}\label{eq: partition}
\begin{split}
{\rm (i)} & \ \    \mathcal{L}^d(E_{\eta,\gamma}\setminus E) \le\eta\gamma^{1/q}  \mathcal{F}_{\rm surf}^{\varphi,\gamma,q}(E),  \quad \quad \quad {\rm dist}_\mathcal{H}(E, E_{\eta,\gamma}) \le \eta\gamma^{1/q},
\\
{\rm (ii)} & \ \          \int_{\partial E_{\eta,\gamma} \cap \Omega} \varphi(\nu_{E_{\eta,\gamma}}) \, {\rm d}\mathcal{H}^{d-1}  \leq (1+C_0\eta\EEE) \, \mathcal{F}_{\rm surf}^{\varphi,\gamma,q}(E)\,,
\end{split}
\end{align}
such that for the connected components $(\tilde \Omega^{\eta,\gamma}_j)_j$ of $\tilde \Omega  \setminus \overline{{E_{\eta,\gamma}}}$ and for every $y  \in  H^1(\Omega \setminus \overline{E};\R^d)$ there exist  corresponding rotations $(R^{\eta,\gamma}_j)_j \subset SO(d)$ \QQQ and vectors $(b^{\eta,\gamma}_{j})_j\subset \R^d$ \EEE  such that 
\begin{align}\label{eq: main rigitity}
\begin{split}
{\rm (i)} & \ \  \sum\nolimits_j \int_{
\OOO\tilde \Omega^{\eta,\gamma}_j\EEE}\big|{\rm sym}\big((R^{\eta,\gamma}_j)^T \nabla y-\mathrm{Id}\big)\big|^2
\leq C_0 \big(1 +  C_\eta \gamma^{-5d/q}\AAA\varepsilon\EEE    \big)  \EEE \int_{\Omega\setminus  \overline{E}} \mathrm{dist}^2(\nabla y,SO(d))
\,,
\\
{\rm (ii)} & \ \   \sum\nolimits_{j}\int_{
\OOO\tilde \Omega^{\eta,\gamma}_j} \big|(R^{\eta,\gamma}_j)^T \nabla y-\mathrm{Id}\big|^2
\leq C_\eta  \gamma^{-2d/q} \int_{\Omega\setminus  \overline{E} \EEE} \mathrm{dist}^2(\nabla y,SO(d))
\,,\\
\ZZZ {\rm (iii)} &\ \  \ZZZ \sum\nolimits_{j} \int_{\tilde \Omega^{\eta,\gamma}_j} \big|y-(R^{\eta,\gamma}_jx+b^{\eta,\gamma}_j)\big|^2
\leq C_\eta  \gamma^{\QQQ (2-4d)/q\EEE} \int_{\Omega\setminus  \overline{E} \EEE} \mathrm{dist}^2(\nabla y,SO(d))
\,,
\end{split}
\end{align}
where for brevity $\eps := \int_{\Omega \setminus \overline{E}} \dist^2(\nabla y(x), SO(d)) 
$. \EEE
\end{theorem}

We note that Theorem \ref{prop:rigidity}, in particular \eqref{eq: main rigitity},  provides a \emph{piecewise geometric rigidity} result in the spirit of \cite{Chambolle-Giacomini-Ponsiglione:2007,  Friedrich:15-4, Friedrich-Schmidt:15}. In fact, global rigidity may fail if the domain $\Omega$ \AAA(or more precisely $\tilde \Omega$) \EEE is disconnected by $E$ into several parts on each of which $y$ is close to a different rigid motion. A separation of the domain into the sets $\OOO(\tilde \Omega^{\eta,\gamma}_j)_j$ might still be necessary \emph{even if  $\Omega \setminus \overline{E}$ is connected}. In fact, this is indispensable if the domain is connected only through a \emph{thin tunnel}, as explained in Example~\ref{ex: tunnel}. Such phenomena are accounted for in our result  by defining the components $\OOO(\tilde \Omega^{\eta,\gamma}_j)_j$ with respect to an appropriate \emph{thickened set} $E_{\eta,\gamma}$ containing $E$.  Note that  \eqref{eq: partition}(i) ensures that we obtain a rigidity result \BBB outside of the \emph{small set} $E_{\eta,\gamma}\setminus E$, which vanishes for $\AAA\eta\EEE,\gamma \to 0$. In addition,  \eqref{eq: partition}(ii) provides a sharp control on the (anisotropic) perimeter of  $E_{\eta,\gamma}$ \BBB as $\eta,\gamma \to 0$, \EEE  which  will be essential for our applications to models involving surface energies, see Section \ref{sec: applications}. 
 
When comparing our result to  \cite{friesecke2002theorem},  the constant  in \eqref{eq: main rigitity} depends on the small parameter $\eta$ and the curvature regularization parameter $\gamma$, \BBB with $C_\eta \to +\infty$ as $\eta \to 0$. \EEE We emphasize, however, that for configurations with \BBB gradient close to the set of rotations, in the sense of \EEE
 \begin{align}\label{eq:CdeltaRate}
\int_{\Omega \setminus \overline{E}} \dist^2(\nabla y, SO(d)) \, {\rm d}x  \le  C_\eta^{-1} \gamma^{5d/q} \,, 
\end{align}
we obtain a uniform control on symmetrized gradients, see \eqref{eq: main rigitity}(i). (The subspace $\R^{d\times d}_{\rm sym}$ \EEE corresponds to the orthogonal space to $SO(d)$ at the identity matrix. Since different rotations appear in our statement, $\mathbb{R}^{d\times d}_{\mathrm{sym}}$ has to be replaced accordingly.) 
In our applications, this uniform control will be essential to obtain compactness for rescaled displacement fields, see \eqref{eq: rescali1} and Propositions \ref{prop: compi1} and \ref{prop: compi2} below. Eventually, \eqref{eq: main rigitity}(ii)  is needed to control higher order terms in the passage to linearized elastic energies, see  Lemma \ref{prop:liminfel}. Note that  \OOO even  \EEE  under the assumption \eqref{eq:CdeltaRate}, a uniform control on the gradients  \OOO independently of the set $E$  \EEE cannot be expected, as in Example \ref{ex: sharp} we show that the estimate is actually sharp. This is related to the fact that the constant in  Korn's inequality  (see e.g.\ \cite{Nitsche}) is not uniform for variable domains $\Omega \setminus \overline{E}$. In the proof, we will first establish  \eqref{eq: main rigitity}(ii) and then derive \eqref{eq: main rigitity}(i) from \eqref{eq: main rigitity}(ii). \EEE

\EEE

\BBB We also emphasize that the choice $q \ge d-1$ for the curvature regularization is essential for the proof, see Lemma \ref{lemma:slicing} and Example \ref{ex: q!!!}. \EEE 
We proceed with several slightly modified  versions of the statement which will be convenient for our applications.\EEE

\begin{corollary}[Version with Dirichlet conditions]\label{cor: rig-cor}
Suppose that $\Omega = U \cup U_D$ for two \BBB  bounded \EEE sets $U,U_D \subset \R^d$ with Lipschitz boundary. Then, for every $E \in \mathcal{A}_{\rm reg}(\Omega)$ and every $y  \in  H^1(\Omega \setminus \overline{E};\R^d)$ with $y = {\rm id}$ on $U_D$, the statement of Theorem \ref{prop:rigidity} holds with the \BBB additional  \EEE property that:
\begin{align*}
\OOO\text{if for some \EEE $\OOO j$  \BBB it holds that $ \mathcal{L}^d\big(\OOO\tilde \Omega^{\eta,\gamma}_j \EEE \cap U_D\big)>0$,\ then \OOO we can take \EEE \ $R^{\eta,\gamma}_j ={\rm Id}$\,,}    
\end{align*}
\BBB where  \EEE the constant $C_\eta$ additionally depends on $U_D$. 
\end{corollary}

In the applications, Dirichlet conditions will indeed be imposed on a set of positive $\mathcal{L}^d$-measure, as it is customary in free discontinuity problems. \EEE

\begin{corollary}[Version for graphs]\label{cor: graphi}
Consider $\Omega = \omega \times (-1,M+1)$ for some \BBB open and bounded  \EEE $\omega \subset \R^{d-1}$ and $M>0$. Suppose that $E = \lbrace (x',x_d) \in \Omega\colon \, x' \in \omega,  x_d > h(x')\rbrace$ for a \BBB regular \EEE function $h\colon \omega \to [0,M]$, i.e., $\partial E \cap \Omega$ is the graph of the function $h$.   Then, in Theorem~\ref{prop:rigidity} we find another set  \BBB $E'_{\eta,\gamma} \supset E_{\eta,\gamma} $, \EEE which is the supergraph of a smooth function $h_{\eta,\gamma}\colon \omega \to [0,M]$ \OOO with  \EEE $h_{\eta,\gamma} \OOO\leq  \EEE h$, i.e., we have $E'_{\eta,\gamma} = \lbrace (x',x_d) \in \Omega\colon \, x' \in \omega,  x_d > h_{\eta,\gamma}(x')\rbrace$ such that 
\begin{equation}\label{eq: graphiii}
{\rm(i)} \ \ \mathcal{L}^d(E'_{\eta,\gamma}\setminus E) \le \eta\gamma^{1/q} \mathcal{F}_{\rm surf}^{\varphi,\gamma,q}(E), \quad  \quad {\rm(ii)}  \
 \int_{\partial E'_{\eta,\gamma} \cap  \OOO\Omega\EEE} \varphi(\nu_{E'_{\eta,\gamma}}) \, {\rm d}\mathcal{H}^{d-1}  \leq C_0\mathcal{F}_{\rm surf}^{\varphi,\gamma,q}(E)\,.
\end{equation}
\end{corollary}
In particular, the thickened set can be chosen as a supergraph, at the expense of a coarser estimate in \eqref{eq: partition}(ii).  
Corollary \ref{cor: rig-cor} and Corollary \ref{cor: graphi} will be proved in Subsection \ref{sec: main proof} and Subsection~\ref{sec: thickening0}, respectively.    \EEE We proceed with some further comments on the result.

\begin{remark}[Version with mean curvature]\label{remark: mean}
{\normalfont

For $d=3$, $q=2$, and a \AAA regular \OOO domain  \EEE $\Omega \subset \R^3$, \OOO there are situations where in estimate \AAA \eqref{eq: partition} \EEE we can replace the second fundamental form $\bm {A}$ by the mean curvature  $\bm{H}\colon \partial E\AAA\cap\Omega\EEE \to \OOO \mathbb{R}$,  \EEE i.e.,  $\bm{H}:= \kappa_1+ \kappa_2$, \OOO where again $\kappa_1$ and $\kappa_2$ are the principal curvatures of $\partial E\AAA\cap\Omega\EEE$. In fact, denote \EEE by \OOO $\BBB \bm{G} \OOO:=\kappa_1\kappa_2$ the Gaussian curvature of $\partial E\AAA\cap\Omega$, by $\chi(\OOO\partial E\OOO\cap\Omega)$ the \emph{Euler characteristic} of $\partial E\cap \Omega$ \EEE and by $\kappa_g$ the \emph{geodesic curvature} of $\partial (\partial E \cap \Omega) \subset \partial \Omega$. (The outermost $\partial$ is meant here to denote the boundary of the 2-dimensional surface $\partial E\cap \Omega$ in the differential geometric sense and we assume \AAA for simplicity \EEE that $\partial(\partial E\cap \Omega)$ is $ C^2$.) Then, the Gauss-Bonnet theorem yields  
\begin{align*}
\int_{\partial E \cap \Omega} |\bm{A}|^2 \, {\rm d}\mathcal{H}^2 & \OOO =\EEE
\int_{\partial E\OOO\cap \Omega\EEE} |\bm{H}|^2 \, {\rm d}\mathcal{H}^2 - 2\int_{\partial E\OOO\cap \Omega} \bm{G}\, {\rm d}\mathcal{H}^2\\
&  = \int_{\partial E\OOO\cap \Omega\EEE} |\bm{H}|^2 \, {\rm d}\mathcal{H}^2 -4\pi\chi(\OOO\partial E\cap\Omega\EEE)  + 2\int_{  \partial (\partial E\cap \Omega) }  \kappa_g \, {\rm d} \mathcal{H}^{\OOO 1\EEE} \,.       
\end{align*} 
Exemplarily, we address two special cases:\\
(a) If $E \subset \subset \Omega$, i.e., $\partial E \cap \Omega=\partial E$ has no boundary, and if  one has 
\begin{align}\label{eq: charac}
-4\pi\gamma\chi(\partial E) \le C_0\eta\,,
\end{align}
then  one can replace    $\gamma \int_{\QQQ\partial E \EEE 
} |\bm{A}|^2 \, {\rm d}\mathcal{H}^2$ by $\gamma \int_{\QQQ\partial E \EEE
} |\bm{H}|^2 \, {\rm d}\mathcal{H}^2$ without essentially affecting estimate \eqref{eq: partition}(ii) (and similarly \eqref{eq: partition}(i)), which in this case would be 
\begin{equation}\label{ineq:surf_bound_mean_curvature}
\int_{\partial E_{\eta,\gamma} \cap \Omega} \varphi(\nu_{E_{\eta,\gamma}}) \, {\rm d}\mathcal{H}^{2}  \leq (1+C_0\eta) \, \left(\int_{\QQQ\partial E\EEE 
} \varphi(\nu_E) \, {\rm d}\mathcal{H}^{2} + \gamma\int_{\QQQ\partial E \EEE 
} |\bm{H}|^2 \, {\rm d}\mathcal{H}^{2}\,+C_0\eta\right)\, .
\end{equation}
For instance, in this case, \eqref{eq: charac} holds true if $\partial E\cap\Omega\AAA=\partial E$ consists of $m$ connected components which are all topologically equivalent to the sphere $\mathbb{S}^2$. In this case, $\chi(\QQQ\partial E \EEE 
) = 2m>0$.\\
(b) In a similar manner\EEE, if $\partial E\cap \Omega$ consists of a single connected component topologically equivalent to the \MMM flat disk \EEE and $\OOO 2\EEE\gamma\int_{  \partial (\partial E \cap \Omega) }  \kappa_g \, {\rm d} \mathcal{H}^{\OOO 1\EEE} \le C_0\eta $, we can again replace \BBB  \eqref{eq: partition}(ii) by \eqref{ineq:surf_bound_mean_curvature}. \EEE

}
\end{remark}

\begin{remark}[Set $\tilde \Omega$]
{\normalfont Due to our proof strategy based on cubic sets, see \eqref{eq: r-cubic set} below, the rigidity estimate is only local, given in terms of $\tilde{\Omega}$. Yet, one \EEE can replace $\OOO\tilde \Omega\EEE$ by $\Omega$, provided that $\Omega$ is regular, \EEE and that we replace \eqref{eq: partition}(ii)  by 
$${ \int_{\partial E_{\eta,\gamma}\BBB \cap \Omega\EEE} \varphi(\nu_{E_{\eta,\gamma}}) \, {\rm d}\mathcal{H}^{d-1}  \leq (1+C_0\eta\EEE) \Big(  \int_{\partial E \cap \Omega} \varphi(\nu_E) \, {\rm d}\mathcal{H}^{d-1} + \gamma\int_{\partial E \cap \Omega} |\bm{A}|^q \, {\rm d}\mathcal{H}^{d-1} \BBB + C_{\Omega,\varphi,\gamma,q}\Big)}
 $$
\BBB  for a suitable constant $C_{\Omega,\varphi,\gamma,q}>0$ independent of $E$. \EEE In fact, this follows by selecting $\Omega_*\supset \supset \Omega$ and applying \BBB Theorem \ref{prop:rigidity} \EEE for $\Omega_*$ in place of $\Omega$, \AAA the set $\Omega$ in place of $\tilde \Omega$, \EEE and for \EEE $E_* = E \cup (\Omega_* \setminus \overline{\Omega})$ in place of $E$. More specifically, the result yields a set $E_*\subset \EEE E_{\eta,\gamma}^* \subset \Omega_*$, and then we define $E_{\eta,\gamma} := E_{\eta,\gamma}^* \cap \Omega$. 
}
\end{remark}

\begin{example}[Thin tunnel]\label{ex: tunnel}
{\normalfont
We give an example for the necessity of thickening the set and refer to the schematic Figure~\ref{fig:thin tunnel} for an illustration.   For $\delta>0$  we suppose that, up to a negligible set,  $\Omega \setminus \overline{E}$ is given by the three sets $U_1 =\QQQ (-1,0)\times (0,1)$\EEE, $U^\delta_2 =  (0,1)\times (\tfrac{1}{2},\tfrac{1}{2}+\delta)$, and $U_3 = (1,2)\times (0,1)$. (Strictly speaking, smooth approximations of $U_1$, $U^\delta_2$, and $U_3$ need to be considered.) For $\sigma \in (0,\pi/2)$ we define 
\begin{align}\label{eq: exi1}
y_{\delta,\sigma}(x) = \begin{cases}  x + \tau^\sigma_1 &x \in U_1\,;\\
( (x_2 - \tfrac{1}{2}) + \tfrac{1}{\sigma}) (\sin(\sigma x_1), \cos(\sigma x_1)) &x\in U^\delta_2\,;\\
R_\sigma x + \tau^\sigma_3 &x  \in U_3\,,
\end{cases}
\end{align}
where $R_\sigma \in SO(2)$ denotes the rotation around the origin by the angle $\sigma$ and $ \tau^\sigma_1, \tau^\sigma_3$ are suitable translations such that $y_{\delta,\sigma}$ is continuous. Then $\nabla y_{\delta,\sigma} \in SO(2)$ on $U_1 \cup U_3$ and on  $\AAA U\EEE^\delta_2$ we have 
$$\nabla y_{\delta,\sigma} (x_1,x_2) = \begin{pmatrix}  (1 + \sigma (x_2 - \tfrac{1}{2}) ) \cos(\sigma x_1) & \sin(\sigma x_1) \\ - (1 + \sigma (x_2 - \tfrac{1}{2}) ) \sin(\sigma x_1) & \cos(\sigma x_1) \end{pmatrix}\,.$$ 
This yields $\dist^2(\nabla y_{\delta,\sigma}, SO(2)) = \AAA\big|\EEE\sqrt{\nabla y_{ \delta,\sigma}^T\nabla y_{ \delta,\sigma }} - {\rm Id}\AAA\big|\EEE^2 \EEE = \sigma^2(x_2-\frac{1}{2})^2$ on $U^{\delta}_{2}$, \EEE and therefore 
\begin{align}\label{eq: exi2}
\int_{\Omega \setminus \overline{E}} \dist^2(\nabla y_{\delta,\sigma}, SO(2))\, {\rm d}x = \EEE \sigma^2\delta^3/3\,.
\end{align} 
It is also easy to see that for all $R \in SO(2)$ \EEE one has 
\begin{align*}
\int_{\Omega\setminus\overline{E}} |\nabla y_{\delta,\sigma}-R|^2\,\mathrm{d}x \geq c \sigma^2
\end{align*}
\EEE for a universal constant $c>0$. Therefore, neither \eqref{eq: main rigitity}(i) nor \eqref{eq: main rigitity}(ii) can hold true on $\Omega \setminus \overline{E}$ with a constant independent of $E$.}
\end{example}
\EEE

\begin{figure}
\begin{tikzpicture}[scale=.7]

\tikzset{>={Latex[width=1mm,length=1mm]}};


\draw(-.4,3.4) node {$\Omega$}; 
\draw(0,0) rectangle (4,3); 
\draw(3.2,3.4) node {$E$};
 \draw[->] (4.5,1.5) to[out=20,in=160] (8+0.5,1.5);
  \draw(6.5,2.2) node{$y$};
  \draw[gray,fill=gray](3,1.65) arc(180:270:.1)--++(0,.1)--++(-.1,0);
   \draw[gray,fill=gray](3.1,1.55)--(3.4,1.55)--++(0,.1)--++(-.3,0)--++(0,-.1);
     \draw[gray,fill=gray](3.4,1.55) arc(270:360:.1)--++(-.1,0)--++(0,-.1);

   \draw[fill=gray] (3,3)--(3,1.65)(3,1.65) arc(180:270:.1)(3.1,1.55)--(3.4,1.55)(3.4,1.55) arc(270:360:.1)(3.5,1.65)--(3.5,3)--(3,3)--(3,1.65);
 
   \draw[fill=gray,gray](3,1.35) arc(180:90:.1)--++(0,-.1)--++(-.1,0);
      \draw[gray,fill=gray](3.1,1.45)--(3.4,1.45)--++(0,-.1)--++(-.3,0)--++(0,.1);
             \draw[fill=gray,gray](3.4,1.45) arc(90:0:.1)--++(-.1,0)--++(0,.1);
                   
                    \draw[fill=gray](3,0)--(3,1.35) (3,1.35) arc(180:90:.1)(3.1,1.45)--(3.4,1.45)(3.4,1.45) arc(90:0:.1)(3.5,1.35)--(3.5,0)--(3,0)--(3,1.35)(3,1.35) arc(180:90:.1)(3.1,1.45)--(3.4,1.45)(3.4,1.45) arc(90:0:.1)(3.5,1.35)--(3.5,0);
 
 \begin{scope}[shift={(9,0)}]
                    
    \draw(3,0)--(0,0)--(0,3)--(3,3);

  \draw[gray,fill=gray]  (3,1.35)arc(180:90:.1)(3.1,1.45)--++(0,-.1)--++(-.1,0);
  \draw[gray,fill=gray] (3.1,1.45)--++(0,-.1)--++(.3,0)--++(0,.1)--++(-.3,0);
  
  \draw[gray,fill=gray](3.4,1.45) arc(90:20:.1)($(3.4,1.35) +(20:.1)$)--++(200:.1)--(3.4,1.45);
  
       \draw[fill=gray,gray](3,0)--(3,1.35)(3,1.35) arc(180:90:.1)(3.1,1.45)--(3.4,1.45)(3.4,1.45) arc(90:20:.1)($(3.4,1.35) +(20:.1)$)--++(290:1.35)--(3,0)--(3,1.35);
    
    \draw(3,0)--(3,1.35)(3,1.35) arc(180:90:.1)(3.1,1.45)--(3.4,1.45)(3.4,1.45) arc(90:20:.1)($(3.4,1.35) +(20:.1)$)--++(290:1.35)--(3,0);
    
    \draw[gray,fill=gray] (3,1.65) arc(180:270:.1)--++(0,.1)--++(-.1,0);
    
    \draw[gray,fill=gray](3.1,1.55)--++(.3,0)--++(0,.2)--++(-.3,0)--++(0,-.2);
    
        \draw[gray,fill=gray](3.4,1.55) arc(270:380:.1)--++(180:.1)--++(270:.1);

\draw[fill=gray] (3,3)--(3,1.65)(3,1.65) arc(180:270:.1)(3.1,1.55)--(3.4,1.55)(3.4,1.55) arc(270:380:.1)($(3.4,1.65)+(20:.1)$)--++(110:1.35)--(3,3)--(3,1.65);

\draw($(3.4,1.35) +(20:.1)$)--++(290:1.35)--++(20:.5)--++(110:3)--++(200:.4);

\end{scope}

\begin{scope}[shift={(15,0)}]
\draw(0,0) rectangle (4,3); 

\draw[pattern=north west lines,pattern color=gray](2.8,0) rectangle (3.5,3); 

\draw(3.2,3.4) node {$E_{\eta,\gamma}$};
\end{scope}

\end{tikzpicture}
\caption{A thin tunnel that leads to failure of uniform rigidity on the unique connected component of $\Omega\setminus\overline{E}$, depicted schematically\EEE. \EEE On the left: The set $\Omega \setminus \overline{E}$, where $E$ is depicted in gray. In the middle: The set $y(\Omega\setminus \overline{E})$. On the right: The set $\Omega \setminus \overline{E_{\eta,\gamma}}$, where $E_{\eta,\gamma}$ is the hatched set.} 
\label{fig:thin tunnel}
\end{figure}
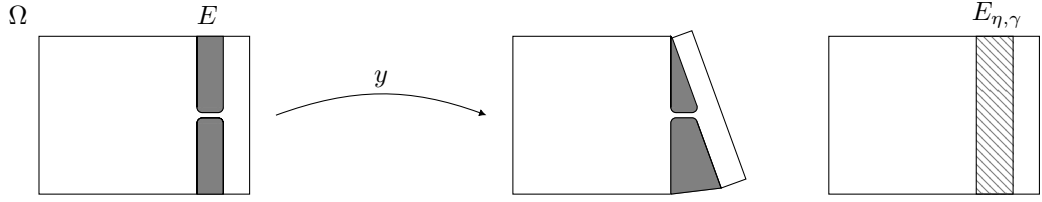

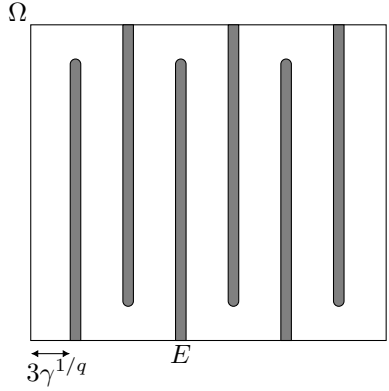
\begin{figure}
\begin{tikzpicture}[scale=0.7]

\tikzset{>={Latex[width=1mm,length=1mm]}};

\draw(0,0) rectangle (6.75,6);

\foreach \k in {0,...,2}{
\begin{scope}[shift={(2*\k+.5,0)}]

\draw[fill=gray](.25,0)--(.25,5.25)arc(180:0:.1)(.45,5.25)--(.45,0)--(.25,0);

\draw[fill=gray](1.25,6)--(1.25,.75)arc(180:360:.1)(1.45,.75)--(1.45,6)--(1.25,6);
\end{scope}

}

\draw(.5,-.6) node{$ 3 \EEE \gamma^{1/q}$};

\draw[<->](0,-.25)--++(.75,0);

\draw(-.25,6.25) node {$\Omega$};

\draw(2.85,-.25) node {$E$};

\end{tikzpicture}
\caption{A set $\Omega\setminus\overline{E}$ that shows that the constant in \eqref{eq: main rigitity}(ii) is sharp, for instance for $\Omega=(0,1)^2$. The set $E$ is depicted in gray.}
\label{fig:sharp}
\end{figure}

\begin{example}[Sharpness of constant]\label{ex: sharp} {\normalfont The constant in \eqref{eq: main rigitity}(ii) is sharp. To this end, consider $\Omega$ and $E$ in dimension $d=2$ as  depicted in Figure~\ref{fig:sharp}, and note that the thickening of the set $E$ will not disconnect $\Omega \setminus \overline{E}$, see \eqref{eq: partition}(i). The set $\Omega \setminus \overline{E}$ consists essentially of $m \sim \gamma^{-1/q}$ ``vertical stripes'' depicted with white color in the picture. We define a deformation $y$ on $\Omega \setminus \overline{E}$ which on each of the stripes bends by an angle of $\sigma:=\gamma^{1/q}$ as indicated in \eqref{eq: exi1} (for $\delta := 3\gamma^{1/q}$) such that between the first and the last stripe a macroscopic rotation is performed. Repeating the argument in \eqref{eq: exi2}, we get $\int_{\Omega \setminus \overline{E}} \dist^2(\nabla y, SO(2))\, {\rm d}x \lesssim m(\gamma^{1/q})^2(3\gamma^{1/q})^3/3 \sim \gamma^{4/q}$ and on the other hand $\int_{\Omega\setminus \overline{E}} |\nabla y-R|^2\,\mathrm{d}x \geq c $ for all $R \in SO(2)$. 
} 
\end{example}

\EEE

\subsection{Proof of Theorem \ref{prop:rigidity}}\label{sec: main proof}

This subsection is devoted to the proof of Theorem \ref{prop:rigidity}. We start with a short outline of the proof collecting the main intermediate steps.  The core of our proof is the construction of the thickened set $E_{\eta,\gamma}$ with the properties in \eqref{eq: partition}. We formulate this in a separate auxiliary result, and for this purpose we recall the definition of  $\mathcal{F}_{\rm surf}^{\varphi,\gamma,q}$ in \eqref{eq: surface energy}. \EEE


\begin{proposition}[\bf Thickening of sets]\label{prop:setmodification} 
\BBB Let $d=2,3$, $q\in [d-1,+\infty)$, $\gamma \in (0,1)$, and  $\varphi$ be a norm on $\R^d$. Let $\Omega \subset \R^d$ be open and bounded and let $\tilde \Omega \subset \subset \Omega$ be an open subset. Then, there exist a \ZZZ constant $C_0=C_0(\varphi)>0$,  $\eta_0 \in (0,1)$ depending only on ${\rm dist}(\partial \Omega,\tilde \Omega)$ and $\varphi$, \EEE 
and  \EEE for each $\eta\in (0,\eta_0]$ there exists $\QQQ c_\eta\EEE\in (0,1)$, \AAA with $c_\eta\to 0$ as $\eta\to 0$, \EEE such that the following holds:\\
Given $E\in \mathcal{A}_{\rm reg}(\Omega)$, we can find an open set $E_{\eta,\gamma}$ such that $E  \subset  E_{\eta,\gamma}  \subset \Omega$, $\partial E_{\eta,\gamma}\cap \Omega$ is a union of finitely many regular \AAA submanifolds\EEE, and   
\begin{align}\label{eq: partition-new}
\begin{split}
{\rm (i)} & \ \   \mathrm{dist}(\OOO x, E\EEE) \geq c_\eta\gamma^{1/q} \quad \text{for all $x  \in \big\{ y\in \Omega  \setminus \overline{E_{\eta,\gamma}}\colon \dist(y, \OOO\tilde \Omega\EEE) <c_\eta  \EEE \gamma^{1/q} \big\}   $}             \,,
\\
{\rm (ii)} & \ \    \mathcal{L}^d(E_{\eta,\gamma}\setminus E) \le \eta\EEE \gamma^{1/q}  \mathcal{F}_{\rm surf}^{\varphi,\gamma,q}(E),  \EEE \quad \quad \quad  \dist_\mathcal{H}(E, E_{\eta,\gamma}) \le \eta\EEE \gamma^{1/q}\,,  \EEE 
\\
{\rm (iii)} & \ \          \int_{\partial E_{\eta,\gamma} \cap \Omega } \varphi(\nu_{E_{\eta,\gamma}}) \, {\rm d}\mathcal{H}^{d-1}  \leq (1+C_0\eta\EEE) \,  \mathcal{F}_{\rm surf}^{\varphi,\gamma,q}(E)\,.
\end{split}
\end{align}
\end{proposition}

 \EEE We defer the proof to Subsection \ref{sec: thickening0}  below. Note that \eqref{eq: partition-new}(ii),(iii) are exactly the properties stated in the main result, see \eqref{eq: partition}. The additional property \eqref{eq: partition-new}(i) is essential for \BBB the proof of \eqref{eq: main rigitity}  \EEE  as it allows to cover 
\begin{align}\label{eq: strange notation}
\OOO\tilde\Omega\EEE_{\eta,\gamma}^E :=\OOO\tilde\Omega \EEE \setminus \overline{E_{\eta,\gamma}} 
\end{align}
\BBB with \EEE cubes which are all contained in $\Omega \setminus \overline{E}$. More precisely,  for $r>0$ and  $\OOO U \EEE \subset \R^d$ open and bounded, recalling the definition in \eqref{eq: tess},  we define the \emph{$r$-cubic set} corresponding to $\OOO U \EEE $ by  
\begin{align}\label{eq: r-cubic set}
(\OOO U\EEE)^r := {\rm int} \, \Big( \bigcup\nolimits_{Q_r \in \mathcal{Q}_r(\OOO U\EEE)} {Q}_r \Big) \,,
\end{align}
where $\mathcal{Q}_r(\OOO U\EEE) := \lbrace Q_r \in \mathcal{Q}_r \colon \, Q_r \cap \OOO U \EEE \neq \emptyset \rbrace$. 
We define 
\begin{align}\label{eq: r-def}
r_{\eta,\gamma} := \frac{c_\eta \, \gamma^{1/q}}{2\sqrt{d}}\,,
\end{align}
where $c_\eta$ is the constant of Proposition~\ref{prop:setmodification}. Now, by using \eqref{eq: partition-new}(i) and \MMM $c_\eta \to 0$ as $\eta \to 0$, \EEE  by possibly passing to a smaller constant \AAA$\eta_0$ depending on \ZZZ  ${\rm dist}(\partial \Omega,\tilde \Omega)$  \EEE  one can check  that  
\begin{align}\label{incl:Eeta}
Q_{r_{\eta,\gamma}}  \in   \mathcal{Q}_{r_{\eta,\gamma}}\big(\OOO\tilde\Omega\EEE_{\eta,\gamma}^E\big) \quad \quad \Rightarrow \quad \quad Q_{2r_{\eta,\gamma}} \subset \Omega \setminus \overline{E}\,.
\end{align}
For general $r$-cubic sets the following rigidity result holds.

\begin{proposition}[Rigidity on $r$-cubic sets] \label{lemma:chain}
\BBB Let \OOO $d \ge 2$, $U \subset \mathbb{R}^d$ \EEE be  \EEE open and bounded, let $r>0$, and suppose that the $r$-cubic set $(\OOO U\EEE)^r$ defined in \eqref{eq: r-cubic set} is connected.  \EEE Then,   there exists an absolute  constant $C>0$ independent of $U$ and $r$ such that for all  $y \in H^{1}((\OOO U\EEE)^r;\mathbb{R}^d)$ there exist $R \in SO(d)$ and $b \in \R^d$ such that 
\begin{align}\label{eq: RRR}
\begin{split}	
\QQQ {\rm(i)} \quad & \QQQ \int_{(U)^r} |\nabla y -R|^2\,\mathrm{d}x \leq C\big( \# \mathcal{Q}_r(U)\big)^2 \int_{(U)^r} \mathrm{dist}^2(\nabla y,SO(d))\,\mathrm{d}x\,.\\
\QQQ {\rm(ii)} \quad & \QQQ r^{-2} \int_{(U)^r} |y(x) -(R\,x + b)|^2\,\mathrm{d}x \leq C\big( \# \mathcal{Q}_r(U)\big)^4  \int_{(U)^r}\mathrm{dist}^2(\nabla y,SO(d))\,\mathrm{d}x\,.
\end{split}
\end{align}
Additionally, if there exists $Q \in \mathcal{Q}_r(\OOO U\EEE)$  \OOO with \EEE  $\mathcal{L}^d(Q \cap\lbrace \nabla y = {\rm Id}\rbrace) \ge cr^d$ for  \BBB some absolute \EEE  constant $c\in (0,1)$, then \eqref{eq: RRR} holds for $R = {\rm Id}$, for a constant \AAA $C>0$ \EEE depending \AAA on \EEE $c$. \EEE 
\end{proposition}

The result is a direct consequence of the rigidity estimate \OOO\eqref{eq: rigidity} \EEE proved by {\sc
Friesecke, James, and M\"uller} \cite{friesecke2002theorem}, \EEE applied on a cube, along with estimating the  variation of the rotations on different cubes. Although the latter argument is well-known and has been performed, e.g., in   \cite[Section 4]{friesecke2002theorem}, we include a short proof in Appendix \ref{appendix_c} for convenience of the reader. 

Observe that typically one has $\# \mathcal{Q}_r(\OOO U\EEE) \sim \OOO \mathcal {L}^d(U)r^{-d}$, which along with \eqref{eq: r-def} explains the scaling in \eqref{eq: main rigitity}(ii). The proof of \eqref{eq: main rigitity}(i) instead will rely on Proposition  \ref{lemma:chain} along with the linearization formula \cite[(3.20)]{friesecke2002theorem}\AAA,\EEE
\begin{align}\label{eq: linear-form}
|{\rm sym}(R^T F- {\rm Id})| =   {\rm dist}(F, SO(d)) + {\rm O}(|F- R|^2)
\end{align}
for $F \in \R^{d\times d}$ and $\OOO R\in \EEE SO(d)$. In fact, the latter shows that it suffices to have a good bound on $\int |\nabla y - R|^4\, {\rm d}x$ \EEE in order to control the symmetrized gradient \AAA in $L^2$\EEE. We are now ready to give the proof of Theorem \ref{prop:rigidity}.

%

\begin{proof}[Proof of Theorem  \ref{prop:rigidity}] Let $q\in [d-1,+\infty)$, $\gamma\in (0,1)$, $\OOO\tilde\Omega \EEE \subset \subset \Omega$ be an open subset, and let $\varphi$ be a norm on $\R^d$. Without restriction we can assume that $\tilde\Omega$ is smooth. \EEE    We let $\eta_0$ as in Proposition~\ref{prop:setmodification} and  $\eta \in (0,\eta_0]$. We \ZZZ can assume that for $c_\eta \EEE$ given in Proposition \ref{prop:setmodification} also \eqref{incl:Eeta} holds,  where  $r_{\eta,\gamma}$  is defined as in \eqref{eq: r-def}. \EEE From now on, we write $r$ in place of $r_{\eta,\gamma}$  for  notational simplicity.

We  let  \EEE $E_{\eta,\gamma}$ be the set obtained from Proposition~\ref{prop:setmodification}. In particular, \eqref{eq: partition} holds by \eqref{eq: partition-new}(ii),(iii). Let  $\OOO \tilde\Omega\EEE_{\eta,\gamma}^E$ be the set in \eqref{eq: strange notation}, and denote  \EEE by $(\OOO \tilde\Omega\EEE^{\eta,\gamma}_j)_j$ the connected components of $\OOO\tilde\Omega\EEE_{\eta,\gamma}^E$. Note that \AAA these \EEE are finitely many due to the regularity of $E_{\eta,\gamma}$ and  $\tilde \Omega$.  \EEE  The main part of the proof now consists in deriving \eqref{eq: main rigitity}. To this end, \BBB similarly \EEE to the proof in \cite{friesecke2002theorem}, a crucial step \OOO is to reduce \EEE the problem to harmonic mappings, see Steps~1--2 below. In Steps~3--4 we then provide the \ZZZ rigidity estimate \eqref{eq: main rigitity}(i),(ii), and briefly indicate the Poincar\'e-type estimate \eqref{eq: main rigitity}(iii) in Step 5. \EEE   In the following, $C>0$ denotes a generic constant only depending \BBB $\Omega$, \EEE which may change from line to line.  \BBB Without restriction, we suppose that the sets 
\begin{align}\label{eq: tiliP}
{\OOO \hat{\Omega}\EEE}^{\eta,\gamma}_j := {\rm int} \Big(\bigcup\nolimits_{Q_r\in \mathcal{Q}_r(\OOO\tilde \Omega\EEE_j^{\eta,\gamma})} \overline{Q_{2r}} \Big) \quad \text{ are pairwise disjoint}.
\end{align}   
Indeed, whenever $\hat{\Omega}^{\eta,\gamma}_i \cap \hat{\Omega}^{\eta,\gamma}_j \neq \emptyset$, one can replace $\tilde\Omega^{\eta,\gamma}_i$ and $\tilde\Omega^{\eta,\gamma}_j$ in the reasoning below by $\tilde\Omega\EEE^{\eta,\gamma}_i \cup \tilde\Omega\EEE^{\eta,\gamma}_j$ and can derive \eqref{eq: main rigitity} for a single rotation on $\tilde\Omega\EEE^{\eta,\gamma}_i \cup \tilde\Omega\EEE^{\eta,\gamma}_j$.   \\[8pt]
\begin{step}{1}(Reduction to Lipschitz \OOO mappings  \EEE on cubes). For every cube $Q_r \in \mathcal{Q}_r(\OOO \tilde\Omega\EEE_{\eta,\gamma}^E)  \EEE $ we have $Q_{2r} \subset \Omega \setminus \overline{E}$ by \eqref{incl:Eeta}. By a variant of \cite[Theorem 6.15]{EvansGariepy92}, see  \OOO also \EEE \cite[Proposition A.1]{friesecke2002theorem},  we let   \EEE $y_Q \in W^{1,\infty}(Q_{2r};\mathbb{R}^d)$ be a Lipschitz truncation obtained from $y$ satisfying
\begin{align}\label{ineq:step2claim}
\begin{split}
{\rm (i)} &  \ \ \| \nabla y_Q\|_{L^\infty(Q_{2r})} \leq C\,,  \EEE 
\\
{\rm (ii)} &  \ \ \int_{Q_{2r}}|\nabla y-\nabla y_Q|^2\,\mathrm{d}x \leq C \int_{Q_{2r}\cap\{|\nabla y| > 2\sqrt{d}\}} |\nabla y|^2\,\mathrm{d}x\,.
\end{split}
\end{align}
Here, with a slight abuse of notation, we write $y_{Q}$ instead of $y_{Q_r}$. We now claim  \EEE  that it suffices to prove that there exist   $(R^{\eta,\gamma}_j)_j \subset SO(d)$  such that \EEE
\begin{align}\label{ineq:Lipschitzrigidtysymm}
\hspace{-0.2em}\sum\nolimits_{j} \sum\nolimits_{Q_r  \in \mathcal{Q}_r(\OOO \tilde\Omega\EEE^{\eta,\gamma}_j)} \int_{Q_r} \big| {\rm sym} \big( (R^{\eta,\gamma}_j)^T \nabla y_Q- {\rm Id} \big) \big|^2\,
\leq  C(1 + r^{-5d}\AAA \eps  \EEE)  \EEE  \int_{\Omega\setminus  \overline {E}\EEE} \mathrm{dist}^2(\nabla y,SO(d))\,
,
\end{align}
where here and in the following we use the shorthand notation $\eps := \int_{\Omega\setminus  \overline{E}\EEE} \mathrm{dist}^2(\nabla y,SO(d))\,\mathrm{d}x$,  \EEE and
\begin{align}\label{ineq:Lipschitzrigidityfull}
\sum\nolimits_{j} \sum\nolimits_{Q_r  \in \mathcal{Q}_r(\OOO \tilde\Omega\EEE^{\eta,\gamma}_j)} \int_{Q_r} \big|  (R^{\eta,\gamma}_j)^T \nabla y_Q- {\rm Id} \big|^2\,\mathrm{d}x  \leq C r^{-2d \EEE} \int_{\Omega\setminus  \overline E\EEE} \mathrm{dist}^2(\nabla y,SO(d))\,\mathrm{d}x\,.
\end{align}
Indeed, let us note that  
\begin{align}\label{ineq:largegradientbound}
 \int_{Q_{2r}\cap\{|\nabla y| > 2\sqrt{d}\}} |\nabla y|^2\,\mathrm{d}x\leq C\int_{Q_{2r}} \mathrm{dist}^2(\nabla y,SO(d))\,\mathrm{d}x\,\AAA,\EEE
\end{align}
\BBB since \EEE  $|F|\leq 2\mathrm{dist}(F,SO(d))$ for all $F\in \mathbb{R}^{d\times d}$ \EEE with $|F|>2\sqrt d$. 
\EEE This along with \eqref{ineq:Qintersectbound}, \eqref{incl:Eeta}, \BBB \eqref{eq: tiliP}, \EEE and \eqref{ineq:step2claim}(ii) \EEE shows that \begin{align}\label{eq:NNN}
\EEE \sum\nolimits_j \sum\nolimits_{Q_r  \in \mathcal{Q}_r(\OOO \tilde\Omega\EEE^{\eta,\gamma}_j)} \OOO\int_{Q_{2r}} |\nabla  y   - \nabla y_Q|^2\,\mathrm{d}x\, \EEE \le C\int_{\Omega\setminus  \overline{E}\EEE} \mathrm{dist}^2(\nabla y,SO(d))\,\mathrm{d}x\,. 
\end{align}
\EEE Then, \eqref{eq: main rigitity}(i),(ii) for a constant $C_\eta=C_\eta(\eta,\AAA\Omega,\tilde\Omega\EEE)>0$  \EEE and \BBB $C_0>0$ (depending on $\Omega$)  \EEE clearly follows from \eqref{ineq:Lipschitzrigidtysymm}--\eqref{ineq:Lipschitzrigidityfull}, \AAA \eqref{eq:NNN}, \EEE the triangle inequality,  the definition of $r=r_{\eta,\gamma}$ in \eqref{eq: r-def}, and the definition of $\mathcal{Q}_r(\OOO \tilde \Omega\EEE^{\eta,\gamma}_j)$ below \eqref{eq: r-cubic set}.  Therefore, it suffices to prove \eqref{ineq:Lipschitzrigidtysymm}--\eqref{ineq:Lipschitzrigidityfull}.  \EEE 
\end{step} \EEE
\\[13pt]
\begin{step}{2}(Reduction to harmonic \OOO mappings\EEE) 
For  \EEE every $Q_r \in \mathcal{Q}_r(\OOO\tilde \Omega\EEE_{\eta,\gamma}^E)  \EEE $, we consider $y_Q= w_Q+z_Q$, where, in the sense of distributions,  
\begin{align*}
\begin{cases} \Delta w_Q = 0 &\text{on } Q_{2r},\\
w_Q =y_Q &\text{on } \partial Q_{2r},
\end{cases}
\quad \text{ and } \quad  
\begin{cases} \Delta z_Q = \mathrm{div}\left(\nabla y_Q-\mathrm{cof}\nabla y_Q\right) &\text{on } Q_{2r},\\
z_Q =0 &\text{on } \partial Q_{2r}.
\end{cases}
\end{align*}
It holds that 
\begin{align}\label{ineq:zQbound}
\int_{Q_{2r}} |\nabla z_Q|^2\,\mathrm{d}x \leq  \int_{Q_{2r}}  |\mathrm{cof}\nabla y_Q - \nabla y_Q  |^2 \, {\rm d}x                      \le   \EEE    C \int_{Q_{2r}} \mathrm{dist}^2(\nabla y_Q,  \EEE SO(d))\, \mathrm{d}x\,.
\end{align}
In fact, this follows from the arguments in \cite[Proof of Theorem 3.1, Step 1]{friesecke2002theorem}, in particular using that $|\mathrm{cof} F - F| \le c\, {\rm dist}(F,SO(d))\EEE$ for all $F \in \R^{d \times d}$ with $|F|\le C$ for some $c=c(C)>0$, where here $C$ denotes the constant of \eqref{ineq:step2claim}(i). In view of \eqref{ineq:step2claim}(ii) and \eqref{ineq:largegradientbound}, \eqref{ineq:zQbound} implies 
\BBB
\begin{align}\label{ineq:zQbound2}
\int_{Q_{2r}} |\nabla y_Q -  \nabla w_Q|^2\,\mathrm{d}x    \leq C \int_{Q_{2r}} \mathrm{dist}^2(\nabla y_Q,  SO(d))\, \mathrm{d}x \le C \int_{Q_{2r}} \mathrm{dist}^2(\nabla y,  SO(d))\, \mathrm{d}x\,.
\end{align}
\EEE  This along with \eqref{ineq:Qintersectbound},  \eqref{incl:Eeta}, \BBB and \eqref{eq: tiliP} \EEE shows that, in order to establish \eqref{ineq:Lipschitzrigidtysymm}--\eqref{ineq:Lipschitzrigidityfull}, it suffices to show that there exist   $(R^{\eta,\gamma}_j)_j \subset SO(d)$   such that  \EEE 
\begin{align}\label{ineq:harmonicrigiditysymm}
\sum\nolimits_{j} \sum\nolimits_{Q_r  \in \mathcal{Q}_r(\OOO\tilde \Omega\EEE^{\eta,\gamma}_j)} \int_{Q_r} \big| {\rm sym} \big( (R^{\eta,\gamma}_j)^T \nabla w_Q- {\rm Id} \big) \big|^2
\leq C(1 + r^{-5d}\AAA\varepsilon\EEE)  \EEE  \int_{\Omega\setminus  \overline {E}\EEE} \mathrm{dist}^2(\nabla y,SO(d))
\end{align}
and
\begin{align}\label{ineq:harmonicrigidityfull}
\sum\nolimits_{j} \sum\nolimits_{Q_r  \in \mathcal{Q}_r(\OOO\tilde \Omega\EEE^{\eta,\gamma}_j)} \int_{Q_r} \big|  (R^{\eta,\gamma}_j)^T \nabla w_Q- {\rm Id} \big|^2\,\mathrm{d}x  \leq C r^{-2d \EEE} \int_{\Omega\setminus  \overline E\EEE} \mathrm{dist}^2(\nabla y,SO(d))\,\mathrm{d}x\,.
\end{align}
\end{step}
\\[3pt]
\begin{step}{3}(Local  ($L^2$-$L^\infty$)-estimate \OOO for harmonic mappings\EEE) In this step we show that for each $\OOO\tilde \Omega\EEE_j^{\eta,\gamma}$ there exists $R^{\eta,\gamma}_j \in SO(d)$ such that 
\begin{align}\label{ineq:globalbad}
\sum\nolimits_{j} \sum\nolimits_{Q_r  \in \mathcal{Q}_r(\OOO\tilde \Omega\EEE^{\eta,\gamma}_j)} \int_{Q_{2r}}|\nabla w_Q- R^{\eta,\gamma}_j|^2\,\mathrm{d}x \leq Cr^{-2d} \int_{\Omega\setminus\overline{E}\EEE}\mathrm{dist}^2(\nabla y,SO(d))\,\mathrm{d}x\,,
\end{align}
\BBB and \EEE   for each $Q_r \in \mathcal{Q}_r(\OOO\tilde \Omega\EEE^{\eta,\gamma}_j)$  \BBB it holds that \EEE   
\begin{align}\label{ineq:Linfty}
\|\nabla w_Q -R^{\eta,\gamma}_j\|_{L^\infty(Q_r)} \leq Cr^{-3d/2}\left(\int_{\Omega\setminus  \overline{E}\EEE} \mathrm{dist}^2(\nabla y,SO(d))\,\mathrm{d}x\right)^{1/2} \BBB = \EEE Cr^{-3d/2}\sqrt{\eps}\,,
\end{align}
where we recall the notation \OOO for  \EEE $\eps$  \OOO below  \EEE \eqref{ineq:Lipschitzrigidtysymm}.  \EEE 
\BBB To see this, \EEE  we apply  Proposition  \ref{lemma:chain} for $2r/3$ in place of $r$ on the function  $y$ and on  the \BBB sets $\hat{\Omega}^{\eta,\gamma}_j$ introduced in \eqref{eq: tiliP} in place of $U$. In view of the fact that $\hat{\Omega}^{\eta,\gamma}_j = (\hat{\Omega}^{\eta,\gamma}_j)^{2r/3}$,  we find   \EEE  $(R^{\eta,\gamma}_j)_j \subset SO(d)$ such that
\begin{align}\label{ineq:yglobalbad-new}
\int_{\OOO \hat{\Omega}\EEE^{\eta,\gamma}_j}|\nabla y- R^{\eta,\gamma}_j|^2\,\mathrm{d}x \leq Cr^{-2d} \int_{\OOO \hat{\Omega}\EEE^{\eta,\gamma}_j}\mathrm{dist}^2(\nabla y,SO(d))\,\mathrm{d}x\,,
\end{align}
where we used that $\# \mathcal{Q}_{2r/3}(\OOO \hat{\Omega}^{\eta,\gamma}_j) \le  \mathcal{L}^d(\Omega) (2r/3)^{-d} $, i.e., $C$ in \eqref{ineq:yglobalbad-new} also depends on $\Omega$. By  \eqref{ineq:Qintersectbound} this yields  \EEE 
\begin{align}\label{ineq:yglobalbad}
\sum\nolimits_{j} \sum\nolimits_{Q_r  \in \mathcal{Q}_r(\OOO \tilde \Omega\EEE^{\eta,\gamma}_j)}  \int_{Q_{2r}} \big|  (R^{\eta,\gamma}_j)^T \nabla y- {\rm Id} \big|^2\,\mathrm{d}x  \leq Cr^{-2d} \int_{\Omega\setminus  \overline{E}\EEE}\mathrm{dist}^2(\nabla y,SO(d))\,\mathrm{d}x\,,
\end{align}
where \BBB as before we employed also  \eqref{incl:Eeta} and \eqref{eq: tiliP}. \EEE In view of \eqref{ineq:Qintersectbound}, \eqref{incl:Eeta}, \BBB  \eqref{eq: tiliP},  \eqref{eq:NNN}, \EEE \eqref{ineq:zQbound2}, and the triangle inequality  we get 
\begin{align}\label{ineq:globalbad-NNN}
\sum\nolimits_{j} \sum\nolimits_{Q_r  \in \mathcal{Q}_r(\OOO \tilde \Omega\EEE^{\eta,\gamma}_j)} \int_{Q_{2r}}|\nabla w_Q- \nabla y|^2\,\mathrm{d}x \leq C\int_{\Omega\setminus\overline{E}\EEE}\mathrm{dist}^2(\nabla y,SO(d))\,\mathrm{d}x\,.
\end{align}\EEE
 Consequently, by  \eqref{ineq:yglobalbad} we finally obtain \eqref{ineq:globalbad}. 
 
 We now address \eqref{ineq:Linfty}.
%
%
For every $j$ and every $Q_r  \in \mathcal{Q}_r(\OOO \tilde \Omega\EEE^{\eta,\gamma}_j)$,  \EEE due to \eqref{ineq:globalbad}, the fact that $w_Q$ is a harmonic \OOO mapping \EEE on $Q_{2r}$ and $Q_{2r} \subset \Omega\setminus \overline{E}\EEE$, as a consequence of the mean value property and \OOO the Cauchy-Schwarz  inequality, \EEE we have
\begin{align}\label{Linfty_estimate}
\QQQ\| \nabla w_Q-R^{\eta,\gamma}_j\|_{L^\infty(Q_r)} \leq \frac{C}{r^{d/2}}\left( \int_{Q_{2r}}|\nabla w_Q- R^{\eta,\gamma}_j|^2\right)^{\frac{1}{2}} \leq \frac{C}{r^{3d/2}}\left(\int_{\Omega\setminus  \overline{E}} \mathrm{dist}^2(\nabla y,SO(d)) \right)^{\frac{1}{2}}\,.\EEE
\end{align}
This yields  \EEE  \eqref{ineq:Linfty}, and Step 3 of the proof is concluded.
\end{step}
\\[13pt]
\begin{step}{4}(Global estimates) 
In this step \AAA we \EEE finally prove \eqref{eq: main rigitity}(i),(ii). In view of  Step 2, it suffices to check \eqref{ineq:harmonicrigiditysymm}--\eqref{ineq:harmonicrigidityfull}.  First,  \eqref{ineq:harmonicrigidityfull} follows directly  from 
\eqref{ineq:globalbad}.  By the linearization formula \eqref{eq: linear-form}, \eqref{ineq:Linfty},  \eqref{ineq:globalbad-NNN},   and Young's inequality   \EEE we have
\begin{align}\label{last estimates}
\begin{split}
\sum\nolimits_{j} &\sum\nolimits_{Q_r  \in \mathcal{Q}_r(\OOO \tilde \Omega\EEE^{\eta,\gamma}_j)}    \int_{Q_r} \big| {\rm sym} \big( (R^{\eta,\gamma}_j)^T \nabla w_Q- {\rm Id} \big) \big|^2\,\mathrm{d}x \\
& \leq C \sum\nolimits_{j} \sum\nolimits_{Q_r  \in \mathcal{Q}_r(\OOO \tilde \Omega\EEE^{\eta,\gamma}_j)}\Big( \int_{Q_r}\mathrm{dist}^2(\nabla w_Q, SO(d))\,\mathrm{d}x + \int_{Q_r} |\nabla w_Q-R^{\eta,\gamma}_j|^4\,\mathrm{d}x \Big) \EEE \\
&\leq C\int_{\Omega\setminus  \overline{E}\EEE} \mathrm{dist}^2(\nabla y,SO(d))\,\mathrm{d}x +Cr^{-3d}\AAA\eps  \EEE  \sum\nolimits_{j} \sum\nolimits_{Q_r  \in \mathcal{Q}_r(\OOO \tilde \Omega\EEE^{\eta,\gamma}_j)} \int_{Q_r} |\nabla w_Q-R^{\eta,\gamma}_j|^2\,\mathrm{d}x\,.
\end{split}
\end{align}
Then, by using  \eqref{ineq:globalbad} we \BBB get \EEE  
\begin{align*}
\sum\nolimits_{j} \sum\nolimits_{Q_r  \in \mathcal{Q}_r(\OOO \tilde \Omega\EEE^{\eta,\gamma}_j)}    \int_{Q_r} \big| {\rm sym} \big( (R^{\eta,\gamma}_j)^T \nabla w_Q- {\rm Id} \big) \big|^2\,\mathrm{d}x \leq C(1 + r^{-5d}\AAA\eps\EEE) \int_{\Omega \setminus \overline{E}} \mathrm{dist}^2(\nabla y,SO(d))\,\mathrm{d}x\,.
\end{align*}
This  yields  \EEE \eqref{ineq:harmonicrigiditysymm} and concludes the proof of \eqref{eq: main rigitity}(i),(ii). 
\end{step}
\\[5pt] \QQQ 
\begin{step}{5}(Poincar\'e estimate) 
We briefly indicate how to derive \eqref{eq: main rigitity}(iii). \QQQ By applying  \eqref{eq: RRR}$\rm{(ii)}$ of Proposition \ref{lemma:chain} for $2r/3$ in place of $r$ on the function $y(x)-R_j^{\eta,\gamma}x$ and on   $\hat{\Omega}^{\eta,\gamma}_j$ in place of $U$, using again that $\# \mathcal{Q}_{2r/3}(\hat{\Omega}^{\eta,\gamma}_j) \le  \mathcal{L}^d(\Omega) (2r/3)^{-d} $, we also find     $(b^{\eta,\gamma}_j)_j \subset \R^d$ such that \EEE

\begin{align} 
\sum\nolimits_{j} \sum\nolimits_{Q_r  \in \mathcal{Q}_r(\tilde \Omega^{\eta,\gamma}_j)}  \int_{Q_{2r}} \big|y(x)- (R^{\eta,\gamma}_j \, x + b^{\eta,\gamma}_j) \big|^2\,\mathrm{d}x  \leq Cr^{2-\QQQ 4\EEE d} \int_{\Omega\setminus  \overline{E}\EEE}\mathrm{dist}^2(\nabla y,SO(d))\,\mathrm{d}x\,\QQQ .
\end{align}
Recalling the definition of $r=r_{\eta,\gamma}$ in \eqref{eq: r-def}, and the definition of $\mathcal{Q}_r(\OOO \tilde \Omega\EEE^{\eta,\gamma}_j)$ below \eqref{eq: r-cubic set} we conclude  \eqref{eq: main rigitity}(iii).
\end{step} \EEE
%
%
%
%
%
%
\end{proof}

\QQQ

\begin{remark}\label{localization_of_rigidity}
\normalfont A closer inspection of the proof shows that Theorem \ref{prop:rigidity} can be \textit{localized}, in the sense that in the rigidity estimates \eqref{eq: main rigitity} the sets $\Omega, \tilde \Omega$ can be replaced by  open sets $U,\tilde U$ respectively, with $\tilde{U} \subset \subset  U \subset \Omega$,  $\tilde{U} \subset \tilde{\Omega}$, and ${\rm dist}(\partial U,\tilde U) \ge {\rm dist}(\partial \Omega,\tilde \Omega)$, where the connected components of $\tilde U\setminus \overline{E}_{\eta,\gamma}$ would then be denoted by $(\tilde U^{\eta,\gamma}_j)_j$.  In that case, the constant $C_\eta$ does not depend on $\Omega$, $\tilde{\Omega}$, but only on $\eta$ and on $(\L^3(U))^2$. 

Indeed, the above arguments essentially rely on \eqref{incl:Eeta} and the estimates on cubic sets given in   Proposition \ref{lemma:chain}. The estimate ${\rm dist}(\partial U,\tilde U) \ge {\rm dist}(\partial \Omega,\tilde \Omega)$ guarantees  \eqref{incl:Eeta} for $U$, $\tilde{U}$. The fact that the constant $C>0$ appearing in \eqref{ineq:yglobalbad-new} depends quadratically on $\mathcal{L}^3(U)$  follows by the comment just below it, while by scaling invariance, $C>0$ appearing in \eqref{Linfty_estimate} can be chosen to be an absolute constant. Therefore, the estimates in \eqref{last estimates} yield this precise dependence. 	
\end{remark}

\EEE

We close this subsection with the short proof of Corollary \ref{cor: rig-cor}.

\begin{proof}[Proof of Corollary \ref{cor: rig-cor}]
A careful inspection of the previous proof shows that we only need to check that, whenever $\mathcal{L}^d(\OOO \tilde \Omega\EEE^{\eta,\gamma}_j \cap U_D)>0$ holds, then in \eqref{ineq:yglobalbad-new} we can choose $R^{\eta,\gamma}_j = {\rm Id}$. To this end, when   $\mathcal{L}^d(\OOO \tilde \Omega\EEE^{\eta,\gamma}_j \cap U_D)>0$,  we \BBB find $Q_r \in  \mathcal{Q}_{r}(\tilde{\Omega}_j^{\eta,\gamma})$  such that $Q_r \cap U_D\neq \emptyset$. Then, we can select  \EEE $Q'_{2r/3}\in \mathcal{Q}_{2r/3}(\hat{\Omega}_j^{\eta,\gamma})$, $Q_{2r/3}' \subset Q_{2r} \subset \hat{\Omega}^{\eta,\gamma}_j $, see \eqref {eq: tiliP},  such that \EEE by  \eqref{eq: r-def}, the fact that $\gamma\AAA\in (0,1)\EEE$, and by the fact that $U_D$ has Lipschitz boundary we get $\mathcal{L}^d(\AAA{Q'_{2r/3}}\EEE \cap U_D) \ge cr^d$ for a small absolute constant  $c\in (0,1)$, provided that     $c_\eta$ is sufficiently small also depending on $U_D$. \EEE Then the desired property follows from the additional statement in Proposition \ref{lemma:chain} \BBB and \AAA the fact that \EEE $y = {\rm id}$ on  $U_D$. \EEE In this context, note that the constant $C_\eta$ in \eqref{eq: main rigitity} depends on $c_\eta$ and therefore $C_\eta$ also depends on $U_D$.
\end{proof}

\subsection{Thickening of sets}\label{sec: thickening0}

\AAA In this subsection we prove Proposition \ref{prop:setmodification}. Without restriction we will assume from now on that $\varphi_{\rm min}:=\min_{\mathbb{S}^{d-1}} \varphi = 1$. Indeed, we can simply perform the proof for $\varphi_{\rm min}^{-1}\varphi$ in place of $\varphi$ and $\varphi_{\rm min}^{-1} \gamma$ in place of $\gamma$ to see that \eqref{eq: partition-new}(iii) holds. \EEE 
The proof essentially relies on a local construction to thicken the set $E$ in a \BBB suitable \EEE  way. To formulate the local statement, we introduce some further notation. Given $\rho >0$ and  a cube $Q_\rho \in \mathcal{Q}_\rho$, \MMM  see \eqref{eq: tess}, \EEE  we denote the set of \emph{neighboring cubes} by 
\begin{align}\label{eq_ neighb-c}
\mathcal{N}(Q_\rho):= \big\{ Q_\rho' \in \mathcal{Q}_\rho \colon  \mathcal{H}^{d-1}(\partial Q_\rho \cap \partial Q'_\rho)>0\big\}\,.
\end{align}
\BBB Note that $\#  \mathcal{N}(Q_\rho) =2d$. \EEE \MMM We also recall the definition of $\mathcal{F}_{\rm surf}^{\varphi,\gamma,q}$ in \eqref{eq: surface energy}. \EEE
Moreover, for notational convenience, we denote the anisotropic perimeter by 
\begin{align}\label{eq: ani-per}
\mathcal{H}^{d-1}_\varphi(\Gamma) := \int_\Gamma \varphi(\nu_\Gamma)\, {\rm d}\mathcal{H}^{d-1}
\end{align}
for  a norm $\varphi$ on $\R^d$ and for a $\QQQ(d-1)$-rectifiable set $\Gamma$, where $\nu_\Gamma$ denotes a measure-theoretical unit normal to $\Gamma$. Note that the integral is invariant under changing the orientation of $\nu_\Gamma$ as $\varphi$ is a norm.  The proof of Proposition \ref{prop:setmodification} will make use of the following lemma, whose proof will be given later in Subsection \ref{sec: thickening2}. 

 \EEE

\begin{lemma}[Local thickening of sets]\label{lemma: localthick}
\BBB Let $d=2,3$, $q\in[d-1,+\infty)\EEE$, $\gamma \in (0,1)$, and  $\varphi$ be a norm on $\R^d$. Let $\Omega \subset \R^d$ be open and bounded, $\tilde \Omega \subset \subset \Omega$ be an open subset, \EEE and   \EEE   let $\Lambda >0$.  Then, there exist constants  $C=C(\varphi,\Lambda)>0$ and   $\eta_0  = \eta_0(\Lambda)   \in (0,1)$ such that for all $\eta \in (0,\eta_0]$  the following holds:\\
For \OOO every \EEE $\OOO 0<\EEE\rho \le   \eta^7   \gamma^{1/q}$ and for each $Q_\rho \in \mathcal{Q}_\rho$ such that $\AAA \overline{Q_{12\rho}} \EEE \subset \Omega$, and 
\begin{align}\label{eq: necessary assu}
\mathcal{F}_{\rm surf}^{\varphi,\gamma,q}(E \AAA ;\EEE Q_{8\rho}) \le \Lambda\rho^{d-1},\ \ \mathcal{F}_{\rm surf}^{\varphi,\gamma,q}(E\AAA ;\EEE Q'_{8\rho}) \le \Lambda\rho^{d-1} \ \ \text{$\forall\ Q_\rho' \in \mathcal{N}(Q_\rho)$}\,,
\end{align}
we can find pairwise disjoint sets $(\Gamma_i)_{i=1}^I$ in  $ \partial E \cap Q_{3\rho}$  with $I \le \BBB C \EEE$, corresponding closed  sets $(T_i)_{i=1}^I \subset Q_{8\rho}$, \AAA with $\partial T_i$ being a union of finitely many regular submanifolds \EEE  and a decomposition $\lbrace 1, \ldots, I\rbrace = \mathcal{I}_{\rm good} \cup \bigcup_{Q'_\rho \in \mathcal{N}(Q_\rho)}  \mathcal{I}_{\rm bad}(Q'_\rho)$ such that
\begin{align}\label{eq: good}
\begin{split}
{\rm (i)} & \ \  \mathcal{H}^{d-1}_\varphi\big(\partial T_i \setminus \overline{E}\big) \le \mathcal{H}^{d-1}_\varphi (\Gamma_i \cap Q_\rho) + C\eta\rho^{d-1}\quad 
\OOO\forall\,  \EEE i \in \mathcal{I}_{\rm good}
,\\[1pt]
{\rm (ii)} & \ \  \mathcal{H}^{d-1}_\varphi\big(\partial T_i \setminus \overline{E}\big) \le \mathcal{H}^{d-1}_\varphi \big(\Gamma_i \cap (Q_\rho \cup Q_\rho')\big) + C\eta\rho^{d-1} \quad 
\OOO \forall \ Q_\rho' \in \mathcal{N}(Q_\rho),\, 
\forall \, i \in \mathcal{I}_{\rm bad}(Q'_\rho),
\end{split}
\end{align} 
and    
\begin{align}\label{eq: bad}
 \OOO  {\rm dist}\Big(\partial E \cap Q_\rho, \Big(\AAA E \cup\EEE\bigcup\nolimits_{i=1}^I T_i\Big)^c\Big)\geq \eta\rho\,.  \EEE 
\end{align}
Moreover,  fixing $Q'_\rho \in  \mathcal{N}(Q_\rho)$, introducing the notation  $\mathcal{I}'_{\rm bad}(Q_\rho)$ as above with respect to the cube $Q'_\rho$, and letting $\Gamma_i'$ and $T_i'$ be the corresponding sets, \EEE  we  \OOO have \EEE
\begin{align}\label{eq: ugly}
i \in \mathcal{I}_{\rm bad}(Q_\rho') \quad \Rightarrow  \quad \text{$\exists$ 
$j \in \mathcal{I}'_{\rm bad}(Q_\rho)$ such that $(\Gamma_i \triangle \Gamma_j') \cap (Q_\rho \cup Q_\rho') = \emptyset$ and  $T_i = T_j'$.}
\end{align}
\end{lemma}

Properties \eqref{eq: good}(i) and \eqref{eq: bad} are the fundamental points \EEE of the lemma:  essentially, in the proof we show that the connected components of  $\partial E\AAA\cap Q_{\rho}\EEE$ can be covered with thin polyhedra\EEE, leading to the definition of \OOO the sets  \EEE $(T_i)_i$. The case \eqref{eq: good}(ii) is only of technical nature\AAA, \EEE as additional care is needed if a component of $\partial E\cap Q_\rho$ is close to a neighboring cube, see Figur\AAA e\EEE ~ \ref{fig:planes}. 

The construction of  $E_{\eta,\gamma}$ in Proposition \ref{prop:setmodification} will rely on suitably modifying $E$ by applying Lemma \ref{lemma: localthick} on cubes intersecting $\partial E$. To this end, we  consider the tessellation of $\R^d$ with \BBB the \EEE  family of  cubes $\mathcal{Q}_\rho$ \MMM for \EEE  $\rho = \eta^7 \gamma^{1/q}$, \OOO so  \EEE that Lemma \ref{lemma: localthick} is applicable. In this context, it is important to control the number of \emph{boundary cubes}, \MMM given by \EEE
\begin{align}\label{def:partialQ}
\AAA\big\{ Q_\rho \in \mathcal{Q}_\rho \colon \, \partial E\cap Q_\rho \neq \emptyset,\,  \overline{Q_{12\rho}} \subset \Omega \big\}\,.\ \EEE
\end{align}
This will be achieved by the following lemma, whose proof will be given in the next subsection.

\begin{lemma}[Small area implies large curvature]\label{lemma:slicing}
Let $d=2,3$, $\Lambda\OOO>0$, $q\in [d-1,+\infty)$, and  $\gamma \in (0,1)$. Then, there exists an absolute  constant $\OOO c_0>0\EEE$ and a constant $c_{\Lambda} >0 \EEE$ only depending on   $\Lambda$ such that for all $0<\rho \le c_{\Lambda} \gamma^{1/q}$,  $E \in \mathcal{A}_{\rm reg}(\Omega)$, and $Q_\rho \in \mathcal{Q}_\rho$ such that $\AAA \overline{Q_{8\rho}}\EEE \subset \Omega$ and \EEE $\partial E\cap Q_{3\rho} \EEE\neq \emptyset$, the following implication holds true:
\begin{equation*}
\mathcal{H}^{d-1}(\partial E\cap Q_{8\rho})\OOO < c_0  \EEE \rho^{d-1}\implies   \gamma\int_{\partial E \cap Q_{8\rho}}  |\bm{A}|^q \, {\rm d}\mathcal{H}^{d-1} \BBB >  \EEE \Lambda\rho^{d-1}\,.
\end{equation*} 
\end{lemma}
Indeed, the implication shows that whenever the surface   $\partial E$ inside  a cube has small but   \OOO nonzero \EEE area, then necessarily the curvature contribution is high. This will allow us to control \AAA the number of boundary cubes\EEE, see particularly \eqref{eq: later reference2} and \eqref{eq: later reference} in the proof below. The result is a consequence of \cite[Corollary~1.3]{Simon1993Willmore} and we present its proof in Subsection \ref{sec: simon-lemma} below. Let us mention that the analog of Lemma \ref{lemma:slicing} is the main obstacle to generalize our result to higher \MMM dimensions, \EEE see Remark~\ref{obstacle_higher_dimensions_generalization} for more details in this direction.


\begin{example}\label{ex: q!!!}
{\normalfont
The statement of Lemma \ref{lemma:slicing} is false for $q<d-1$. In fact, let $E = B_{\AAA \sigma\EEE} \subset Q_{8\rho}$ be a ball of radius $\AAA \sigma\EEE$ for $\AAA \sigma\EEE>0$ small. Then, clearly $\mathcal{H}^{d-1}(\partial E\cap Q_{8\rho}) < c_0  \rho^{d-1}$ for $\AAA \sigma\EEE$ small enough. On the other hand, $\int_{\partial E \cap Q_{8\rho}}  |\bm{A}|^q \, {\rm d}\mathcal{H}^{d-1}$ coincides up to a constant with $\AAA \sigma\EEE^{d-1} \AAA \sigma\EEE^{-q}$.
}
\end{example}
\EEE

\EEE We are now in  \OOO the \EEE position to give the proof of Proposition~\ref{prop:setmodification}.  \EEE   


\begin{proof}[Proof of Proposition \ref{prop:setmodification}]
Recall that without restriction we have assumed that $\min_{\mathbb{S}^{d-1}} \varphi\AAA= \EEE 1$. In the following proof we will write $\varphi_{\max} = \max_{\mathbb{S}^{d-1}} \varphi$ for brevity. First of all, we define the constant $\OOO\Lambda := 2d \,12^{d-1}15^d\varphi_{\max}$, whose role will become clear in \eqref{ineq:casesa-cXXXnichtweg} below. For this $\Lambda$, we apply Lemma~\ref{lemma: localthick} to obtain $\eta_0$, and from now on we fix  $\OOO \eta\in (0,\eta_0]$\EEE.  We consider the tessellation of $\R^d$ with the collection of cubes $\mathcal{Q}_\rho$, where 
\begin{align}\label{eq: rhochice}
\rho:= \eta^7\gamma^{1/q}
\end{align} 
is chosen in such a way that Lemma \ref{lemma: localthick} is applicable. Here, without restriction, up to passing to a \OOO smaller \EEE constant $\eta_0$, we can assume that $\eta_0 \le c^{1/7}_{\Lambda}$, and therefore also Lemma \ref{lemma:slicing} is applicable.  Moreover, we can \OOO further \EEE choose $\eta_0\OOO>0$ also  depending  on $\Omega, \tilde\Omega$ such that for all  $\eta\in (0,\eta_0]$  we have  $20\sqrt{d}\rho \le \eta \dist(\tilde\Omega, \partial \Omega)$. Then, with a standard layering argument \EEE and recalling \eqref{eq: thick-def}, we can find an open set $\Omega'$ with $\tilde\Omega\subset\subset \Omega'\subset\subset \Omega$\AAA, \EEE and    
\begin{align}\label{eq: tubo}
(\partial \Omega')_{1\AAA 5\EEE\sqrt{d}\rho}   \subset \Omega \setminus \tilde\Omega,
\end{align}  
such that for a constant $C>0$ only depending on $\varphi$ it holds that
\begin{align}\label{eq: tubo2}
\mathcal{H}^{d-1}_{\varphi}\big( \partial E \cap  (\partial \Omega')_{3\sqrt{d}\rho}        \big)  \le  C\rho (\dist(\tilde\Omega, \partial \Omega))^{-1} \mathcal{H}^{d-1}_{\varphi}(\partial E \cap \Omega) \le C\eta \mathcal{H}^{d-1}_{\varphi}(\partial E \cap \Omega)\,.
\end{align} 
\ZZZ Note that the choice of $\eta_0$ ensures that $\eta_0$ depends only on ${\rm dist}(\partial \Omega,\tilde \Omega)$ and $\varphi$. \EEE \\
\EEE  \begin{step}{1}(Boundary cubes)
\MMM We define the collection of \emph{boundary cubes} by \EEE
\begin{equation}\label{def:partialQ_new}
\AAA\mathcal{Q}^\partial_\rho:=\big\{Q_\rho \in \mathcal{Q}_\rho \colon \, \overline{Q_{12\rho}} \subset \Omega \text{\ \ \MMM and: \AAA \ } \partial E\cap Q_\rho \neq \emptyset \text{\ \ or \  } \mathcal{F}_{\rm surf}^{\varphi,\gamma,q}(\AAA E ; Q_{8\rho}) > \Lambda\rho^{d-1}\,  
\big\}\,.
\end{equation} 
\MMM (For technical reasons, the definition slightly differs from \eqref{def:partialQ} mentioned above.) We decompose $\mathcal{Q}^\partial_\rho$ \EEE
as follows: first, we let $\mathcal{Q}^{\rm acc}_\rho$ be the collection of the cubes $Q_\rho \in \AAA\mathcal{Q}^\partial_\rho\EEE$  satisfying \EEE
\begin{align}\label{eq: accu}
\mathcal{F}_{\rm surf}^{\varphi,\gamma,q}(\AAA E ;\EEE Q_{8\rho}) > \Lambda\rho^{d-1}\,.
\end{align}
This definition collects the cubes \OOO whose 8-times enlargement \EEE  \emph{accumulates} a lot of surface energy. We further let   $\mathcal{Q}^{\rm neigh}_\rho \BBB \subset \mathcal{Q}^\partial_\rho \setminus  \mathcal{Q}^{\rm acc}_\rho$ be the collection of cubes $Q_\rho$  in a \emph{neighborhood} of $\mathcal{Q}^{\rm acc}_\rho $, i.e.,  
\begin{align}\label{eq: neigh}
\text{there exists} \quad    Q'_\rho \in \mathcal{Q}_\rho^{\rm acc} \quad \text{such that} \quad   Q_\rho \subset   Q'_{12\rho}\,.
\end{align}
Eventually, we set $\mathcal{Q}^{\rm flat}_\rho := \mathcal{Q}^\partial_\rho \setminus (\mathcal{Q}^{\rm acc}_\rho \cup \mathcal{Q}^{\rm neigh}_\rho)$. As we will see in the statement of Lemmata \EEE~\ref{lemma: curve graph}--\ref{Simons_lemma} below, the latter  collection corresponds to the cubes where the surface  $\partial E$ is approximately \emph{flat}.  For later purposes, we observe that by applying Lemma~\ref{lemma:slicing}  we find that \begin{align}\label{bounds_in_c}
\mathcal{H}^{d-1}(\partial E\cap Q_{8\rho} ) \ge \OOO c_0  \EEE \rho^{d-1} \ \text{\  and \  }   \   \mathcal{F}_{\rm surf}^{\varphi,\gamma,q} (\AAA E ;\EEE  Q_{8\rho})  \le \Lambda\rho^{d-1} \quad 
\OOO \forall \, \EEE Q_\rho \in \mathcal{Q}^{\rm neigh}_\rho \cup \mathcal{Q}^{\rm flat}_\rho
\,.
\end{align}
The set $E_{ \eta,\OOO\gamma}$ be  will defined by  \EEE 
\begin{align}\label{def:Eeta}
E_{\eta,\OOO\gamma\EEE}  :\EEE= \mathrm{int}\big(E \cup \bigcup\nolimits_{Q_\rho \in \mathcal{Q}^\partial_\rho} E_{\eta,\gamma\EEE}(Q_\rho)\big)\,,
\end{align}
 where the definition of the sets $E_{ \eta,\OOO\gamma}(Q_\rho)$ for  $Q_\rho \in \mathcal{Q}^\partial_\rho$ is given in Step 2 of the proof. In Step 3 we address \eqref{eq: partition-new}(i),(ii), and eventually Step 4 is devoted to the proof of \eqref{eq: partition-new}(iii). 
\end{step}\\ 
\begin{step}{2}(Definition of the sets $E_{\eta,\gamma}(Q_\rho)$). We address the three cases $\mathcal{Q}^{\rm acc}_\rho$, $\mathcal{Q}^{\rm neigh}_\rho$, and $\mathcal{Q}^{\rm flat}_\rho$ separately.\\
\OOO(a) \EEE First, if $Q_\rho \in \mathcal{Q}^{\rm acc}_\rho$, we set  \EEE 	 $E_{ \eta,\OOO\gamma\EEE}(Q_\rho) := \overline{Q_{12\rho}\EEE}$.\EEE\\
\OOO(b) \EEE If $Q_\rho \in \mathcal{Q}^{\rm neigh}_\rho$, we set $E_{ \eta,\OOO\gamma \EEE}(Q_\rho):= \emptyset$.\\
\OOO(c) \EEE  Finally, we address \OOO the case of \EEE $Q_\rho \in \mathcal{Q}^{\rm flat}_\rho$. If $Q_\rho \cap \Omega' = \emptyset$, we let $E_{ \eta,\OOO\gamma \EEE}(Q_\rho):= \emptyset$. Otherwise, by  \eqref{eq: tubo} \EEE we have $\AAA\overline{Q_{14\rho}}\EEE \subset \Omega$ and, \EEE in view of \AAA\eqref{def:partialQ_new} and \eqref{eq: neigh}, for every cube $Q'\in \mathcal{N}(Q_\rho)$ we have that $\mathcal{F}_{\mathrm {surf}}^{\varphi,\gamma,q}(E;Q'_{8\rho})\leq \Lambda\rho^{d-1}.$ \MMM This along with \eqref{bounds_in_c} allows to \EEE  apply Lemma~\ref{lemma: localthick} for $Q_\rho \in \mathcal{Q}_\rho^{\rm flat}$. \EEE We obtain finitely many corresponding pairwise disjoint sets $\AAA(\Gamma_i^Q )_{i=1}^I\EEE$ in $\partial E \cap Q_{3\rho}$ and closed  sets $\AAA(T_i^Q )_{i=1}^I\EEE$, \AAA with \EEE $T_i^Q  \subset Q_{8\rho}$ \AAA and $\partial \NNN T^Q_i$ being a union of finitely many regular submanifolds, \EEE such that \eqref{eq: good}--\eqref{eq: ugly} hold. In this case, we define
\begin{align}\label{eq: case-c}
E_{\eta,\gamma}(Q_\rho) =  \bigcup\nolimits_i T^Q_i\,. 
\end{align}
\MMM By definition it is clear that $E_{ \eta,\OOO\gamma} \subset \Omega$ and that $\partial E_{\eta,\gamma}\cap \Omega$ is a union of finitely many regular  submanifolds. \EEE We  now confirm   \eqref{eq: partition-new}.
\end{step}\\
\begin{step}{3}(Proof of \eqref{eq: partition-new}(i),(ii)) We start with the proof of  \eqref{eq: partition-new}(i). To this end, it suffices to check that 
\begin{align}\label{eq: to check}
\OOO \dist\big(\AAA y \EEE, \Omega\setminus\overline{E_{\eta,\gamma}}\big)\geq \eta\rho \quad \text{for all $y \in \partial E\cap \Omega'$}\,. \EEE
\end{align}
\OOO Indeed, let us assume for a moment that we have \eqref{eq: to check}, and let us  set $c_\eta:=\eta^8$. Consider \EEE an arbitrary $x\in \Omega\setminus \overline{E_{\eta,\gamma}}$ with $\mathrm{dist}(x,\tilde \Omega)<\eta\rho=c_\eta \EEE\gamma^{1/q}$, where the last equality follows from the choice of $\rho$ in \eqref{eq: rhochice}.  Since $E\subset E_{\eta,\gamma}$ we have that $\mathrm{dist}(x,E)=\mathrm{dist}(x,\partial E)$. In view of \eqref{eq: to check} it \BBB remains \OOO  to check that for every $y\in \partial E\setminus\Omega'$ we have that $|y-x|\geq \eta\rho$. \BBB This \OOO is trivial by the fact that $\mathrm{dist}(x,\tilde \Omega)<\eta\rho$ and \eqref{eq: tubo}.\EEE

\EEE To  \OOO verify  \EEE \eqref{eq: to check}, we first observe that each $y \EEE \in \partial E \cap \OOO\Omega'\EEE$ is contained in some cube of $\mathcal{Q}_\rho^\partial$, see  \eqref{eq: tubo} \MMM and  \eqref{def:partialQ_new}.  \EEE \OOO Therefore, \EEE let $y \EEE \in Q_\rho$ for some $Q_\rho \in \mathcal{Q}^\partial_\rho$ with $Q_\rho \cap \Omega' \neq \emptyset$. \EEE If $Q_\rho \in \mathcal{Q}_\rho^{\rm acc}$, then $\OOO \dist(y, \Omega\setminus E_{\eta,\gamma}(Q_\rho)) \EEE \geq 11\rho/2 \EEE$, and \eqref{eq: to check} follows in view of \eqref{def:Eeta}. If $Q_\rho \in \mathcal{Q}_\rho^{\rm neigh}$, by \eqref{eq: neigh} we find some $Q'_\rho \in \mathcal{Q}_\rho^{\rm acc}$ such that $Q_\rho \subset\overline{Q'_{12\rho}} = E_{\eta,\gamma}(Q_\rho')$. As $\dist(\partial Q_\rho, \partial Q_{12\rho}') \ge \rho/2$ \BBB by the definition of $\mathcal{Q}_\rho$, \EEE we get $\dist(y,\BBB \Omega \setminus E_{\eta,\gamma}(Q'_\rho)) \ge \rho/2$, and as before,  \OOO as long as $0<\eta\leq\eta_0\leq 1/2$, \EEE \eqref{eq: to check} follows from \eqref{def:Eeta}. Eventually, we suppose that $Q_\rho \in \mathcal{Q}_\rho^{\rm flat}$. Then by  \eqref{eq: bad} along with \eqref{eq: case-c} we get $\dist(y,\BBB \Omega \setminus E_{\eta,\gamma}(Q_\rho)) \EEE \ge \eta \rho $ and we conclude as before.\\[-10pt]  

We now show  \eqref{eq: partition-new}(ii). The estimate $\dist_\mathcal{H}(\OOO E, E_{\eta,\gamma}\EEE) \le \eta\gamma^{1/q}$ follows immediately from \eqref{def:Eeta} and the fact that each $E_{\eta,\gamma}(Q_\rho)$, $Q_\rho \in \mathcal{Q}_\rho^\partial$, is contained in $\overline{Q_{12\rho}}$, thus having \EEE diameter controlled by  $12\sqrt{d}\rho \le\eta\gamma^{1/q}$, for $\eta_{0}$ sufficiently small, see \OOO the \EEE definitions in Step~2 and \eqref{eq: rhochice}. In a similar fashion, as $E_{\eta,\gamma}(Q_\rho) \subset \overline{Q_{12\rho}}$  for all $Q_\rho \in \mathcal{Q}^\partial_\rho$, and $E_{\eta,\gamma}(Q_\rho) = \emptyset$ for $Q_\rho \in \mathcal{Q}_\rho^{\rm neigh}$, we obtain  
\begin{align*}
\mathcal{L}^d \big(E_{\eta,\gamma} \setminus E \big) \leq  \sum\nolimits_{Q _\rho\in \mathcal{Q}_\rho^{\rm acc} \cup \mathcal{Q}_\rho^{\rm flat} } \mathcal{L}^d(E_{\eta,\gamma}(Q_\rho))  \le (12\rho)^d \#\big(\mathcal{Q}_\rho^{\rm acc} \cup \mathcal{Q}_\rho^{\rm flat}  \big) \,.
\end{align*}
In view of \eqref{eq: accu} and \eqref{bounds_in_c} we have $\mathcal{F}_{\rm surf}^{\varphi,\gamma,q}(\AAA E ; Q_{8\rho}\EEE) \ge \min\lbrace \Lambda, c_{\OOO 0}  \EEE \rbrace \rho^{d-1}$ for all \EEE $Q_\rho \in \mathcal{Q}_\rho^{\rm acc} \cup \mathcal{Q}_\rho^{\rm flat}$, where we used the fact that we assumed $\varphi_{\rm min}=1$. \EEE  Now, by \eqref{ineq:Qintersectbound} we conclude
\begin{align}\label{eq: later reference2}
\mathcal{L}^d \big(E_{\eta,\gamma} \setminus E \big) \leq  C\rho \,  \mathcal{F}_{\rm surf}^{\varphi,\gamma,q}(E)
\end{align}   
for $C>0$ depending on $\varphi$. In view of \eqref{eq: rhochice}, \MMM for $\eta_0$ sufficiently small \EEE this concludes the proof of \eqref{eq: partition-new}(ii).
\end{step}\\[3pt]
\begin{step}{4}(Proof of  \eqref{eq: partition-new}(iii))  First, the construction of $E_{\eta,\gamma}$ in \eqref{def:Eeta} and  \eqref{eq: partition-new}(i)  imply that  \OOO
\begin{equation*}
\partial E_{\eta,\gamma}  \cap \Omega \subset \bigcup_{Q_\rho \in \mathcal{Q}^\partial_\rho} (\partial (E_{\eta,\gamma}(Q_\rho))\setminus \overline{E}) \cup \partial^{\rm rest} E,\ \  \text{where}\ \ \partial^{\rm rest} E := (\partial E  \cap \Omega)\setminus \bigcup_{Q_\rho \in \mathcal{Q}^\partial_\rho} E_{\eta,\gamma}(Q_\rho).
\end{equation*} \EEE 
Hence, as $E_{\eta,\gamma}(Q_\rho) = \emptyset$ for $Q_\rho \in \mathcal{Q}_\rho^{\rm neigh}$, recalling the notation in \eqref{eq: ani-per}, we find  \EEE  
\begin{align}\label{ineq:Eetaboundarybound}
\mathcal{H}^{d-1}_\varphi(\partial E_{\eta, \OOO\gamma } \cap \Omega) &\leq   \hspace{-0.38cm}
\sum_{Q_\rho  \in \mathcal{Q}^{\rm acc}_\rho } \mathcal{H}^{d-1}_\varphi\big(\partial (E_{\eta, \OOO\gamma\EEE}(Q_\rho))\setminus \overline{E}\big) + \hspace{-0.38cm}  \sum_{Q_\rho  \in \mathcal{Q}^{\rm flat}_\rho } \mathcal{H}^{d-1}_\varphi\big(\partial (E_{\eta, \OOO\gamma \EEE}(Q_\rho))\setminus \overline{E}\big) +   \mathcal{H}^{d-1}_\varphi(\partial^{\rm rest} E)\,.
\end{align}
We now estimate the terms on the  right  \EEE hand side of \eqref{ineq:Eetaboundarybound} separately. Let $\mathcal{Q}^{{\rm flat}}_{\rho,\Omega'}$ be the subset of cubes in $\mathcal{Q}^{\rm flat}_\rho$ intersecting $\Omega'$. \EEE First, by construction, in particular by \OOO the fact that \EEE $\partial E \cap Q_\rho \subset \OOO \partial E\cap E_{\eta,\gamma}(Q_\rho)$ for $Q_\rho \in \mathcal{Q}^{\rm flat}_{\MMM \rho, \Omega'}$ \OOO(recall \eqref{eq: bad} \BBB and the construction in Step 2)\EEE,  we have
\begin{align}\label{ineq:casesa-c}
\mathcal{H}^{d-1}_\varphi(\partial^{\rm rest} E) \le \mathcal{H}^{d-1}_\varphi\Big((\partial E \cap \Omega) \setminus \Big(   \bigcup\nolimits_{Q_\rho \in \mathcal{Q}_\rho^{\rm acc}} \OOO\overline{Q_{12\rho}} \EEE \cup  \bigcup\nolimits_{Q_\rho \in  \mathcal{Q}^{{\rm flat}}_{\rho,\Omega'}} Q_{\rho} \Big)      \Big)\,. 
\end{align}
We continue with the first term.  Since $\mathcal{H}^{d-1}_\varphi(\partial (E_{\eta,\gamma}(Q_\rho))) \le   \varphi_{\max} \,  2d (12\rho)^{d-1}$ \BBB for $Q_\rho \in \mathcal{Q}_\rho^{\rm acc}$, \EEE in view of \eqref{eq: accu}, we calculate  
\begin{align*}
\sum\nolimits_{Q_\rho \in \mathcal{Q}_\rho^{\rm acc}} \mathcal{H}^{d-1}_\varphi(\partial (E_{\eta,\gamma}(Q_\rho)))  & \leq \frac{\varphi_{\max} \, 2d (12\rho)^{d-1}}{\Lambda\rho^{d-1}}  \sum\nolimits_{Q_\rho \in \mathcal{Q}_\rho^{\rm acc}}\mathcal{F}_{\rm surf}^{\varphi,\gamma,q}(\AAA E ;\EEE Q_{8\rho})\\ \notag 
& \le   \frac{\varphi_{\max} \,2d 12^{d-1}15^d}{\Lambda}  \mathcal{F}_{\rm surf}^{\varphi,\gamma,q}\Big(\AAA E ;\EEE\bigcup\nolimits_{Q_\rho \in \mathcal{Q}_\rho^{\rm acc}} Q_{8\rho} \Big) \,,
\end{align*}
where in the second step we used that each point \OOO in   $\bigcup\nolimits_{Q_\rho \in \mathcal{Q}_\rho^{\rm acc}} Q_{8\rho}$ \EEE is contained in at most $15^d$ different cubes $Q_{8\rho}$, see \eqref{ineq:Qintersectbound}.  By the definition of $\Lambda = \varphi_{\max} \, 2d \, 12^{d-1}15^d$ at the beginning of the proof, this exactly gives
\begin{align}\label{ineq:casesa-cXXXnichtweg}
\sum\nolimits_{Q_\rho \in \mathcal{Q}_\rho^{\rm acc}} \mathcal{H}^{d-1}_\varphi(\partial \AAA(\EEE E_{ \eta,\OOO\gamma}(Q_\rho)\AAA)\EEE)  \leq     \mathcal{F}_{\rm surf}^{\varphi,\gamma,q}\Big(\AAA E ;\EEE \bigcup\nolimits_{Q_\rho \in \mathcal{Q}_\rho^{\rm acc}} Q_{8\rho} \Big) \,.  
\end{align}
\EEE Finally, for the second term on the right-hand side of \eqref{ineq:Eetaboundarybound}  we \OOO will  \EEE prove  that
\begin{align}\label{ineq:cubesd}
\sum\nolimits_{Q_\rho \in \mathcal{Q}_\rho^{\rm flat}}  \mathcal{H}^{d-1}_\varphi(\partial\AAA(\EEE E_{\eta,\OOO\gamma \EEE}(Q_\rho)\AAA)\EEE \setminus \overline{E}) \leq \mathcal{H}^{d-1}_\varphi\Big(  \partial E \cap   \bigcup\nolimits_{Q_\rho \in \mathcal{Q}_{\rho, \Omega'}^{{\rm flat}}\EEE}    Q_{3\rho} \Big) +  C_0\eta\EEE \mathcal{H}^{d-1}_\varphi(\partial E\cap \Omega)\,,
\end{align}
\ZZZ for $C_0>0$ depending only on  $\varphi$. \EEE To this end, we enumerate the cubes in $\mathcal{Q}_{\rho,\Omega'}^{{\rm flat}}$ \EEE by $\lbrace Q^1_\rho, \ldots, Q^N_\rho \rbrace $, and for each $Q^n_\rho$, $n=1,\ldots,N$, we denote by $( \Gamma^n_i)_{i=1}^{I_n}$ the pairwise disjoint sets in $\AAA\partial E\EEE\cap Q_{3\rho}^n$ and by  $(T^n_i)_{i=1}^{I_n}$ the \MMM  sets \EEE obtained by Lemma \ref{lemma: localthick}. Accordingly, we denote the set of indices by $\mathcal{I}^n_{\rm good}$ and $\mathcal{I}^n_{\rm bad}$. We let $\mathcal{J}_1 = \lbrace 1,\ldots, I_1\rbrace$, and given $\MMM \mathcal{J}_{1}, \ldots, \EEE \mathcal{J}_{n-1}$ for $n \in \lbrace 2,\ldots,N\rbrace$, we define $\mathcal{J}_n$ as the subset of $\lbrace 1,\ldots, I_n \rbrace$ which does not contain the indices $\mathcal{I}^n_{\rm bad}(Q_\rho')$ for $Q'_\rho \in \mathcal{N}(Q_\rho^n) \cap \lbrace Q^1_\rho,\ldots,Q^{n-1}_\rho \rbrace$, i.e., the indices related to parts of $\partial E$ which have been covered already by the procedure, related to one of \EEE the previous cubes $\lbrace Q^1_\rho,\ldots,Q^{n-1}_\rho \rbrace$. \BBB Thus, as a consequence of \eqref{eq: good} and \eqref{eq: ugly}, \EEE for each $n \in \lbrace 1,\ldots,N\rbrace$ and each $i \in \mathcal{J}_n$  we find sets $\Psi^n_i \subset \partial E\AAA\cap Q^n_{3\rho}\EEE$ such that $(\Psi^n_i)_{n,i}$ are pairwise disjoint and 
\begin{align}\label{eq: for each piece}
\mathcal{H}^{d-1}_\varphi(\partial T^n_i \setminus \overline{E})\le \mathcal{H}^{d-1}_\varphi(\Psi^n_i) + C\eta\rho^{d-1}\,.
\end{align} 
Indeed, if $i \in \mathcal{I}^n_{\rm good}$, one takes $\Psi^n_i = \Gamma^n_i \cap \MMM Q^n_\rho\EEE $. If $i \in \mathcal{I}^n_{\rm bad}(Q_\rho')$, we set $\Psi^n_i = \Gamma^n_i \cap (  \MMM Q^n_\rho\EEE  \cup Q'_\rho)$. We also note that $\#\mathcal{J}_n \le I_n \le \BBB C= C(\varphi,\Lambda)\EEE$, see Lemma \ref{lemma: localthick}. The construction along with \eqref{eq: for each piece} shows \AAA that\EEE 
\begin{align*}
\sum_{n=1}^N    \mathcal{H}^{d-1}_\varphi(\partial\AAA(\EEE E_{\eta,\gamma}(Q^n_\rho)\AAA)\EEE \setminus \overline{E}) & \le \sum_{n=1}^N  \sum_{i \in J_n} \mathcal{H}^{d-1}_\varphi(\partial T^n_i \setminus \overline{E}) \le  \sum_{n=1}^N  \sum_{i \in J_n} (\mathcal{H}^{d-1}_\varphi(\Psi^n_i) + C\eta\rho^{d-1})\\
&\le \mathcal{H}^{d-1}_\varphi\Big( \partial E \cap \bigcup\nolimits_{Q_\rho \in  \mathcal{Q}_{\rho,\Omega'}^{{\rm flat}}\EEE }  {Q_{3\rho}}   \Big) + \BBB C\EEE \eta\rho^{d-1} \# \mathcal{Q}_\rho^{\rm flat}\,,
\end{align*}
where in the last step we used the fact that $(\Psi^n_i)_{n,i}$ are pairwise disjoint and \QQQ their union is \EEE contained in $\partial E\cap\bigcup_{Q_\rho \in \mathcal{Q}_{\rho,\Omega'}^{{\rm flat}}\EEE}  {Q_{3\rho}}$.   This along with \eqref{bounds_in_c} shows \AAA that \EEE
\begin{align}\label{eq: later reference}
\hspace{-3.8pt}\sum_{Q_\rho \in \mathcal{Q}_\rho^{\rm flat}}   \mathcal{H}^{d-1}_\varphi(\partial \AAA(\EEE E_{\eta,\gamma}(Q_\rho)\AAA)\EEE \setminus \overline{E}) & \le \mathcal{H}^{d-1}_\varphi\Big( \partial E \cap \bigcup\nolimits_{Q_\rho \in \mathcal{Q}_{\rho,\Omega'}^{{\rm flat}}\EEE}  Q_{3\rho}   \Big) + C  \eta   \sum\nolimits_{Q_\rho \in \mathcal{Q}_\rho^{\rm flat}}  \mathcal{H}^{d-1}\big( \partial E \cap Q_{8\rho} \big)\notag \\
& \le \mathcal{H}^{d-1}_\varphi\Big( \partial E \cap \bigcup\nolimits_{Q_\rho \in \mathcal{Q}_{\rho,\Omega'}^{{\rm flat}}\EEE}  Q_{3\rho}   \Big)  + C \eta  \,  \mathcal{H}^{d-1}(\partial E \cap  \Omega)\,,
\end{align}
where in the last step we again used that each point \BBB in  $\bigcup\nolimits_{Q_\rho \in \mathcal{Q}_\rho^{\rm flat}} Q_{8\rho}$ \EEE is contained in at most $15^d$ different cubes $Q_{8\rho}$,  see \eqref{ineq:Qintersectbound}, \EEE and $Q_{8\rho} \subset \Omega$, see \eqref{eq: tubo}.  As $\Lambda$ itself  is  \BBB  a constant depending only on $\varphi$, \EEE we obtain \eqref{ineq:cubesd}.

We now conclude the proof as follows: note that \EEE $Q_\rho \in \mathcal{Q}_\rho^{\rm flat}$ implies $Q_{3\rho} \cap Q'_{8\rho} =\emptyset$ for all $Q'_\rho \in \mathcal{Q}^{\rm acc}_\rho$, see \eqref{eq: neigh}. Moreover,  we get that 
$${\Big(\partial E \cap \bigcup\nolimits_{Q_\rho \in \mathcal{Q}_{\rho,\Omega'}^{{\rm flat}}} Q_{3\rho} \Big) \setminus  \bigcup\nolimits_{Q_\rho \in \mathcal{Q}_{\rho,\Omega'}^{{\rm flat}}} Q_{\rho} \subset  \bigcup\nolimits_{Q_\rho \in \mathcal{Q}_\rho^{\rm acc}}\overline{Q_{12\rho}}\cup \AAA\big(\partial E\cap\big(\partial \Omega')_{3\sqrt{d}\rho}\big)\EEE\,.}$$ \EEE
Then, combining \eqref{ineq:Eetaboundarybound} and \eqref{ineq:casesa-c}--\eqref{ineq:cubesd}, and using \eqref{eq: tubo}--\eqref{eq: tubo2},  we obtain  \eqref{eq: partition-new}(iii), where $C_0>0$ indeed only depends on $\varphi$. 
\end{step}
\end{proof}

We close this subsection with the version for graphs.

\begin{proof}[Proof of Corollary \ref{cor: graphi}]
Consider $\Omega = \omega \times (-1,M+1)$ for some  \BBB open and  bounded \EEE $\omega \subset \R^{d-1}$ and $M>0$. Suppose that $E = \lbrace (x',x_{\OOO d\EEE}) \in \Omega\colon \, x' \in \omega,  x_d > h(x')\rbrace$ for a \BBB regular \EEE function $h\colon \omega \to [0,M]$.
We start by introducing the set 
$$E_{\eta,\gamma}^{*} = \mathrm{int}\big(E \cup\bigcup\nolimits_{Q_\rho \in  \mathcal{Q}^\partial_\rho}\overline{Q_{12\rho}}\big)\,.$$
Clearly, by construction $E_{\eta,\gamma}^{*} \supset  E_{\eta,\gamma}$.  Moreover, by Lemma \ref{lemma:slicing}  we find that
$${\AAA\mathcal{L}^d(E_{\eta,\gamma}^{*}\setminus E) \le C_0  \rho  \mathcal{F}_{\rm surf}^{\varphi,\gamma,q}(E),\EEE \quad \quad \quad  \mathcal{H}^{d-1}(\partial E_{\eta,\gamma}^{*} \cap \Omega) \le C_0 \mathcal{F}_{\rm surf}^{\varphi,\gamma,q}(E)}\,,  $$
for an absolute   constant $C_0>0$,  where we use the definition of $\rho$ in \eqref{eq: rhochice}. We note that the set $\Omega \setminus E_{\eta,\gamma}^{*}$ can be seen as the subgraph \BBB of \EEE a suitable $BV$-function with $\mathcal{H}^{d-1}(\MMM \overline{\partial^* E_{\eta,\gamma}^{*}} \EEE \setminus \partial^*E_{\eta,\gamma}^{*})=0$\EEE.  The desired set $E'_{\eta,\gamma} \supset E_{\eta,\gamma}^{*}$ is then obtained by approximating the set $\Omega \setminus E_{\eta,\gamma}^{*}$ from below with a suitable smooth graph \OOO so that \eqref{eq: graphiii} holds true, \EEE see \cite[Lemma 6.3]{Crismale}. 
\end{proof}

\subsection{Small area implies large curvature: Proof of Lemma \ref{lemma:slicing}}\label{sec: simon-lemma}

This subsection is devoted to the proof of Lemma \ref{lemma:slicing}. We start with a lemma due to {\sc
L. Simon}, whose original statement and proof can be found in \cite[Corollary 1.3]{Simon1993Willmore}. \AAA In the next statement, by $\partial \Sigma$ we intend the boundary of a regular surface $\Sigma$ in the differential-geometric sense.\EEE

\begin{lemma}\label{touchtwo_spheres_with_small_curvature_implies_large_area}	
Given $R>0$ and $\mu \in (0,1)$, there exist $\alpha_0=\alpha_0(\mu)>0$ and $\AAA c\EEE_0=\AAA c\EEE_0(\mu)>0$ such that the following holds: consider a connected, regular, two-dimensional surface $\Sigma$ in $\mathbb{R}^3$ \AAA with $\mathcal{H}^1(\partial \Sigma\cap \overline{B_R})=0$ \EEE such that  
\begin{align*}
\int_{\Sigma\cap B_R}|\bm{A}|\, {\rm d}\mathcal{H}^2<\alpha_0 R, 
\quad \Sigma\cap \partial B_{R}\neq \emptyset,\OOO\quad \text{and} \quad \EEE \Sigma\cap \partial B_{\BBB \mu \EEE R}\neq \emptyset\,.
\end{align*}
Then, we have 
\begin{equation*}
\mathcal{H}^2(\Sigma\cap B_R)\geq \AAA c\EEE_0 R^2.
\end{equation*}
\end{lemma}	

We proceed now with the proof of Lemma \ref{lemma:slicing}.\\[-8pt]

\begin{proof}[Proof of Lemma \ref{lemma:slicing}] We first treat the elementary case $d=2$, and then we address the case $d=3$ by using Lemma \ref{touchtwo_spheres_with_small_curvature_implies_large_area}.\\
\begin{step}{1}($d=2$)
Let $c_{\OOO 0\EEE}=1$ and $c_{\Lambda}= (\Lambda+1)^{-1/q}\in(0,1)$.  \EEE Consider  \EEE $Q_\rho \in \mathcal{Q}_\rho$ such that $\QQQ\overline{Q_{8\rho}}\EEE\subset \Omega$, \EEE $\partial E \cap Q_{3\rho} \EEE \neq \emptyset$, and  \EEE $\mathcal{H}^1(\partial E\cap Q_{8\rho})\OOO< \EEE \rho$. Let $\ggamma$ be a connected component of $\partial E\cap Q_{8\rho}$ intersecting $\AAA Q_{3\rho} \EEE$. Clearly, $\ggamma$ is a \BBB regular \EEE planar curve and we \BBB  \OOO have  \EEE 
\begin{equation*}
\mathrm{diam}(\ggamma)\leq \mathcal{H}^1(\ggamma)\leq \mathcal{H}^1(\partial E\cap Q_{8\rho})\OOO< \EEE \rho.
\end{equation*}
Therefore, $\ggamma$ is \BBB a \EEE closed curve inside the cube $Q_{8\rho}$. Hence, for all $0<\rho\leq c_{\Lambda}\gamma^{1/q}$, recalling that $c_{\Lambda} =  (\Lambda+1)^{-1/q}$ and $q \ge 1$, \EEE Lemma \ref{1_dim_curv_est}  yields
\begin{align*}
{\gamma\int_{\partial E\cap Q_{8\rho}}|\bm{A}|^q\, {\rm d}\mathcal{H}^1\geq c^{-q}_{\Lambda}\rho^q\int_{\ggamma}|\kappa_{\ggamma}|^q\, {\rm d}\mathcal{H}^1\geq (\Lambda+1) \, \rho^q  \, \big(\mathrm{diam}(\ggamma)\big)^{1-q} \ge  \EEE (\Lambda+1)\rho^q\rho^{1-q}>\Lambda\rho.}
\end{align*}  
\end{step}\\
\begin{step}{2}($d=3$) \MMM Let $\AAA c\EEE_0 = \AAA c\EEE_0(3\sqrt{3}/8)$ and $\alpha_0=\alpha_0(3\sqrt{3}/8)$ be the constants  in Lemma~\ref{touchtwo_spheres_with_small_curvature_implies_large_area}, applied for $R=4\rho$ and \EEE $\mu = 3\sqrt{3}/8$.  We define 
\begin{align}\label{eq: clambda-def} 
c_{\Lambda}:=\AAA\min\EEE\Big\{\AAA c\EEE_{\OOO 0\EEE}^{\frac{1-q}{q}} 4 \EEE \alpha_0 (\Lambda+1)^{-1/q}, (4\pi)^{1/2} \AAA c\EEE_{\OOO 0\EEE}^{1/q-1/2} (\Lambda+1)^{-1/q}\Big\}\,.
\end{align}
\EEE Consider $Q_\rho \in \mathcal{Q}_\rho$ such that $\QQQ\overline{Q_{8\rho}}\EEE\subset \Omega$, $\partial E \cap Q_{3\rho} \EEE \neq \emptyset$, and  \EEE  
\begin{align}\label{eq: to contradict}
\mathcal{H}^2(\partial E\cap Q_{8\rho})\OOO< \AAA c\EEE_{0}\rho^2 <  \AAA c\EEE_0  (4\rho)^2\EEE\,.
\end{align}
Let $K$ be a connected component of $\partial E\AAA\cap \Omega\EEE$ such that $K\cap Q_{3\rho} \EEE\neq \emptyset$. As $\partial E\AAA\cap \Omega\EEE$ is a regular surface, we note that $K$ is a regular surface as well\AAA, with $\partial K\subset \partial \Omega$. \EEE We first suppose that  \EEE $K\cap \partial Q_{8\rho}\neq\emptyset$. Then, the connectedness of the regular surface $K$ and the fact that \EEE  $Q_{3\rho} \subset B_{ 3\sqrt 3\rho/2}\subset B_{4\rho}\subset \AAA{Q_{8\rho}}\EEE\subset \Omega\EEE$ imply \EEE  that
\begin{equation*}
{K\cap \partial B_{3\sqrt 3\rho/2}\neq\emptyset, \ \ \ K\cap \partial B_{4\rho} \EEE \neq \emptyset\,.}
\end{equation*}
\AAA Moreover, since $\partial K\subset \partial \Omega$ and $\overline{B_{4\rho}}\subset \AAA \overline{Q_{8\rho}}\EEE\subset \Omega$, we have that $\mathcal{H}^1(\partial K\cap\overline{B_{4\rho}})=0$. \EEE Therefore, in view of \eqref{eq: to contradict}, by  applying Lemma~\ref{touchtwo_spheres_with_small_curvature_implies_large_area} \OOO(or more precisely its negation) \EEE for $R=4\rho>0$,  $\mu=3\sqrt 3/8\in (0,1)$, \EEE and $\Sigma=K$, we deduce that  
\begin{equation}\label{eq: AAA2}
\int_{K\cap B_{4\rho}}|\bm{A}|\, {\rm d}\mathcal{H}^2\geq 4\alpha_0\rho\,.
\end{equation}
Using  H\"older's inequality, \eqref{eq: to contradict}, \eqref{eq: AAA2}, and the fact that $q \ge 2>1$,  \EEE we obtain  for all $0<\rho\leq c_{\Lambda}\gamma^{1/q}$\AAA,\EEE
\begin{align*}
\gamma\int_{\partial E\cap Q_{8\rho}}|\bm{A}|^q\, {\rm d}\mathcal{H}^2&\geq c^{-q}_{\Lambda}\rho^q\int_{K\cap B_{4\rho}}|\bm{A}|^q\, {\rm d}\mathcal{H}^2\geq c^{-q}_{\Lambda}\rho^q\big(\mathcal{H}^2(K\cap B_{4\rho})\big)^{1-q}\Big(\int_{K\cap B_{4\rho}}|\bm{A}|\, {\rm d}\mathcal{H}^2\Big)^q\\
& \geq (c^{-q}_{\Lambda}\rho^q)(\AAA c\EEE_{\OOO 0\EEE}\rho^{2})^{1-q}(4\EEE\alpha_0\rho)^q = \OOO( \EEE4^q\EEE\alpha_0^q\AAA c\EEE^{1-q}_{\OOO 0\EEE}c^{-q}_{\Lambda}\OOO)\EEE\rho^2>\Lambda \rho^2\,,
\end{align*}
where the last step follows from the definition of $c_{\Lambda}$ in \eqref{eq: clambda-def}. \EEE

If instead we have $K\cap \partial Q_{8\rho}=\emptyset$, then $K$ is closed inside the cube $Q_{8\rho}$, i.e., $\partial(K\cap Q_{8\rho})=\emptyset$. By a classical topological-differential geometric fact regarding a lower bound on the Willmore energy of closed surfaces, we then have  that
\begin{equation*}
{\int_{K\cap Q_{8\rho}}|\bm{A}|^2\, {\rm d}\mathcal{H}^2 \EEE \geq 4\pi\,,}
\end{equation*}
see e.g.\ \cite[formula (0.2)]{Simon1993Willmore} and the references therein for its simple proof. By using again H\"older's inequality, the fact that $q\geq 2$, and  \eqref{eq: to contradict}, as before we estimate 
\begin{align*}
\gamma\int_{\partial E\cap Q_{8\rho}}|\bm{A}|^q\, {\rm d}\mathcal{H}^2&\geq c^{-q}_{\Lambda}\rho^q\int_{K\cap \OOO Q\EEE_{8\rho}}|\bm{A}|^q\, {\rm d}\mathcal{H}^2\geq c^{-q}_{\Lambda}\rho^q\big(\mathcal{H}^2(K\cap \OOO Q\EEE_{8\rho})\big)^{1-q/2}\Big(\int_{K\cap \OOO Q\EEE_{8\rho}}|\bm{A}|^2\, {\rm d}\mathcal{H}^2\Big)^{q/2}\notag\\
&\geq (c^{-q}_{\Lambda}\rho^q) (\AAA c\EEE_{\OOO 0\EEE}\rho^2)^{1-q/2}(4\pi)^{q/2}\geq (4\pi)^{q/2} \AAA c\EEE_{\OOO 0\EEE}^{1-q/2}c^{-q}_{\Lambda}\rho^2>\Lambda \rho^2\,,
\end{align*}
where the last step again follows from the definition of $c_{\Lambda}$ in \eqref{eq: clambda-def}. This concludes the proof. 
\end{step}
\end{proof}

\subsection{Local thickening of sets: Proof of Lemma \ref{lemma: localthick}}\label{sec: thickening2}
 
This subsection is devoted to the proof of Lemma \ref{lemma: localthick}. We start \OOO with \EEE a preliminary observation: given $\eta, \gamma >0$  \EEE and   $Q_\rho \in \mathcal{Q}_\rho$ for some $0<\rho \le \eta^7\gamma^{1/q}$ such that $\mathcal{F}_{\rm surf}^{\varphi,\gamma,q}(E \AAA;\EEE Q_{8\rho}) \le \Lambda\rho^{d-1}$, see \eqref{eq: necessary assu}, then by \eqref{eq: surface energy},   H\"older's inequality, and by  $\min_{\mathbb{S}^{d-1}}\varphi \OOO = \EEE 1$ \OOO (which was \EEE assumed without loss of generality), \EEE we obtain \EEE 
\begin{align}\label{eq: forlaterusage}
\begin{split}
\int_{\partial E\cap Q_{8\rho}}|\bm{A}| \, {\rm d}\mathcal{H}^{d-1}&\leq \big(\mathcal{H}^{d-1}(\partial E\cap Q_{8\rho})\big)^{\frac{q-1}{q}}\Big(\int_{\partial E\cap Q_{8\rho}}|\bm{A}|^q
\Big)^{\frac{1}{q}}\leq \left(\Lambda\rho^{d-1}\right)^{\frac{q-1}{q}}\left(\frac{\Lambda}{\gamma}\rho^{d-1}\right)^{\frac{1}{q}} 
\\ & 
\leq \Lambda\gamma^{-1/q}\OOO\rho^{d-1} \EEE \le \Lambda\eta^7\rho^{d-2}\,.
\end{split}
\end{align}
Therefore, we can ensure that the $L^1$-norm of $|\bm{A}|$ \AAA in \EEE  $\partial E \cap Q_{8\rho}$ is small compared to $\rho^{d-2}$ . \EEE

Our proof fundamentally relies on  the fact that, under the above bound on the curvature, $\partial E \cap Q_{\rho}$ is essentially a finite union of graphs of  \BBB regular \EEE functions \OOO with suitable a priori $C^1$-estimates. \EEE To state this result, we introduce the following notation: given an affine subspace $L \subset \R^d$ of codimension $1$ (i.e., a line in $\R^2$ or a plane in $\R^3$), we denote by $L^\perp$  the one-dimensional subspace  spanned by a unit normal vector $\nu_L$ to $L$. Accordingly, for $U \subset L$ and $u\colon U \to L^\perp$, we define ${\rm graph}(u):= \lbrace x + u(x) \colon x\in U \rbrace \subset \R^d$.  \EEE   We first state the result for $d=2$ \OOO separately\EEE, since its proof \EEE is significantly easier than for $d=3$\EEE. Note that the parameter $\varepsilon$ which appears in the next lemmata should not be confused with the one used below \eqref{eq: main rigitity}, since it serves a totally different purpose.

\begin{lemma}[Almost straight curves]\label{lemma: curve graph}
There exist  \EEE  $\eps_0>0$ and an absolute constant $C_1 \ge 1$ such that \BBB for every $\Lambda>0$  \EEE the following holds: for  \OOO every \EEE $\OOO \eps \in(0,\eps_0]$, every square $Q_\rho \MMM \subset \R^2\EEE$, $\rho>0$, and every $E \in \mathcal{A}_{\rm reg}(\R^2)$ satisfying  
\OOO
\begin{equation*}
\partial E \cap Q_{3\rho} \neq \emptyset,\ \mathcal{H}^1(\partial E\cap Q_{8\rho})\leq \Lambda\rho, \ \text{and}\ \int_{\partial E\cap Q_{8\rho}}|\bm{A}|\, {\rm d}\mathcal{H}^1\leq \varepsilon\,,
\end{equation*}\EEE
there exist \BBB regular \EEE curves $(\ggamma_i)_{i=1}^M$ with $M \le \Lambda$ such that
$$\OOO \partial E\cap Q_{3\rho} = \bigcup\nolimits_{i=1}^M \ggamma_i\cap Q_{3\rho}\,,$$ 
\hspace{-0.5em}corresponding lines $L_i$ and functions $u_i \colon \AAA \overline{U_i}\EEE \to L_i^\perp$, where $U_i \subset L_i$ are \AAA open \EEE   segments \OOO with $\mathrm{diam}(U_i) \BBB \le C_1 \rho$, \EEE  such that ${\rm graph}(u_i) = {\ggamma}_i$ for $i=1,\ldots,M$, \BBB and \EEE
$$\Vert u_i \Vert_{\OOO L^\infty(U_i)\EEE} \le C_1\eps\rho, \quad \quad \quad  \Vert u'_i \Vert_{\OOO L^\infty(U_i)\EEE} \le C_1\eps\,.$$
\end{lemma}

The proof is elementary, and we refer to Appendix \ref{appendix_a}. The analog\OOO ous \EEE statement in \OOO dimension \EEE $d=3$ is more involved: it is known as  the \textit{Approximate Graphical Decomposition Lemma} proved by {\sc 
L. Simon}, see \cite[Lemma 2.1]{Simon1993Willmore}. 

\begin{lemma}[Simon's \OOO Approximate Graphical Decomposition Lemma]\label{Simons_lemma}
For any $\Lambda>0$ and  $\mu \in (0,1)$  there exist  \EEE  $\OOO\varepsilon_S=\varepsilon_{S}\EEE(\Lambda, \mu )\in (0,1)\EEE$ and  a constant $C_1 = C_1(\BBB \Lambda, \EEE \mu) \BBB \ge 1\EEE$  \EEE such that for every $\varepsilon\in (0,\OOO\varepsilon_S\EEE]$, every $\rho>0$,  and every $E \in \mathcal{A}_{\rm reg}(\R^3)$  satisfying    $\OOO \partial  \EEE E \cap B_{\mu\rho} \neq \emptyset$,  and    
\begin{equation*}
\mathcal{H}^2(\partial E\cap B_\rho)\leq \Lambda\rho^2 \quad \quad  \mathrm{and} \quad \quad  \int_{\partial E\cap B_\rho}|\bm{A}|\, d\mathcal{H}^2\leq \varepsilon \rho\,,
\end{equation*} 
the following holds true: there exist pairwise disjoint, closed sets $P_1,\dots,P_N\subset \partial E$ with
\begin{equation*}
\sum\nolimits_{j=1}^N\mathrm{diam}(P_j)\leq C_1\sqrt{\varepsilon}\rho
\end{equation*}
and  functions  $u_i\in C^{2}(\overline{U_i};L_i^{\perp})$ \EEE for  $i=1,\dots,M$, with $M\leq \BBB C_1 \EEE $,  such that 
\begin{equation*}
\big(\partial E\cap B_{\mu\rho}\big)\setminus \bigcup\nolimits_{j=1}^N P_j=\Big(\bigcup\nolimits_{i=1}^M\mathrm{graph}(u_i)\Big)\cap B_{\mu \rho}\,.
\end{equation*}
Here, \OOO for every $i=1,\dots,M$, \EEE  $L_i$ is a two-dimensional plane in $\mathbb{R}^3$, $U_i\subset L_i$ is a smooth bounded  domain \OOO with \EEE 
\begin{equation}\label{ineq: size of U_i}
\mathrm{diam}(U_i) \BBB \le C_1 \EEE  \rho
\end{equation} \EEE 
of the form $U_i=U_i^0\setminus (\bigcup_{k=1}^{R_i}d_{i,k})$, where $U_i^0$ is a simply connected subdomain of $L_i$ and $(d_{i,k})_{k=1}^{R_i}$  are pairwise disjoint closed disks in $L_i$, which do not intersect $\partial U_i^0$.
Moreover, $\mathrm{graph}(u_i)$ is connected and $u_i$ also satisfies the estimates
\begin{equation*}
\sup\nolimits_{x \in U_i}|u_i(x)| \le  C_1\varepsilon^{1/6}\rho, \quad \quad \quad  \sup\nolimits_{x\in U_i}|\nabla u_i(x)|\leq C_1\varepsilon^{1/6}\,.
\end{equation*}
\end{lemma}

Here, $\partial U_i^0$ has to be understood with respect to the relative topology of $L_i$. Roughly speaking, the result states that, apart from sets $(P_j)_{\OOO j=1,\dots,N}$ of small diameter, so-called \emph{pimples},  $\partial E \cap B_{\mu\rho}$ can be written as the union of finitely many graphs of regular \EEE functions with small \OOO heights and small  gradients\EEE. Compare the result to the easier statement of Lemma \ref{lemma: curve graph}.

\begin{remark}[Adaptions to \BBB the \EEE original statement]\label{rem: adapt}
{\normalfont We have phrased the result slightly differently compared to the original statement in \cite[Lemma 2.1]{Simon1993Willmore}, where the lemma was stated only for $\mu=1/2$ but for general smooth, closed $2$-dimensional manifolds $\Sigma$. \EEE However, it is easy to verify through the proof that it is an interior \textit{$\varepsilon$-regularity} result, valid for every $\partial E \in C^2$ and for \EEE every $\mu\in (0,1)$, up to adapting the constants. \EEE The estimate \eqref{ineq: size of U_i} is implicitly mentioned in  the original statement, being a simple geometric observation, see also the proof of Lemma \ref{lemma: curve graph} in Appendix~\ref{appendix_a} for the analogous fact in two dimensions. Finally, the original statement has the assumption $0 \in \partial E$ which however can readily be generalized to requiring that $\partial E \cap B_{\mu\rho} \neq \emptyset$. \EEE
}
\end{remark} 
\begin{remark}[Result on cubes]\label{remark: simon}
{\normalfont

As in \cite{Simon1993Willmore}, the lemma is phrased  as  \OOO an interior \EEE statement for balls in $\mathbb{R}^3$. In the application below, we will apply \EEE it  on  cubes, by using $Q_{8\rho}$ in place of  $B_\rho$  and $Q_{3\rho}$ in place of $B_{\mu\rho}$. Indeed, to this end, it suffices to note that $B_{4\rho} \subset Q_{8\rho}$ and $Q_{3\rho} \subset B_{4\mu\rho}$ for $\mu \in (3\sqrt{3}/8,1)$. } 
\end{remark}

Both statements above involve functions which are defined with respect to suitable lines or planes, respectively. As a second preparation, we need to distinguish between \emph{good and bad lines and planes} for a cube $Q_\rho$. We discuss the following definitions and properties  only for planes in $\R^3$ as the analogous definitions for lines in $\R^2$ can be simply obtained by identifying lines in $\R^2$ with planes in $\R^3$ with one tangent vector given by $e_3$.

\EEE

Without restriction we suppose for the following arguments that $Q_\rho$ is centered in $0$, as this can always be achieved by a translation. 
\begin{definition}\label{theta_good_definition}
Let $\theta \in (0,1/\sqrt{3})$ and let $L$ be a plane with normal $\nu_L=(\nu_1,\nu_2,\nu_3) \in \mathbb{S}^2\EEE$ such that   $(L)_{3\eta\rho} \cap Q_\rho \neq \emptyset$, see \eqref{eq: thick-def}. We say that $L$ is a  \textit{$\theta$-good plane} for $Q_\rho$ if and only if  one of the following two properties holds true:
\begin{itemize}
\item[(1)] There exist $i,j \in \{1,2,3\}$, $i\neq j$, such that $|\nu_i|,|\nu_j|\geq \theta$;
\item[(2)] There exists $k \in \{1,2,3\}$ such that $|\nu_k| \geq \theta$ and
\begin{align}\label{eq: assu-bad plane}
\mathrm{dist}\big(L \cap Q_{3\rho}, \lbrace x_k = -\rho/2\rbrace \cup   \lbrace x_k = \rho/2\rbrace  \EEE \big) \geq 20\theta\rho\,.
\end{align}
\end{itemize}
\end{definition}
If $L$ is not a $\theta$-good plane \BBB for $Q_\rho$, \EEE we say that it is a  \textit{$\theta$-bad plane} \BBB for $Q_\rho$. \EEE In the statement of Lemma \ref{lemma: localthick}, the two different possibilities, namely  $\theta$-good or $\theta$-bad planes, are  reflected in the two cases described in \eqref{eq: good}(i) and \eqref{eq: good}(ii), respectively. The different cases of good and bad planes are depicted in Figure~\ref{fig:planes}\EEE. In the following, \OOO we denote by $\nu_L$  \BBB a unit \AAA vector normal \EEE to $L$ whose orientation will be specified in the proof below.    Recall also the shorthand \OOO  notation  \EEE for the anistropic perimeter in \eqref{eq: ani-per}. 
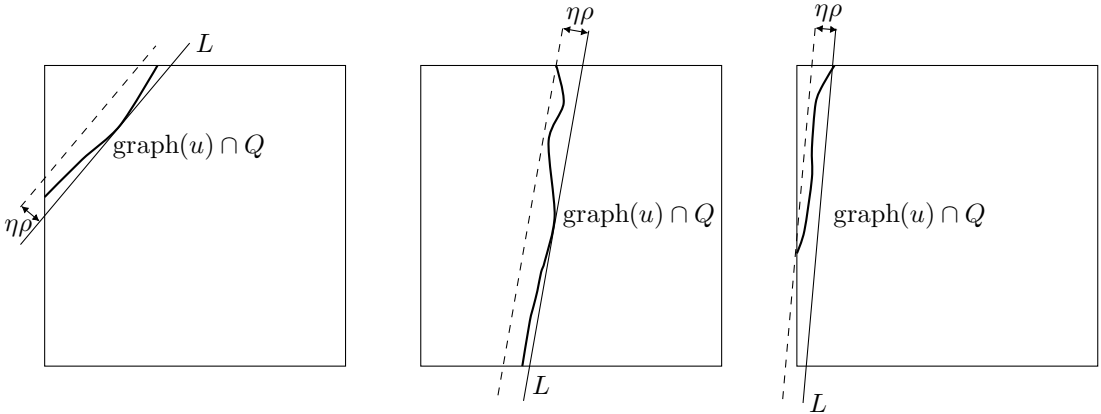
\begin{figure}[H]\label{figure 3}
\begin{tikzpicture}
\tikzset{>={Latex[width=1mm,length=1mm]}};
\begin{scope}[shift={(5,0)}]
\draw(1.6,-.25) node {$L$};
	
\draw(0,0) rectangle (4,4);
\draw(1.8,2)++(80:2.5)--++(260:5);
	
\draw[thick] plot [smooth, tension=.5] coordinates {(1.35,0)(1.45,.6)(1.5,.8)(1.55,1)(1.6,1.25) (1.65,1.4)(1.78,2)(1.7,3)(1.9,3.5)(1.8,4)};
\draw(2.9,2) node {$\mathrm{graph}(u) \cap Q$};
\draw[<->](1.8,4)++(80:.5)--++(350:.35);
\draw(1.8,4)++(80:.7)++(350:.2) node {$\eta \rho$};
	
\draw[dashed](1.8,4)++(80:.5)--++(260:5);
\end{scope}
	
\begin{scope}[shift={(0,0)}]
\draw(0,0) rectangle (4,4);
\draw(0,2)++(230:.5)--++(50:3.5);
\draw[dashed](0,2.5)++(230:.5)--++(50:2.8);
	
\draw[thick] plot [smooth, tension=.5] coordinates {(0,2.25)(.5,2.75)(1,3.2)(1.5,4)};
\draw[<->](0,2)++(230:.1)--++(140:.325);
	
\draw(0,2)++(230:.35)++(140:.16) node {$\eta \rho$};
	
\draw(1,2)++(45:1.3) node {$\mathrm{graph}(u) \cap Q$};
	
\draw(.2,2)++(230:.5)++(50:3.5) node {$L$};
	
\end{scope}
	
\begin{scope}[shift={(10,0)}]
\draw(0,0) rectangle (4,4);
	
\draw(.3,2)++(85:2.5)--++(265:5);
	
\draw(.5,2)++(85:-2.5) node {$L$};
	
\draw[dashed](.2,4)++(85:.5)--++(265:5);
\begin{scope}[shift={(.1,0)}]
\draw[thick] plot [smooth, tension=.5] coordinates {(-.1,1.5)(0,1.8)(.1,2.5)(.1,2.9)(.15,3.5)(.25,3.75)(.4,4)};
	
\end{scope}
	
\draw[<->](.2,4)++(85:.5)--++(175:-.3);
\draw(.1,4.05)++(80:.7)++(350:.2) node {$\eta \rho$};
\draw(1.5,2) node {$\mathrm{graph}(u) \cap Q$};
	
\end{scope}
	
\end{tikzpicture}
\caption{Different positions of planes inside a cube. In the left and in the middle cube, the two different cases of good planes are depicted whereas the \AAA figure on the right \EEE shows a bad plane. The thick surfaces illustrate $\mathrm{graph}(u)\cap Q_\rho$ and the dashed planes are at distance $\eta\rho$ to $L$, i.e, at the maximal distance of $\mathrm{graph}(u)$ from the plane $L$ inside $Q_\rho$. In the two pictures \AAA on the left\EEE,  the area of the dashed plane inside $Q_\rho$ and \AAA of \EEE the plane $L$ inside $Q_\rho$ are  comparable up to an error of order $\eta\rho^2$. This is the key observation for the proof of \eqref{eq: goodtheta1}--\eqref{ineq:area_of_domains_of_parametrization}. In the case of a bad plane, this is in general not true. 
} 
\label{fig:planes}
\end{figure}

\begin{lemma}[Surface estimate for $\theta$-good planes]\label{lemma:thetagood} 
There exists $\theta \in (0,1/\sqrt{3})$ small enough and a constant $C_\theta>0$ such that for any $\rho>0$ and any $\theta$-good plane $L$ \OOO for $Q_\rho$ \EEE the following holds:\\
\noindent  {\rm (i)} By letting $S_L :=  Q_{\OOO(1+  6\OOO\eta)\rho\EEE} \cap (L)_{3\eta\rho}$\AAA,\EEE we obtain
\begin{align}\label{eq: goodtheta1}
\mathcal{H}^2_\varphi(\partial^-S_L) \le    \mathcal{H}^2_\varphi(L\cap Q_\rho)   +C_\theta \OOO\varphi_{\mathrm{max}} \EEE \eta\rho^2\,,
\end{align}
where $\partial^- S_L :=\OOO \partial S_L \setminus (-3\eta\rho\nu_L+L)$. \\
\noindent  {\rm (ii)} Let $u\NNN\in L^{\infty} \EEE (\AAA U\EEE;L^\perp)\EEE$ for some \OOO bounded domain $\MMM L\cap Q_\rho\EEE \subset \AAA U\EEE \subset L$ with $\Vert  u \Vert_{L^\infty(\AAA U\EEE)} \le \MMM 2 \EEE \eta\rho$. \EEE Define \EEE $\AAA\omega_{u}^{\rho}\EEE := \Pi_L\big(\mathrm{graph}(u)\cap {Q_\rho}  \EEE \big)$, 
\EEE where $\Pi_L$ denotes the orthogonal projection onto the plane $L$. Then
\begin{equation}\label{ineq:area_of_domains_of_parametrization}
\mathcal{H}^2( \AAA\omega_{u}^{\rho}\EEE  \triangle (L \cap Q_\rho)   ) \le \MMM C_\theta  \EEE \eta\rho^2\,.
\end{equation}
\end{lemma}
\begin{lemma}[$\theta$-bad planes]\label{lemma: bad plane}
 There exists $\theta \in (0,1/\sqrt{3})$ small enough   such that for any $\rho>0$  and any $\theta$-bad plane $L$ \OOO for $Q_\rho$ \EEE the following holds: let $k\in \{1,2,3\}$ be the unique component such that   $|\nu_k|\geq \theta$. Then, we have 
\begin{align}\label{ineq:uniformcloseness}
\mathrm{either}\ -\frac{3\rho}{4} \BBB < \EEE x \cdot e_k <-\frac{\rho}{4} \ \ \forall \, x \in L\cap Q_{3\rho}, \ \ \ \mathrm{or}\ \ \ \frac{\rho}{4}<x \cdot e_k \BBB < \EEE \frac{3\rho}{4}  \ \  \forall \, x \in L\cap Q_{3\rho}. 
\end{align}
\end{lemma}

The proofs of the above lemmata   are elementary but tedious. They are deferred to  Appendix~\ref{appendix_b}. We are now in  \OOO the \EEE position to give the proof of Lemma \ref{lemma: localthick}.

\begin{proof}[\NNN Proof of Lemma \ref{lemma: localthick}]
Let $\gamma \in (0,1)$ \BBB and without restriction $\Lambda \ge 1$.    \EEE Consider $Q_\rho$ centered \OOO without restriction \EEE at $0$ such that $\AAA \overline{Q_{12\rho}} \EEE \subset \Omega$ and   \eqref{eq: necessary assu} holds.  In the case $d=2$, we choose $\eta_0 = \eta_0(\Lambda)$ such that $\Lambda\eta_0^7 \le \eps_0$, where $\eps_0\OOO>0\EEE$ is the constant of Lemma \ref{lemma: curve graph}. Then, by \BBB \eqref{eq: necessary assu} and \EEE \eqref{eq: forlaterusage} it is possible to apply Lemma \ref{lemma: curve graph}.  In the case $d=3$, we choose $\eta_0 = \eta_0(\Lambda)$ such that $\Lambda\eta_0 \le  \min\lbrace C_1\varepsilon_{\OOO S\EEE}^{1/6}, C_1^{-6},\OOO 2^{-(1+q/2)} \EEE c_{\OOO 0\EEE}\rbrace$, where $\eps_{\OOO S\EEE}$ and \BBB $C_1 \ge 1$ \EEE are the constants in Lemma \ref{Simons_lemma}, and $c_{\OOO 0\EEE}$ is the constant in Lemma~\ref{lemma:slicing}.  \EEE Consequently, in view of \eqref{eq: forlaterusage}, we have
\begin{align*}
\int_{\partial E\cap Q_{8\rho}}|\bm{A}| \, {\rm d}\mathcal{H}^2\leq  \Lambda\eta_0\eta^6\rho^{\OOO d\EEE-2} \le  \left(\eta/C_1\right)^6 \rho\,.
\end{align*} 
In particular, as $\eta \le \eta_0 \le C_1 \varepsilon_{\OOO S\EEE}^{1/6}$, we get $ \int_{\partial E\cap Q_{8\rho}}|\bm{A}| \, {\rm d}\mathcal{H}^2\leq   \eps_{\OOO S\EEE}\rho$. \BBB This along with  \eqref{eq: necessary assu} allows to \EEE apply Lemma~\ref{Simons_lemma} in the version of Remark \ref{remark: simon}. From now on, we only treat the case $d=3$ since the case $d=2$ is simpler (in the latter case, the \MMM sets \EEE  $\OOO( \EEE P_j\OOO)\EEE_{\OOO j}$ below in \eqref{ineq:diskestimate} can be chosen empty). In the following,  $C>0$ \BBB again \EEE  denotes a generic absolute constant, whose value is allowed to vary from line to line. \\
\begin{step}{1}(Application of Lemma \ref{Simons_lemma})
By Lemma \ref{Simons_lemma} in the version of Remark \ref{remark: simon}, \OOO applied for $\varepsilon:=(\eta/C_1)^6 \BBB \le \eps_S\EEE$, \EEE there exist   planes $L_i \subset \R^3$ and functions $u_i\colon \NNN \overline{U_i} \EEE \to L_i^\perp$ for $i=1,\ldots,M$, with  $M \le \BBB C_1 \EEE $, where $U_i= U_i^0 \setminus  {\bigcup_{k=1}^{R_i}} d_{i,k}$ for    (two-dimensional) disks \OOO$(d_{i,k})_{i,k}$ \EEE in the planes $L_i$, as well as pairwise disjoint closed subsets  $( P_{j})_{j=1}^{N} \BBB \subset \EEE \partial E$ such that
\begin{align}\label{eq:partialESimon}
\partial E \cap  Q_{3\rho}  = \left(\bigcup\nolimits_{i=1}^M\mathrm{graph}(u_i) \cup \bigcup\nolimits_{j=1}^{N} P_{j}\right)\cap Q_{3\rho}\,\AAA.\EEE
\end{align}
\AAA Moreover, \EEE 
\begin{align}\label{ineq:diskestimate}
 \sum\nolimits_{j=1}^N\mathrm{diam}(P_{j}) \leq  C_1(\eta/C_1)^3\rho \le   \eta\rho\,,
\end{align}
and for the functions $\OOO(u_i)_{i=1,\dots,M}$ we have \OOO the $C^1$-estimates  \EEE 
\begin{align}\label{eq:boundui}
\sup\nolimits_{x \in U_i}  |u_i(x)|  \leq  \OOO C_1 (\eta/C_1) \EEE \rho   =\OOO \eta\rho \NNN \le 2\eta\rho\EEE \,,  \quad \quad \quad  \sup\nolimits_{x \in U_i}  |\nabla u_i(x)|  \le  C_1 (\eta/C_1) = \eta\,.
\end{align}
Here, in the estimates \eqref{ineq:diskestimate}--\eqref{eq:boundui} we used that $\varepsilon=(\eta/C_1)^6$ \EEE and \BBB the fact that we can choose $\eta \leq \eta_0\le C_1$. \EEE \NNN The nonoptimal estimate with $2\eta\rho$ is introduced for later purposes in Step 4 below. \EEE \NNN To simplify the exposition, we \EEE assume for the moment that there are no pimples in $\partial E \cap Q_{3\rho}$, \NNN i.e., \EEE by  \eqref{eq:partialESimon}, that we have
\begin{align}\label{graphical approximation free of pimples}
\partial E \cap  Q_{3\rho}  = \bigcup\nolimits_{i=1}^M\mathrm{graph}(u_i) \cap Q_{3\rho}\,.
\end{align}
We defer the analysis of the case \MMM with \EEE pimples to Step 4. \EEE  We fix $\theta>0$ sufficiently small such that Lemmata \EEE \ref{lemma:thetagood}--\ref{lemma: bad plane} are applicable.  We distinguish the two cases
$$\text{{\rm (i)} \ \  $L_i$   is a $\theta$-good plane for $Q_{\rho}$,  \quad \quad \quad \quad  {\rm (ii)} \ \ $L_i$  is a $\theta$-bad plane  for $Q_{\rho}$\,.} $$
\AAA Let us  note that  $I:=M \le C_1  $ by the statement of Lemma~\ref{Simons_lemma}.\EEE
\end{step}

\begin{step}{2}(Good planes) 
\MMM First, \EEE let $L_i$ be a $\theta$-good plane for $Q_\rho$ and consider  $u_i \colon \NNN \overline{U_i}\EEE \subset L_i \to L_i^\perp$. \MMM In this case, we will define $\Gamma_i := {\rm graph}(u_i) \cap Q_\rho$ and thus it is not restrictive to assume that ${\rm graph}(u_i)$ intersects $Q_\rho$. \EEE   In the following, for notational convenience, we drop the subscript $i$ and simply write $L$ for the plane, $\nu_L$ for a unit normal to $L$,   $u$ for the function, and $\AAA U\EEE$ for its corresponding domain.  

\MMM 

We will first verify that $L\cap Q_{2\rho}\subset  U$. Indeed, by \eqref{eq:boundui} and \AAA the fact that \MMM ${\rm graph}(u) \cap Q_\rho \neq \emptyset$ we get that $U \cap Q_{2\rho} \neq \emptyset$ for $\eta$ sufficiently small. Moreover,  by \eqref{eq:boundui}  and \AAA by taking $\eta$ small\AAA er if necessary\MMM, we get  $| x+u(x)|_\infty <\frac{3}{2}\rho$ for all $x\in L \cap Q_{2\rho} \cap \partial U$.  Since $\partial\AAA(\partial E \cap Q_{3\rho})\MMM \subset \partial Q_{3\rho}$, \eqref{graphical approximation free of pimples} implies \AAA that \MMM $\partial U \cap Q_{2\rho} = \emptyset$. As $U \cap Q_{2\rho} \neq \emptyset$, we conclude $U  \supset  L  \cap Q_{2\rho}$, as desired. \EEE

%
%
%
%

  We choose the following orientation for $\nu_L$, which is important  for the definition of $\partial^- S_L$   in Lemma~\ref{lemma:thetagood}(i): \BBB we denote by $\bm{n}(x)$ the outer  unit normal  to  $\partial E\AAA\cap \Omega\EEE$  at $x$ and choose the orientation $\nu_L$ as well as an orthonormal basis $(\tau_1,\tau_2)$ \EEE of $L$ such \AAA that \EEE the normal vector $\tilde{\bm{n}}(x) = -(\partial_{\tau_1}u) \tau_1 -(\partial_{\tau_2}u) \tau_2 + \nu_L$  to ${\rm graph}(u)$ at the point $x+ u(x)$ satisfies $\bm{n}(x) = \tilde{\bm{n}}(x)/|\tilde{\bm{n}}(x)|$. Then, in view of \eqref{eq:boundui}, we have
\begin{align}\label{eq: orientL}
\Vert\AAA \bm{n}-\nu_L\EEE\Vert_{L^\infty(U)} \le C\eta\,.
\end{align}
As in Lemma \ref{lemma:thetagood}, we introduce the \emph{stripes} $S_L:= Q_{(1+  6\EEE\eta)\rho} \cap (L)_{3\eta\rho}$. We claim that
\begin{align}\label{ineq:distance}
\mathrm{dist}\big(\mathrm{graph} (u)\cap {Q_\rho}\EEE, S_L^c\big) \geq \NNN \eta\rho\, \EEE
\end{align}
and 
\begin{align}\label{ineq:graphdelE-}
\mathcal{H}^2_\varphi\big(\mathrm{graph}(u) \cap {Q_\rho}\EEE\big) \geq  \mathcal{H}^2_\varphi(L\cap Q_\rho)  - \BBB C_\theta \EEE \eta\rho^2  \ge  \mathcal{H}^2_\varphi(\partial^- S_L) - \BBB C_\theta \EEE \eta\rho^2\,,
\end{align}
where $\BBB C_\theta:=C(\theta,\varphi)>0$ is a constant depending only on $\theta$ and $\varphi$. Here, recall the notation in \eqref{eq: ani-per} and the definition of $\partial^- S_L$ in Lemma \ref{lemma:thetagood}(i). To obtain  \eqref{ineq:distance}, it suffices to check that \OOO
\begin{equation*}
{\rm (a)}\ \  \mathrm{dist}\Big(\mathrm{graph} (u)\cap \OOO{Q_\rho}, Q^c_{(1+  6\OOO\eta)\rho}\Big) \geq \NNN \eta\rho \quad \text{and}\quad {\rm (b)}\ \  \mathrm{dist}\Big(\mathrm{graph} (u)\cap \OOO{Q_\rho}, (L)^c_{3\eta\rho}\Big) \geq \NNN \eta\rho\,.\EEE
\end{equation*}
\EEE Item (a) is clear. To see (b),  we  first note  that $\dist(L, (L)^c_{3\eta\rho})   \OOO = \EEE 3\eta\rho$. Then, in view of \eqref{eq:boundui},  for each $y \in {\rm graph}(u) \cap \OOO{Q_\rho}\EEE$ we have $\dist(y,L) \le \MMM 2 \EEE \eta \rho$. Consequently, 
$${\mathrm{dist}\big(\mathrm{graph} (u)\cap \OOO{Q_\rho}\EEE, (L)^c_{3\eta\rho}\big) \geq \mathrm{dist}\big(L, (L)^c_{3\eta\rho}\big) -\NNN 2\eta\rho \EEE \ge   3\eta\rho - 2\eta\rho  \OOO =  \NNN \eta\rho \,.\EEE} $$  
Regarding \eqref{ineq:graphdelE-} we argue as follows: set \AAA as before \EEE $\AAA\omega_{u}^{\rho}\EEE = \Pi_L\AAA\big(\EEE\mathrm{graph}(u)\cap {Q_\rho}  \AAA \big)\EEE$, where $\Pi_L$ denotes the orthogonal projection onto the plane $L$. By Lemma \ref{lemma:thetagood}(ii) \MMM and the fact that $L \cap Q_\rho \subset {U}$, \EEE we have 
\begin{align}\label{eq:before}
\mathcal{H}^2\big(\AAA\omega_{u}^{\rho}\EEE  \triangle (L \cap Q_\rho)   \big) \le  C_\theta  \eta\rho^2 \,. \EEE
\end{align}   
Due to \eqref{eq: orientL} and  the fact that $\varphi$ is Lipschitz, being a \EEE norm, we get  $\Vert\varphi(\AAA\bm{n}\EEE) - \varphi(\AAA \nu_L\EEE)\Vert_{L^\infty(U)}\le C'\eta$, for a constant \BBB $C'$ \EEE depending \OOO additionally \EEE on $\varphi$.  \EEE Therefore, by \eqref{eq:before}, and the fact that we have assumed without restriction that $\min_{\mathbb{S}^2}\varphi \BBB = \EEE 1$, \EEE   we obtain 
\begin{align*}
\begin{split}
\mathcal{H}^2_\varphi\big(\mathrm{graph}(u) \cap {Q_\rho}\big) &= \int_{\AAA\omega_{u}^{\rho}\EEE} \varphi(\bm{n}(x)) \sqrt{1+|\nabla u(x)|^2}\, \mathrm{d}\mathcal{H}^2(x) \geq \mathcal{H}^2(\AAA\omega_{u}^{\rho}\EEE) (\varphi(\nu_L) -  C' \EEE \eta)\\ 
&\OOO\geq \big(\mathcal{H}^2(L\cap Q_{\rho})-\mathcal{H}^2(\AAA\omega_{u}^{\rho}\EEE\triangle(L\cap Q_{\rho}))\big)(\varphi(\nu_L)- C' \eta)\\
&\geq \mathcal{H}^2_\varphi(L\cap Q_\rho) - C_\theta\eta\rho^2\,, \EEE
\end{split}
\end{align*}
where in the last step we also used the obvious bound $ \mathcal{H}^2(L\cap Q_\rho) \le 3\rho^2$.  Then, by Lemma \ref{lemma:thetagood}(i) 
$$ 
\mathcal{H}^2_{\varphi }\big(\mathrm{graph}(u) \cap Q_\rho\big)  \geq \mathcal{H}^2_\varphi(L\cap Q_\rho) - C_\theta  \eta\rho^2 \ge \mathcal{H}^2_\varphi(\partial^-S_L) - C_\theta  \eta\rho^2\,.\EEE
$$
This concludes the proof of \eqref{ineq:graphdelE-}.

\begin{figure}
\begin{tikzpicture}
			
\begin{scope}[rotate=90]
\tikzset{>={Latex[width=1mm,length=1mm]}};
			
\draw(1.6,-.25) node {$L$};
			
\draw(0,0) rectangle (4,4);
\draw(1.8,2)++(80:2.5)--++(260:5);
			
\draw[fill=gray!10!white] plot [smooth, tension=.5] coordinates {(1.35,0)(1.45,.6)(1.5,.8)(1.55,1)(1.6,1.25) (1.65,1.4)(1.78,2)(1.7,3)(1.9,3.5)(1.8,4)}--(0,4)--(0,0)--(1.35,0);
			
\draw[thick] plot [smooth, tension=.5] coordinates {(1.35,0)(1.45,.6)(1.5,.8)(1.55,1)(1.6,1.25) (1.65,1.4)(1.78,2)(1.7,3)(1.9,3.5)(1.8,4)};
\draw(1.9,4.25) node {$\Gamma_i$};
			
\draw[pattern=north west lines,pattern color=gray](1.35,0)--(2,0)--++(80:4.05)--(1.8,4) plot [smooth, tension=.5] coordinates {(1.8,4)(1.9,3.5)(1.7,3)(1.78,2)(1.65,1.4)(1.6,1.25)    (1.55,1)(1.5,.8)(1.45,.6)(1.35,0) };
		
\draw(2.55,2) node {$T_i$};
			
\draw(.75,2) node {$E$};
			
\end{scope}

\begin{scope}[shift={(6,0)},rotate=90]
\tikzset{>={Latex[width=1mm,length=1mm]}};
			
\draw(1.6,-.25) node {$L$};
			
\draw(0,0) rectangle (4,4);
\draw(1.8,2)++(80:2.5)--++(260:5);

\draw[dashed,fill=gray!50!white](1.35,0)--++(-.65,0)--++(80:4.05)--(1.6,4)plot [smooth, tension=.5] coordinates {(1.6,4)(1.7,3.8)(1.5,3.2)(1.6,2.8)(1.5,2.2)(1.4,1.4)(1.3,.6)(1.2,0)};
			
\draw[fill=gray!10!white] plot [smooth, tension=.5] coordinates {(1.35,0)(1.45,.6)(1.5,.8)(1.55,1)(1.6,1.25) (1.65,1.4)(1.78,2)(1.7,3)(1.9,3.5)(1.8,4)}--(1.6,4) plot [smooth, tension=.5] coordinates {(1.6,4)(1.7,3.8)(1.5,3.2)(1.6,2.8)(1.5,2.2)(1.4,1.4)(1.3,.6)(1.2,0)}--(1.35,0);
			
\draw[thick] plot [smooth, tension=.5] coordinates {(1.6,4)(1.7,3.8)(1.5,3.2)(1.6,2.8)(1.5,2.2)(1.4,1.4)(1.3,.6)(1.2,0)};
			
\draw[thick] plot [smooth, tension=.5] coordinates {(1.35,0)(1.45,.6)(1.5,.8)(1.55,1)(1.6,1.25) (1.65,1.4)(1.78,2)(1.7,3)(1.9,3.5)(1.8,4)};
\draw(1.9,4.25) node {$\Gamma_i$};
\draw(1,-.25) node {$\Gamma_j$};

\draw[pattern=north west lines,pattern color=gray](1.35,0)--(2,0)--++(80:4.05)--(1.8,4) plot [smooth, tension=.5] coordinates {(1.8,4)(1.9,3.5)(1.7,3)(1.78,2)(1.65,1.4)(1.6,1.25)    (1.55,1)(1.5,.8)(1.45,.6)(1.35,0) };
			
\draw(2.55,2) node {$T_i$};
			
\draw(1.63,2) node {$E$};
			
\end{scope}
\end{tikzpicture}
			
\caption{The two fundamental cases for the selection of the sets $T_i$ (hatched).  In the second figure, \EEE the dark gray set is another connected component of $\overline{S_L}\setminus E$ that is not selected.}
\label{fig:Ti}
\end{figure}
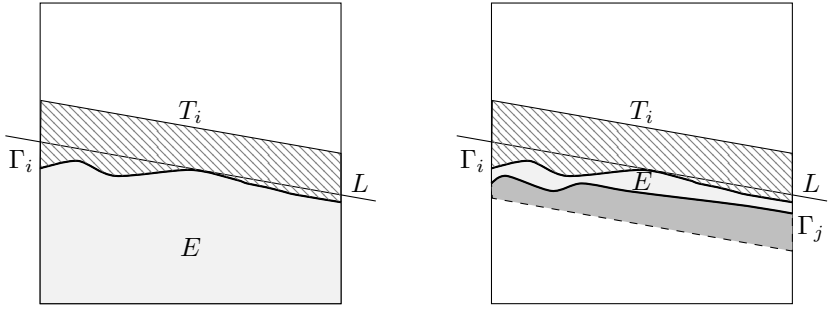
		
\EEE
		
We now set $\Gamma_i = \BBB {\rm graph}(u_i) \EEE \cap \OOO{Q_\rho}\EEE$ and let $T_i$ be the connected component of $\BBB \overline{S_{L_i}}\EEE\setminus {E}$ which contains $\Gamma_i$, see Figure~\ref{fig:Ti}. \EEE (From now on, \EEE for clarification we again add the index $i$ as the definition depends on $u_i \colon \AAA U_i\EEE \to L_i^\perp$.) Let us verify \eqref{eq: good}(i).  To this end, we observe by \eqref{eq:boundui} 
that
$${\partial T_i \setminus \overline{E}  \BBB \subset  \partial S_{L_i} \setminus (- 3\eta\rho\nu_{L_i} +   L_i)  = \partial^- S_{L_i}     \,,}$$
where the last identity \EEE is the definition of $\partial^-S_{L_i}$.   \EEE Thus,  \eqref{ineq:graphdelE-} implies \eqref{eq: good}(i). We close this step with the observation that \BBB for $x \notin E$ \EEE 
\begin{align}\label{eq: SSSS}
\dist(x, \Gamma_i) \BBB \le \EEE  \dist\Big(x, \bigcup\nolimits^M_{\BBB j\EEE=1} {\rm graph}(u_{\BBB j \EEE}) \cap \OOO{Q_\rho}\EEE\Big) <  \EEE \NNN \eta\rho \EEE \quad \quad \Rightarrow \quad \quad x \in T_i \,.
\end{align}
In fact, by using  \eqref{ineq:distance} and \OOO by assuming that  \EEE $\dist(x, \Gamma_i)  \OOO <  \NNN \eta\rho\EEE$,  we get $x \in S_{L_i}$, \BBB in particular $x \in {S_{L_i}} \setminus E$. \EEE If we had $x \in S_{L_i}\setminus T_i$, then we would necessarily find $\Gamma_j$, $j \neq i$, such that $\dist(x,\Gamma_j) < \dist(x,\Gamma_i)$, see also Figure~\ref{fig:Ti}. This is \EEE   a contradiction. \EEE
\end{step}\\	
\begin{step}{3}(Bad planes) Now we suppose that \BBB $L_i$ \EEE is a $\theta$-bad plane \OOO for $Q_{\rho}$. \EEE  Then, there exists exactly one  $k\in \{1,2,3\}$ such that $|\nu_k|\geq \theta$ and $|\nu_j|< \theta$ for $j \neq k$, and by Lemma \ref{lemma: bad plane} we find that  \eqref{ineq:uniformcloseness} holds. 
Without restriction we suppose that $k=1$ and that $x \in L_i\cap Q_{3\rho}$ implies that $-\frac{3\rho}{4}\leq x \cdot e_{\OOO 1}<-\frac{\rho}{4}$. In fact, the other cases can be treated along  \OOO the very same \EEE lines. We let $Q_\rho' := Q_\rho - \rho e_{\OOO 1\EEE} \in \mathcal{N}(Q_\rho)$ be the neighboring cube \OOO of $Q_{\rho}$ to the left of it, recall notation \eqref{eq_ neighb-c}.
		
Due to \eqref{ineq:uniformcloseness}, we have that $L_i$ is a $\theta$-good plane   for the shifted cube $\tilde{Q}_\rho :=  Q_\rho -\frac{\rho}{2}e_{\OOO 1\EEE} $.  \EEE In fact, \BBB Case~(2) \EEE of Definition \ref{theta_good_definition} \EEE is satisfied, provided that $\theta\OOO>0\EEE$ is chosen small enough. Consequently, \BBB given \eqref{eq: necessary assu} and the fact that $\tilde{Q}_{8\rho}\subset \Omega$, \EEE we can repeat the above reasoning for $\tilde{Q}_\rho$ in place of $Q_\rho$. Accordingly, we define $\Gamma_i \AAA:\EEE= {\rm graph}(u_{\OOO i\EEE}) \cap \tilde{Q}_\rho$  and $T_i$ as the connected component of  $\overline{S_{L_i}}\setminus {E}$ containing $\Gamma_i$, where now $S_{L_i} =   (\tilde{Q}_\rho)_{(1+ 6\EEE\eta)\rho} \cap (L_i)_{3\eta\rho}$. Then  \eqref{eq: good}(ii) can be proved along similar  lines as   \eqref{eq: good}(i) above, by using $\tilde{Q}_\rho \subset Q_\rho \cup Q_\rho'$.  In the same way, we obtain \eqref{eq: SSSS} \NNN in this case. We now observe that \eqref{eq: SSSS} for good and bad planes  yields \eqref{eq: bad}. \EEE In particular, we also note that
\eqref{ineq:uniformcloseness} and \eqref{ineq:graphdelE-} imply  
\begin{align}\label{eq LLLL}
\mathcal{H}^2_\varphi\big(\mathrm{graph}(u_i) \cap (Q_\rho \cup Q_\rho')\big) \geq  \frac{1}{C} \rho^2 \,,
\end{align}
provided that $\eta_0\OOO>0\EEE$ is chosen sufficiently small. 
		
Next, we confirm \eqref{eq: ugly}. To this end, we exemplarily apply the construction for the neighboring cube $Q_\rho' = Q_\rho - \rho e_{\OOO 1\EEE} $. By Lemma \ref{Simons_lemma} and Remark \ref{remark: simon} \EEE  (which are applicable by \eqref{eq: necessary assu} and the fact that $Q_{8\rho}'\subset \Omega$) \EEE we find planes $\BBB L'_j \EEE \subset \R^3$,  \MMM open sets \EEE $U'_j$ in $L'_j$, and  functions \OOO$u'_{j}$ \EEE such that \eqref{eq:partialESimon}--\eqref{eq:boundui} hold. Given the $\theta$-bad plane $L_i$ with corresponding ${\rm graph}(u_i)$ for the original cube $Q_\rho$ considered above, in view of \AAA \eqref{eq:partialESimon} \EEE applied for both $Q_\rho$ and $Q_\rho'$,  and by using \eqref{eq LLLL}, we observe that there exists a unique \OOO function \EEE $u_j'$ such that ${\rm graph}(u_j') \cap {\rm graph}(u_i) \cap (Q_\rho \cup Q_\rho') \neq \emptyset$. \BBB (In fact, since $\partial E\AAA \cap \Omega\EEE$ is a \AAA regular manifold with boundary only in $\partial \Omega$ and $\overline {Q_{12\rho}}\subset \Omega$\EEE, different graphs cannot intersect and the graphs of the functions in the above representation are unique\EEE.) \EEE  Then we observe that
one could replace \AAA ${\rm graph}(u_{i})$ and ${\rm graph}(u_j')$ \EEE in  Lemma~\ref{Simons_lemma} applied on $Q_\rho$ and $Q_\rho'$, respectively, by the union ${\rm graph}(u_{i}) \cup {\rm graph}(u_j')$ which can again be understood as the graph of a function defined on the plane $L_i$. This shows that the objects $\Gamma_j'$ and $T_j'$ for $L_j'$ can be chosen identical to $\Gamma_i$ and $T_i$, i.e., the sets can indeed be constructed such that \eqref{eq: ugly} is ensured. 
\end{step}\\
\begin{step}{4}(\NNN Presence of  pimples\EEE)
Now we argue how to reduce the case of existence of pimples to the case of non-existence of pimples. \EEE As a preparation, \EEE  we first show that for every pimple $P_j\subset \partial E$ such that $P_j\cap Q_{3\rho}\neq \emptyset$ \EEE there exists $i \in \{1,\ldots,M\}$ such that $P_j\cap \mathrm{graph}(u_i)\neq \emptyset$.  In fact, suppose by contradiction that this \EEE was \EEE not the case. Due to  the fact that Lemma~\ref{Simons_lemma} guarantees that $P_k\cap P_j=\emptyset$ for all $k\neq j$, 
\EEE we would get that $P_j$ is a compact manifold without boundary. Thus, applying \cite[Lemma~1.1]{Simon1993Willmore} we get 
$$\mathcal{H}^2(P_j) \le ({\rm diam}(P_j))^2  \int_{P_j} |\bm{H}|^2\,{\rm d}\mathcal{H}^2\,,$$ 
where $\bm{H}$ denotes the mean curvature. As the estimate  clearly still holds with $\bm{A}$ in place of $\bm{H}$ \BBB up to a factor of $2$, \EEE we get along with  \EEE H\"older's inequality for $q/2\geq 1$, \EEE \eqref{eq: surface energy}, \eqref{eq: necessary assu},   and  \eqref{ineq:diskestimate} that 
\begin{align*}
\mathcal{H}^2(P_j) &\leq \OOO 2 \EEE (\mathrm{diam}(P_j))^2 \int_{P_j}|\bm{A}|^2\,\mathrm{d}\mathcal{H}^2\leq \OOO 2\EEE\eta^2\rho^2 \big(\mathcal{H}^2(P_j)\big)^{1-2/q} \Big(\int_{P_j}|\bm{A}|^q\,\mathrm{d}\mathcal{H}^2\Big)^{2/q} \\&\leq \OOO 2\EEE\eta^2\rho^2 \big(\mathcal{H}^2(P_j)\big)^{1-2/q} \big(\Lambda\OOO\gamma^{-1}\EEE\rho^2\big)^{2/q}\,.
\end{align*}
Simplifying the above formula and using the assumption $\rho \le  \eta^7  \EEE \gamma^{1/q}$,  we have
\begin{align*}
\mathcal{H}^2(P_j) \leq \OOO 2^{q/2} \EEE \Lambda \gamma^{-1} \rho^2 \, \rho^q \eta^q \leq \OOO 2^{q/2}\EEE\Lambda\eta^{8q}  \rho^2\OOO< \EEE c_{\OOO 0 \EEE}\rho^2\,,
\end{align*}  
where the last step follows from the fact that \EEE $2^{q/2}\Lambda \eta_0^{8q} \le 2^{q/2}\Lambda \eta_0 \EEE <  c_{0}$, see \OOO the \EEE beginning of the proof \OOO for our choice of $\eta_0$. \EEE    By Lemma~\ref{lemma:slicing} applied to $P_j$ we would then obtain the estimate $\mathcal{F}_{\rm surf}^{\varphi,\gamma,q}(\AAA E;\EEE Q_{8\rho}) \ge \mathcal{F}_{\rm surf}^{\varphi,\gamma,q}(P_j) > \Lambda \rho^{2}$, where we used that $P_j \subset \partial E \cap Q_{8\rho}$, which follows from \eqref{ineq:diskestimate}. This is a contradiction.  Therefore, for all $j \in \{1,\ldots,N\}$ there exists an index $i \in \{1,\ldots,M\}$ such that $P_j \cap \mathrm{graph}(u_i) \neq \emptyset$. 

 Now, omitting again  the indices for simplicity,  we consider a plane $L$ and  a function $u\colon \AAA \overline{U}\EEE\subset L\to L^\perp$ satisfying \eqref{eq:boundui}, in particular $\Vert u \Vert_\infty \le \eta\rho$. \MMM Here, $U$ is
of the form $U=U^0\setminus \bigcup_{k}d_{k}$, where $U^0$ is a simply connected subdomain of $L$ and $(d_{k})_{k}$  are pairwise disjoint closed disks in $L$, which do not intersect $\partial U^0$. \EEE    If a pimple $P$ touches ${\rm graph}(u)$,  it can be covered by a cube  that also touches ${\rm graph}(u)$, has normal $\nu_L$ to one of its faces  (the orientation of the others being irrelevant), and sidelength \MMM $\mathrm{diam}(P)$. \EEE Due to \eqref{ineq:diskestimate}, performing this construction for  every pimple, the additional surface introduced by the cubes is bounded by $C\eta^2\rho^2$ \AAA for an absolute constant $C>0$\EEE. Furthermore, by this procedure we obtain a piecewise smooth function \EEE $\tilde{u}\colon \BBB \AAA U^0\EEE \subset L \to L^\perp$ such that $\Vert \tilde{u} \Vert_\infty \le \Vert u \Vert_\infty  +  \MMM \max_{j=1}^N \EEE \mathrm{diam}(P_{j})\le 2\eta\rho$, i.e.,  \eqref{eq:boundui} holds true, where  (the classical gradient) $\nabla \tilde{u}$ is well-defined up to a set of $\mathcal{H}^{\AAA 1\EEE}$-measure zero. Additionally, due to the diameter bound on the cubes, see  \eqref{ineq:diskestimate}, we have
\begin{align*}
\mathcal{H}^2(\mathrm{graph}(\tilde{u}) \cap Q_\rho) \leq \mathcal{H}^2(\mathrm{graph}(u) \cap Q_\rho) +C\eta^2\rho^2\,.
\end{align*}
Now Step 2 and Step 3 can be performed for the function $\tilde{u}$ instead of the function $u$ in order to conclude the proof. \EEE
\end{step}
\end{proof}

\begin{remark}[Obstacles in higher dimensions]\label{obstacle_higher_dimensions_generalization}
{\normalfont
We close this section by commenting on the current obstacles to generalize our results to higher dimensions. The two essential ingredients depending crucially on the dimension are Lemma~\ref{lemma:slicing} and Lemma~\ref{Simons_lemma}\AAA, \EEE whereas  the rest of our proof strategy can be carried along with very minor modifications.  Lemma~\ref{Simons_lemma} can in some sense be generalized to any dimension $d\geq 2$ in the spirit of \textit{$\varepsilon$-regularity results}, with respect to the $L^q$-norm of the second fundamental form, \AAA but \EEE for $q>d-1$. The result is due to {\sc
Hutchinson}, see \cite{allard1986geometric}, pages 281-306, in particular Theorem 3.7 on page 295, as well as \cite{Hutchinson}, and it is a \textit{graphical representation} rather than an \textit{approximation} result, i.e., the condition $q>d-1$ excludes the presence of pimples. As we have seen in Lemma \ref{lemma: curve graph}, for $d=2$ this graphical representation can easily be obtained for every $q\geq 1$, while for $d=3$ Simon's lemma also handles the case $1\leq q\leq 2$, modulo the presence of small pimples. For $d>3$, it would be interesting to investigate to which extent Simon's lemma can be generalized for $q=d-1$.

\AAA The other obstacle \EEE to generalize our result to higher dimensions, especially for the critical case $q=d-1$, \BBB is Lemma \ref{lemma:slicing}. As in the statement of Lemma \ref{touchtwo_spheres_with_small_curvature_implies_large_area}, the main question consists in the validity of the implication that 
\begin{equation*}
\mathcal{H}^{d-1}(\Sigma\cap B_R)\geq \AAA c\EEE_0 R^{d-1}
\end{equation*}
for 
every connected, \AAA regular \EEE  $(d-1)$-dimensional hypersurface $\Sigma$ in $\mathbb{R}^d$ with \AAA $\mathcal{H}^{\QQQ d-2\EEE}(\partial\Sigma\cap \overline{B_R})=0$, and \EEE   
\begin{align*}
\int_{\Sigma\cap B_R}|\bm{A}|^{d-2}\, {\rm d}\mathcal{H}^{d-1}<\alpha_0 R, 
\quad \Sigma\cap \partial B_{R}\neq \emptyset, \quad \text{and} \quad \Sigma\cap \partial B_{\mu R}\neq \emptyset\,,
\end{align*}
for suitable $\alpha_0=\alpha_0(\AAA d,\EEE\mu)>0$ and $\AAA c\EEE_0=\AAA c\EEE_0(\AAA d,\EEE\mu)>0$, $\mu\in (0,1)$.  In fact, this would allow us to repeat the proof of Lemma \ref{lemma:slicing} for \BBB $q \ge d-1$. \EEE  Whereas the above implication holds true in $d=2$ and $d=3$, to the best of our knowledge it is an open question for $d>3$. For related results in higher dimensions, yet not sufficient for our purposes, we refer to \cite[Theorem 1.1]{topping2008relating} and \cite[Theorem A]{Menne2017}.

}
\end{remark}

  \section{Applications}\label{sec: applications}

This section is devoted to applications of our rigidity result. We identify effective linearized models of nonlinear elastic energies in the small-strain limit in two settings, namely \EEE for a model with  material voids in elastically stressed solids  and for epitaxially strained \OOO elastic thin \EEE films.  In the following, \OOO for $d=2,3$ \EEE we let $\Omega \subset \mathbb{R}^{\OOO d\EEE}$ be a \BBB bounded  \EEE Lipschitz domain, and $W \colon \mathbb{R}^{\OOO d\times d\EEE} \to [0,+\infty)$ be a frame-indifferent stored \OOO elastic \EEE energy density with the usual assumptions \OOO in \EEE nonlinear elasticity. Altogether, we suppose that $W$ satisfies the following assumptions \EEE
\begin{align}\label{eq: nonlinear energy}
\begin{split}
{\rm (i)} & \ \  \text{Frame indifference: $W(RF) = W(F)$ for all $R \in SO(d)$, $F\in \mathbb{R}^{d\times d}$}\,,\\ 
{\rm (ii)} & \ \ \OOO\text{Single energy-well structure: } \EEE  \{W=0\} = SO(d)\,,\\
{\rm (iii)} & \ \ \OOO\text{Regularity:} \EEE   \ \ \text{$W \in C^3$ in a neighborhood of $SO(d)$}\,,\\
{\rm (iv)} & \ \ \OOO\text{Coercivity:} \EEE \ \  \text{There exists $c>0$ such that for all $F \in \mathbb{R}^{d\times d}$ it holds that} \\ 
& \quad   \quad \quad  \quad \quad \quad \, W(F) \geq c\, \mathrm{dist}^2(F,SO(d))\,. 
\end{split}
\end{align}
\OOO Notice that the above assumptions \BBB particularly \OOO  imply that $DW(\BBB {\rm Id} \EEE )=0$. \EEE The general approach  in  linearization results  in many different settings (see, e.g., \cite{alicandro.dalmaso.lazzaroni.palombaro, Braides-Solci-Vitali:07, DalMasoNegriPercivale:02, MFMK, Friedrich-Schmidt:15, NegriToader:2013, Schmidt:08, Schmidt:2009}) is to consider sequences of deformations $(y_\delta)_{\delta\OOO>0\EEE}$ with small elastic energy, more precisely 
\OOO
\begin{equation*}
\sup\nolimits_{\delta\OOO>0\EEE} \delta^{-2}\int_\Omega W(\nabla y_\delta)\ \, {\rm d}x \EEE <+\infty\,,
\end{equation*}
\EEE and to pass to \OOO the \EEE small-strain limit \OOO as   \EEE $\delta \to 0$, in terms of \emph{rescaled displacement fields}, i.e., \OOO mappings \EEE
\begin{align}\label{eq: rescali1}
u_\delta = \frac{1}{\delta}(y_{\OOO \delta\EEE} - {\rm id}).
\end{align}
These  \OOO maps  \EEE measure the distance of the deformations from the identity, rescaled by the typical strain $\delta\OOO>0\EEE$. This yields a linearization of the elastic energy, which can be expressed in terms of the quadratic form   $\mathcal{Q}\colon \R^{d \times d} \to [0,+\infty)$  defined by 
 \begin{align}\label{eq: quadratic form}
\mathcal{Q}(F) := D^2W({\rm Id})F : F \quad \text{ for all $F \in \R^{d \times d}$}.
\end{align}
In view of  \eqref{eq: nonlinear energy},  $\mathcal{Q}$ is positive-definite on $\R^{d \times d}_{\rm sym}$ and vanishes on $\R^{d \times d }_{\rm skew}$. We will consider models containing surface energies with an additional curvature regularization as indicated in \eqref{eq: surface energy}, where we choose  \BBB a sequence of scaling parameters $(\gamma_\delta)_{\delta >0} \subset (0,+\infty)$ for which we require \EEE 
 \begin{align}\label{eq:CdeltaRate-appli}
 { \BBB \gamma_\delta \to 0 \quad \quad \quad \text{and} \EEE \quad \quad \quad 
\liminf_{\delta \to 0}\left(   \delta^{-\frac{q}{3d}}  \gamma_\delta   \right) =+\infty\,.}
\end{align}
In fact, this allows us to define a sequence $(\kappa_\delta)_{\delta\OOO>0\EEE} \subset (0,+\infty)\EEE$ satisfying 
\begin{align}\label{eq: kappa-pro}
\delta \kappa_\delta^3 \to 0, \quad \quad \quad \gamma_\delta^{d/q}\OOO \kappa\EEE_\delta \to \infty\, \quad \OOO\text{as}\ \ \delta\to 0,
\end{align}
which will play a pivotal role in the linearization procedure. 
In the following, we will focus on \AAA a \EEE curvature regularization in terms of the second fundamental form $\bm{A}$. Under certain assumptions however, \OOO in the case $d=3$, $q=2$, \EEE $\bm{A}$ can be replaced by the mean curvature \OOO $\bm{H}$. \EEE We refer to  Corollaries \ref{cor: mean1} and \ref{cor: mean2} for details in this direction.

In our applications, it will turn out that limiting  \OOO mappings  \EEE lie in the space of generalized special functions of bounded deformation $GSBD^2(\Omega)$. For basic properties of $GSBD^2(\Omega)$, we refer to \cite{DalMaso:13} and  \BBB Appendix  \ref{sec: GSBD}  \EEE below. In particular, for $u\in GSBD^2(\Omega)$, we will denote by $e(u) = \frac{1}{2}(\nabla u + \nabla u^{T})$  the approximate symmetric differential   and by $J_u$ the jump set \OOO of $u$ \EEE with  measure-theoretical normal $\nu_u$. Moreover, by $L^0(\Omega;\R^d)$ we denote  the space of $\mathcal{L}^d$-measurable \OOO mappings \EEE $v \colon \Omega \to \R^d$, endowed with the topology of the convergence in measure. For any $s \in [0,1]$ and any $E \in \M(\Omega)$, $E^s$ denotes the set of points with  \OOO $d$-dimensional \EEE density $s$ \OOO with respect to  \EEE $E$. By $\partial^* E$ we indicate \OOO the  \EEE essential boundary \OOO of $E$\EEE,  see \cite[Definition 3.60]{Ambrosio-Fusco-Pallara:2000}.
 
We \BBB now \EEE present  our two applications in Subsections \ref{sec: results1}--\ref{sec: results2}. The proofs \EEE of the results are deferred to Subsections \ref{sec: compre}--\ref{sec: compre2}.

\subsection{Material voids in elastically stressed solids}\label{sec: results1}

We study boundary value problems for  elastically stressed solids with voids. We suppose that the boundary data are  imposed  on an \BBB open \EEE subset  $\partial_D \Omega \subset \partial \Omega$ and are close to the identity. To this end, let   $u_0 \in W^{1,\infty}(\R^d;\R^d)$, \OOO$d=2,3$\EEE, and for $\delta >0$ define $y_0^\delta := {\rm id} + \delta u_0$.   Let  further  $\varphi$ be a norm, \OOO $\AAA q\in [d-1,+\infty)\EEE$, and  \BBB $(\gamma_\delta)_{\delta >0}$ \EEE as in \eqref{eq:CdeltaRate-appli}.  Then for the density  $W\colon  \R^{d \times d}\to  [0,\infty) $ introduced in \eqref{eq: nonlinear energy}, we let $F_\delta\colon  L^0( \Omega;  \R^d) \times \M(\Omega) \to \AAA [0,+\infty]\EEE$   be   the functional  defined by 
\begin{equation}\label{eq: F functional}
F_\delta(y,E) \OOO:\EEE= 
\OOO \frac{1}{\delta^2}\EEE\int_{\Omega \setminus \OOO \overline{E}} 
W(\nabla y)\, \dx + \int_{\OOO\partial E\EEE\cap\Omega
}  
\varphi(\nu_E) \, \d\mathcal{H}^{d-1} 
+ \gamma_\delta\int_{\partial E \cap \Omega} |\bm{A}|^q \, {\rm d}\mathcal{H}^{d-1}\,,
\end{equation}
if $E \in\mathcal{A}_{\rm reg}(\Omega)$, $\QQQ\overline{E}\cap \partial_D\Omega=\emptyset\EEE$,  $y|_{\Omega\setminus \overline E} \in H^1( \Omega \setminus \overline E;\R^d)$,  $y|_E={\rm id}$, and $\mathrm{tr}(y)  =\mathrm{tr}(y^\delta_0)$ on $\QQQ\partial_D\Omega\EEE
$, and $F_\delta(y,E) = +\infty$   \OOO otherwise. \EEE Here,  $\nu_E$ denotes \OOO again  \EEE the outer \OOO unit \EEE normal to $\OOO\partial E\EEE$. We emphasize that the energy is determined by $E$ and the values of $y$ on $\Omega \setminus \overline E$. The condition $y|_E={\rm id}$ is for definiteness only. The relaxation of this model \AAA without the curvature regularization term \EEE has been studied in \BBB \cite{BraChaSol07, santilli}. \EEE Here, instead, we are interested in an effective description in the small-strain limit $\delta \to 0$, in terms of displacement fields \OOO defined in  \EEE \eqref{eq: rescali1}.  From now on, we write 
$$\mathcal{F}_\delta(u,E) := F_\delta({\rm id} + \delta u, E)$$
for notational convenience. We start with a compactness result which fundamentally relies on Theorem \ref{prop:rigidity}. Note that in what follows, the \AAA sets $\omega_u^\delta, \omega_u$ serve \EEE a totally different purpose, and should not be confused with the set \AAA $\omega_u^\rho$ \EEE in Section~\ref{sec: rigidi}, see for instance \eqref{ineq:area_of_domains_of_parametrization}.\EEE 

\begin{proposition}[Compactness, void case\EEE]\label{prop: compi1}
For  every \OOO sequence of pairs $(u_\delta, E_\delta)_{\delta>0}$ with 
$$\QQQ M:=\EEE\sup\nolimits_{\delta>0} \mathcal{F}_\delta(u_\delta, E_\delta) < +\infty,$$ there exist a subsequence (not relabeled), $u \in GSBD^2(\Omega)$,  sets of finite perimeter $E \in \M(\Omega)$,  $(E_\delta^*)_{\delta\OOO>0} \subset \M(\R^d)$ with  $E_\delta \subset E_\delta^*$, as well as sets $\omega_u , (\omega_u^\delta)_{\delta>0} \subset \M(\Omega)$  such that $u\equiv 0$ on $E \cup \omega_u$,  \QQQ$$\mathcal{H}^{d-1}(\partial^* \omega_u)+ \sup_{\delta > 0} \mathcal{H}^{d-1}(\partial^* \omega_u^\delta) \le C_M $$ for a constant $C_M>0$ depending only on $M$, \EEE and as $\delta\to 0$,
\begin{align}\label{eq:lsc0}
\begin{split}
{\rm (i)} & \ \ u_\delta \to u \text{ in measure on $\Omega \setminus \omega_u$}\,,\\
{\rm (ii)} &  \ \ \chi_{\Omega \setminus (E^*_\delta \cup \omega^\delta_u)} \EEE e(u_\delta) \rightharpoonup \OOO \chi_{\Omega \setminus (E \cup \omega_u)} e(u)  \ \ \text{weakly in $L^2_{\rm loc}(\Omega;\R^{d\times d}_{\rm sym})$}\,,\\
{\rm (iii)} & \ \ \mathcal{L}^d(\lbrace |\nabla u_\delta|>\kappa_\delta \rbrace \setminus \omega_u )  \to 0\,,\\
{\rm (iv)} & \ \   \liminf_{\delta \to 0}   \int_{\partial{E_{\delta}^{*}}\cap \Omega } \varphi(\nu_{E^*_{\delta}}) \, {\rm d}\mathcal{H}^{d-1} \le  \liminf_{\delta \to 0}\mathcal{F}_{\rm surf}^{\varphi,\gamma_\delta,q}(E_\delta)\,,\\
\QQQ{\rm (v)} &\QQQ \ \ \lim_{\delta \to 0 }\mathcal{L}^d(\omega_u^\delta \triangle \omega_u) = \lim_{\delta \to 0} \mathcal{L}^d(E_\delta^*  \setminus  E_\delta)=\lim_{\delta\to 0}\mathcal{L}^d(E_\delta\triangle E) =0\,, 
 \end{split}
\end{align}
where $\kappa_\delta$ is defined in \eqref{eq: kappa-pro} and $\mathcal{F}_{\rm surf}^{\varphi,\gamma_\delta,q}$ in \eqref{eq: surface energy}.

\end{proposition}
In the following, we say that  a sequence $(u_\delta,E_\delta)_{\OOO \delta>0\EEE} \subset L^0(\Omega;\mathbb{R}^{\OOO d\EEE}) \times \mathfrak{M}(\Omega)$ converges to \OOO a pair  \EEE $(u,E) \in L^0(\Omega;\mathbb{R}^{\OOO d\EEE}) \times \mathfrak{M}(\Omega)$ in the $\tau$-sense \EEE and write $(u_\delta,E_\delta) \overset{\tau}{\to} (u,E)$ \OOO iff \EEE there \NNN exists a set \EEE     $\omega_u \in \M(\Omega)$ such that  $\chi_{E_\delta} \to \chi_E \text{ in }L^1(\Omega)$, $u_\delta \to u$ in measure on $\Omega \setminus \omega_u$, and $u\equiv 0$ on $E \cup \omega_u$.\\[-10pt] 

The compactness result is non-standard in the sense that the behavior of the sequence \OOO$(u_\delta)_{\delta>0}$ \NNN on \EEE $\omega_u$  cannot be controlled. This set is related to the fact that   $\Omega \setminus \overline{E_{\BBB \delta}}$ might be disconnected into various \BBB connected components \EEE $(P^{\delta}_j)_j$  by $E_{\OOO\delta\EEE}$, and on the sets not intersecting $\partial_D \Omega$ the corresponding rotations $R^{\delta}_j$, obtained from \EEE \eqref{eq: main rigitity}, cannot be controlled. It is however essential that $|R^{\delta}_j - {\rm Id}|$ is at most of order $\delta$, as otherwise $u_\delta$ defined in \eqref{eq: rescali1} blows up on $P^{\delta}_j$. In this sense, roughly speaking, \BBB $\omega^\delta_u$ \EEE consists of the components $(P^{\delta}_j)_j$ not intersecting $\partial_D \Omega$. \BBB Moreover, the sets $E_\delta$ need to be replaced by the slightly larger sets $E_\delta^*$ corresponding to the sets in \eqref{eq: partition}. \EEE

We now introduce the linearized model studied in \cite{Crismale}. 
Given $u \in GSBD^2(\Omega)$ and $E  \in \M(\Omega)$  with $\mathcal{H}^{d-1}(\partial^* E) <+\infty$,  we first define the \emph{boundary energy \EEE term} by
\begin{align}\label{eq: bdaypart}
\mathcal{F}^{\rm bd\OOO r \EEE y}(u,E) :=   \int_{\OOO\partial^* E\cap\EEE\partial_D\Omega}  \varphi  (\nu_E) \,{\rm d}\mathcal{H}^{d-1}  +  \int\limits_{ \{ \mathrm{tr}(u)  \neq \mathrm{tr}(u_0)  \} \cap  (\partial_D \Omega \setminus \partial^* E) }  \hspace{-0.5cm} 2 \, \varphi(  \nu_\Omega  ) \, {\rm d}\mathcal{H}^{d-1}, 
\end{align}
which is nontrivial if the void goes up to the \KKK Dirichlet part of the \EEE boundary or the   \OOO mapping  \EEE $u$ does not satisfy the imposed boundary conditions. Here, $\nu_\Omega$ denotes the outer unit normal to \EEE $\partial \Omega$, and ${\rm tr}(u)$ indicates the trace of $u$ at $\partial \Omega$, which  is well defined for functions in $GSBD^2(\Omega)$, see \OOO Appendix \ref{sec: GSBD}.   \EEE Recalling the definition of $\mathcal{Q}$ in \eqref{eq: quadratic form}, we introduce the \emph{effective  limiting energy} $\mathcal{F}_0 \colon  L^0(\Omega;\R^d) \times \M(\Omega) \to [0,+\infty]\EEE$  by
\begin{align}\label{def:F}
\mathcal{F}_0(u,E) &\OOO: \EEE =   \frac{1}{2}\int_{\Omega\setminus E} \mathcal{Q}(e(u)) \,\mathrm{d}x  + \int\limits_{\partial^* E\AAA\cap \Omega\EEE} \hspace{-0.2cm} \varphi  (\nu_E) \,{\rm d}\mathcal{H}^{d-1}  + \int\limits_{J_u \setminus \partial^* E}\hspace{-0.2cm} 2\, \varphi(\nu_u)  \,{\rm d}\mathcal{H}^{d-1} + \mathcal{F}^{\rm bd\OOO r \EEE y}(u,E) 
\end{align}
if $\mathcal{H}^{d-1}(\partial^* E) <+\infty$ and $u = \OOO \chi_{\Omega \setminus E} \EEE u \in GSBD^2(\Omega)$, and  $\mathcal{F}_0(u,E) = +\infty$ otherwise.

We now address that \eqref{def:F} can be identified as \OOO the \EEE $\Gamma$-limit of \eqref{eq: F functional} for $\delta\to 0$. In fact, the functional \eqref{def:F} is effective in two respects: first, in the small-strain limit the density of nonlinear elasticity is replaced by \OOO its linearized version \EEE $\mathcal{Q}$. Secondly, the fact that $\mathcal{F}_\delta$ is not lower semicontinuous  in the variable $E$ with respect to $L^1$-convergence of sets  is remedied by a suitable relaxation. Indeed, in the limiting process, the voids $E$ may collapse into a discontinuity of the displacement $u$. In particular, this phenomenon is taken into account in the relaxed  \OOO functional \EEE since collapsed surfaces are counted twice in the \OOO surface \EEE energy. Eventually, we point out \EEE that, due to \OOO the fact that \EEE $\gamma_\delta \to 0$ as $\delta\to 0$, the curvature regularization of the nonlinear energy $\mathcal{F}_\delta$ does not affect the linearized limit.

 For the $\Gamma$-limsup inequality,   more precisely for the application of a  density result in $GSBD^2$,  see \cite[Lemma 5.7]{Crismale},   we make the following geometrical assumption on the Dirichlet boundary $\partial_D\Omega$: there exists a decomposition $\partial \Omega = \partial_D \Omega \cup \partial_N\Omega \cup N$ with 
\begin{align*}
\partial_D \Omega, \partial_N\Omega \text{ relatively open}, \ \  \  \mathcal{H}^{d-1}(N) = 0, \ \ \  \partial_D\Omega \cap \partial_N \Omega = \emptyset, \ \ \  \partial (\partial_D \Omega) = \partial (\partial_N \Omega),
\end{align*}
\OOO where the outermost boundary has to be understood in the relative sense,  \EEE and there exist $\bar{\sigma}>0\EEE$ small \OOO enough \EEE and $x_0 \in\R^d$ such that for all $\sigma\EEE \in (0,\bar{\sigma}\EEE)$ it holds that
\begin{align*}
O_{\sigma\EEE,x_0}(\partial_D   \Omega   ) \subset \Omega,
\end{align*}
where $O_{\sigma\EEE,x_0}(x) := x_0 + (1-\sigma\EEE)(x-x_0)$.  \EEE  Recall the convergence $\tau$ introduced below Proposition~\ref{prop: compi1}.

\begin{theorem}[$\Gamma$-convergence\OOO, void case\EEE] \label{theorem:convergence}  Under the above assumptions, as $\delta\to 0$, we have that the sequence of functionals $\OOO(\EEE\mathcal{F}_\delta\OOO)\EEE_{\OOO \delta>0\EEE}$ $\Gamma$-converges to $\mathcal{F}_0$ with respect to the convergence $\tau$.   
\end{theorem}
\begin{remark}[Volume of voids]\label{rem:voids}
{\normalfont (i) In the previous result, if $\mathcal{L}^d(E)>0$, then for any $(u,E) \in L^0(\Omega;\R^d){\times}\M(\Omega)$ there exists a recovery sequence $(u_\delta,E_\delta)_{\delta\OOO>0\EEE}\subset L^0(\Omega;\R^d){\times}\M(\Omega)$ such that $\mathcal{L}^d(E_\delta) = \mathcal{L}^d(E)$ for all $\delta >0$. This shows that it is possible to incorporate a volume constraint on $E$ in the $\Gamma$-convergence result.\\
(ii) If we impose  \OOO the  \EEE condition $\mathcal{L}^d(E_\delta) \to 0$ along the sequence, we obtain $E = \emptyset$, and the limiting model corresponds to an (anisotropic) \emph{Griffith energy of brittle fracture}.
} 
\end{remark}

We address an alternative formulation with the \EEE mean curvature in place of the \EEE second fundamental form, \OOO in the case $d=3$, $q=2$. \EEE 

\begin{corollary}[Mean curvature regularization]\label{cor: mean1}
\normalfont We consider \eqref{eq: F functional} with $|\bm{H}|^2$ in place of $|\bm{A}|^{2}$ \OOO when $d=3$, $q=2$\EEE. \EEE We suppose that for $\mathcal{F}_\delta$, only sets $E$ satisfying $E \subset \subset \Omega$ and $- 4\pi\chi(\OOO\partial \EEE E) \le  \lambda_\delta\gamma_\delta^{-1}$
for some $\lambda_\delta \to 0$ \OOO are admissible, \EEE where $\chi(\OOO\partial \EEE E)$ indicates the Euler characteristic \OOO of $\partial E$. (For instance, this holds if $\OOO \partial E$ consists of  connected components topologically equivalent to the sphere $\mathbb{S}^2$.) Then,  the statements of Proposition~\ref{prop: compi1}  and Theorem~\ref{theorem:convergence} hold. 
 \end{corollary}


%
%
%
%
%

\subsection{Energies on domains with a subgraph constraint: epitaxially strained films}\label{sec: results2}

We now address a second application, \BBB namely \EEE deformations of an elastic material in a domain which is the subgraph of an unknown nonnegative function $h$.  Assuming that $h$ is defined on a smooth \OOO bounded \EEE domain $\omega \subset \R^{d-1}$, $d=2,3$, \EEE  deformations $y$ will be defined on the
subgraph 
$$\Omega_h^+ := \{ x \in \omega \times \R \colon 0  <  x_d   <   h(x')\},$$
where here and in the following we use the notation $x = (x',x_d)$ for $x \in \R^d$. To model Dirichlet boundary data on \EEE the flat surface $\omega \times \lbrace 0 \rbrace$, we will suppose that  \OOO mappings  \EEE are extended to the set $\Omega_h := \{ x \in  \omega \times \R  \colon -1 <  x_d   <  h(x')\}$ and satisfy $y = y^\delta_0  := {\rm id} + \delta u_0$ on $\omega{\times}(-1,0\OOO ]\EEE$ for a given function $u_0 \in W^{1,\infty}(\omega{\times}(-1,0\OOO]\EEE;\R^d)$.  In the application to epitaxially strained films, $y^\delta_0$ represents the interaction with the substrate and   $h$ indicates the profile of  the  free surface of the film. \BBB We refer to \cite{BonCha02, ChaSol07, Crismale} for a thorough description of the model and  a detailed account of the available literature. \EEE

For convenience, we introduce the reference domain $\Omega:=\omega{\times} (-1, M+1)$ for some $M>0$. For \EEE $q\in [d-1,+\infty)\EEE$, $\gamma_\delta$ as in \eqref{eq:CdeltaRate-appli}, and the density  $W\colon  \R^{d \times d}\to  [0,\infty) $ introduced in \eqref{eq: nonlinear energy}, we  define the energy $G_\delta\colon    L^0(\Omega;\R^d) \times L^1(\omega;[0,M]) \to \AAA [0,+\infty]\EEE$ by  
\begin{equation}\label{eq: Gfunctional}
G_\delta(y,h) \OOO:\EEE=  
\frac{1}{\delta^2}\int_{\Omega_h^+} W ( \nabla y(x)) \, \dx + \mathcal{H}^{d-1}\big(\partial \Omega_h \EEE \cap \Omega \big) + \gamma_\delta\int_{\partial \Omega_h \EEE \cap \Omega} |\bm{A}|^q \, {\rm d}\mathcal{H}^{d-1}\,,
\end{equation}
if $h \in \BBB C^2\EEE (\omega;[0,M])$,  $y|_{\Omega_h} \in H^1(\Omega_h;\R^d)$,  $y={\rm id}$  in $\Omega\setminus \overline{\Omega_h}$, $y=y^\delta_0$  in $\omega{\times}(-1,0\OOO]\EEE$, and $G_\delta(y,h) := +\infty$ otherwise. We emphasize that the two surface terms only contribute in terms of  the upper surface $\partial \Omega_h\EEE \cap \Omega$ of the film, which exactly corresponds to the graph of $h$. In other words, the first surface term is exactly $\int_{\omega} \sqrt{1 + |\nabla h(x')|^2} \, \dx' $. On the other hand, the curvature term can be written as $\int_{\omega}|\nabla^2 h(x')|^{q}(1+|\nabla h(x')|^2)^{\frac{1-q}{2}}\,\mathrm{d}x' $.  \EEE Note that this model can be seen as a special case of \eqref{eq: F functional} when we choose  $E = \Omega \setminus \overline{\Omega_h}$. As \OOO in \EEE Subsection \ref{sec: results1}, the assumption $y={\rm id}$  in $\Omega\setminus\overline{\Omega_h}\EEE$ is for definiteness only. 

The relaxation of this model has been studied in \cite{ChaSol07}. Notice  that, in contrast to \cite{BonCha02, ChaSol07}, \OOO here  \EEE we assume  that the functions  $h$ are  equibounded by a value $M$: this is for technical reasons only and is justified from a mechanical point of view, as indeed other \OOO physical \EEE effects come into play for  very high crystal profiles. In the present work, we address the effective behavior of the model  in the small-strain limit $\delta \to 0$, again in terms of displacement fields \OOO as defined in  \EEE \eqref{eq: rescali1}. From now on, we write 
$$\mathcal{G}_\delta(u,h) \OOO: \EEE = G_\delta({\rm id} + \delta u, h)$$
for notational convenience.  Based on Theorem \ref{prop:rigidity}, we obtain the following compactness result.

\begin{proposition}[Compactness, graph case]\label{prop: compi2} 
For any \OOO sequence of pairs $(u_\delta, h_\delta)_{\delta\OOO>0\EEE}$ with 
$$\QQQ K:=\EEE{\sup\nolimits_{\delta\OOO>0} \mathcal{G}_\delta(u_\delta,h_\delta) <+\infty,}$$
 there exist a subsequence (not relabeled),   sets of finite perimeter $(E_\delta^*)_{\delta>0} \subset \M(\Omega)$ with $\Omega \setminus \overline{\Omega_{h_\delta}} \subset E_\delta^*$, as well as $(\omega_u^\delta)_{\delta>0}  \subset \M(\Omega)$,  and functions $u \in GSBD^2(\Omega)$, $h \in BV(\omega;[0,M])$  with $u = \chi_{\Omega_h} \EEE u$ and $u=u_0$ on $\omega \times (-1,0]$  such that  \QQQ$$\sup_{\delta > 0} \mathcal{H}^{d-1}(\partial^* \omega_u^\delta) \le C_K $$ for a constant $C_K>0$ depending only on $K$, and as $\delta\to 0$, \EEE
\begin{align}\label{eq:lsc0XXXX}
\begin{split}
{\rm (i)} & \ \ u_\delta \to u \text{ in measure on $\Omega$}\,,
\\
{\rm (ii)} &  \ \ \chi_{\Omega \setminus (E^*_\delta \cup \omega_u^\delta)}  e(u_\delta) \rightharpoonup e(u) = \chi_{\Omega_h} e(u)   \ \ \text{weakly in $L^2_{\rm loc}(\Omega;\R^{d\times d}_{\rm sym})$}\,,
\\
{\rm (iii)} & \ \ \mathcal{L}^d(\lbrace |\nabla u_\delta|>\kappa_\delta \rbrace )  \to 0\,,
\\
{\rm (iv)} & \ \    \liminf\nolimits_{\delta \to 0}   \mathcal{H}^{d-1}(\partial E^*_{\delta} \cap \Omega)\le  \liminf\nolimits_{\delta \to 0}\mathcal{F}_{\rm surf}^{q,\delta}(E_\delta)\,,\\
\QQQ {\rm (v)} & \QQQ\ \ \lim_{\delta\to 0}\|h_\delta-h\|_{L^1(\omega)}=\lim_{\delta \to 0 }\mathcal{L}^d(\omega_u^\delta)  = \lim_{\delta \to 0 }\mathcal{L}^d(E_\delta^* \cap \Omega_{h_\delta}) = 0\,, 
\end{split} 
\end{align}
 where  $\kappa_\delta$ is defined in \eqref{eq: kappa-pro}, and $\mathcal{F}_{\rm surf}^{q\OOO,\delta\EEE}$ in \eqref{eq: surface energy} for $\varphi \equiv 1$ \OOO and $\gamma=\gamma_\delta$. \EEE
\end{proposition}
We note that in contrast to Proposition \ref{prop: compi1} no exceptional set $\omega_u$ is needed \OOO here. \EEE Indeed, in this setting we obtain a stronger compactness result due to the graph constraint on $\partial \Omega^+_{h_\delta} \cap\Omega$.

We now introduce the effective model studied \OOO in \EEE \cite{Crismale}.  Recalling the definition of $\mathcal{Q}$ in \eqref{eq: quadratic form}, we introduce  $\mathcal{G}_0 \colon  L^0(\Omega;\R^d) \times L^1(\omega; [0,M]) \to \AAA[0,+\infty]\EEE$  by
\begin{align}\label{def:G}
\mathcal{G}_0(u,h) &\OOO :\EEE=   \frac{1}{2}\int_{\Omega_h^+} \mathcal{Q}(e(u)) \,\mathrm{d}x  +\mathcal{H}^{d-1}(\partial^* \Omega_h   \cap \Omega  ) + 2 \mathcal{H}^{d-1}(J_u' \cap \Omega_h^1)
\end{align}
if  $u=  \OOO \chi_{ \Omega_h } \EEE u  \in GSBD^2(\Omega)$, $u=u_0 \text{ in }\omega{\times}(-1,0\OOO]\EEE$, $h \in BV (\omega;  [0,M]  )$, and  $\mathcal{G}_0(u,h) = +\infty$ otherwise.   Here,    $e(u) = \frac{1}{2}\left(\nabla u +\nabla u^T\right)$ again denotes the symmetric part of the  (approximate)  gradient of $u \in GSBD^2(\Omega)$, \BBB $\Omega^1_h$ denotes the set of points with  density $1$,   \EEE and 
\begin{align}\label{eq: Ju'}
J_u' := \lbrace (x',x_d + t)\colon  \, x \in J_u, \,  t \ge 0 \rbrace\,.
\end{align}
 As for the functional  \eqref{def:F}, \OOO the \EEE energy  \eqref{def:G}  is effective  in the sense  that the \AAA elastic energy \EEE density $W$ is replaced by the linearized density $\mathcal{Q}$ and the model accounts for ``vertical cuts'' $J_u' \cap \Omega^1_h$ (\OOO see \EEE\cite{FonFusLeoMor07}) which may appear along the relaxation process. Similarly to the \AAA corresponding \EEE term in \eqref{def:F}, this part is counted twice in the energy. The set $(\partial^{\BBB *} \Omega_h \cap \Omega) \cup \BBB (J_u' \cap \Omega^1_h) \EEE $  can be interpreted as a ``generalized interface'', cf.~Figure~\ref{fig:graphcase} for a two dimensional section of a possible limiting $\Omega_h$. As before, due to \OOO the fact that \EEE $\gamma_\delta \to 0$ as $\delta\to 0$, the curvature regularization of the nonlinear energy $\mathcal{G}_\delta$ does not affect the linearized limit.   

We work under the additional assumption that   $\omega \subset \R^{d-1}$ is uniformly star-shaped with respect to the origin, i.e.,
\begin{align*}
tx  \OOO \in  \EEE \omega \ \ \ \text{for all} \ \ t \in \AAA [\EEE 0,1), \, x \in \partial \omega.
\end{align*}
This condition, however, is   only of technical nature and could be dropped  at  the expense of more elaborated estimates, see also \cite{ChaSol07, Crismale}.  We obtain the following result.

\begin{theorem}[$\Gamma$-convergence\OOO, graph case\EEE]\label{theorem:convergence2} Under the above assumptions, as $\delta\to 0$, we have that the sequence of functionals $\OOO(\EEE\mathcal{G_\delta}\OOO)\EEE_{\delta\OOO>0\EEE}$ $\Gamma$-converges to \AAA the functional \EEE $\mathcal{G}_0$ with respect to the topology of \linebreak $L^0(\Omega;\R^d){\times}L^1(\omega; [0,M])$.
\end{theorem}

\begin{remark}[Volume constraint]\label{rem: vol2}
{\normalfont 
We note that along the linearization process one could consider an additional volume constraint on the film, i.e., $\mathcal{L}^d(\Omega_h^+) = \int_\omega h(x') \, \d x'$ is fixed. } 
\end{remark}

\begin{figure}
\begin{tikzpicture}
\draw[fill=gray](0,1)--(8,1)--(8,2)--(0,2);
\draw[fill=gray!10!white] plot [smooth] coordinates {(0,3)(1,3)(3,4)(4,4)(4.2,2)}--(0,2);
\draw[fill=gray!10!white] plot [smooth] coordinates {(8,3.75)(6,3.5)(5.5,2.5)(5,3)(4.3,4)(4.2,2)}--(8,2);

\draw(9,1.5) node {$\omega \times (-1,0)$};
\draw(9,3.8) node {$\partial^* \Omega_h   \cap \Omega $};
\draw(4,.75) node {$\omega$};
\draw(-.25,3) node {$h$};
\draw(2,3.5)--++(270:1);
\draw(2.25,2.5) node {$J_u' \cap \Omega_h^1\EEE$};

\end{tikzpicture}
\caption{Possible limiting set $\Omega_h$.}
\label{fig:graphcase}
\end{figure}
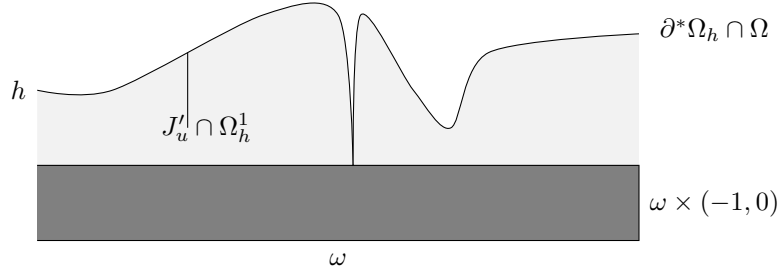

%
We close this section with a result for an alternative  \OOO setting  \EEE where in \eqref{eq: Gfunctional} the second fundamental form is replaced by the mean curvature,  \OOO again in the case $d=3$, $q=2$. \EEE 

\begin{corollary}[Mean curvature regularization]\label{cor: mean2}
We consider \eqref{eq: Gfunctional} with $|\bm{H}|^2$ in place of $|\bm{A}|^{\AAA 2 \EEE}$ \OOO when $d=3$\AAA, $q=2$\EEE. We suppose that for $\mathcal{G}_\delta$ only functions $h$ are admissible such that $\Gamma_h := \partial \Omega_h \cap \Omega$ satisfies that $\partial \Gamma_h$ is $C^2$ and that \EEE
$$  \int_{  \partial \Gamma_h  }  \kappa_{\AAA h,g\EEE} \, {\rm d} \mathcal{H}^{\OOO 1\EEE}  \le  \lambda_\delta\gamma_\delta^{-1}$$
for some $\lambda_\delta \to 0$ \OOO as $\delta\to 0$, \EEE where $\kappa_{\AAA h,g \EEE}$ denotes the geodesic curvature of $\partial \Gamma_h$. Then, the statements of Proposition~\ref{prop: compi2}  and Theorem~\ref{theorem:convergence2} hold. 
 \end{corollary}

The next subsections are devoted to the proofs announced in this section. As the proofs for both applications are similar, we proceed simultaneously. We  \EEE first address the compactness statements in Subsection \ref{sec: compre}, and afterwards the $\Gamma$-convergence results in Subsection  \ref{sec: compre2}.

\subsection{Compactness results}\label{sec: compre}

We start   with the proof of Proposition \ref{prop: compi1}. Afterwards, we present the small adaptions necessary for the proof of Proposition \ref{prop: compi2}.

\begin{proof}[Proof of Proposition \ref{prop: compi1}]
Consider a sequence $(u_\delta, E_\delta)_{\delta>0}$ with $ \mathcal{F}_\delta(u_\delta, E_\delta) \le \QQQ M<+\infty \EEE $ for all $\delta>0$, \QQQ where $M>0$ is defined in the statement of the proposition. \EEE
 \ Hence, $\partial E_\delta\cap \partial_D\Omega=\emptyset$ and, as \EEE  $\min_{\mathbb{S}^{d-1}}\varphi>0$, it holds that   $\sup_{\delta>0} \mathcal{H}^{d-1}(\partial E_\delta) < +\infty$. Thus,   a compactness result for sets of finite perimeter (see \cite[Theorem 3.39]{Ambrosio-Fusco-Pallara:2000}) implies that   there exists a set of finite perimeter $E \subset \Omega$ with $\mathcal{H}^{d-1}(\partial^* E)<+\infty$  such that $\chi_{E_\delta} \to \chi_E$ in $L^1(\Omega)$, up to a subsequence (not relabeled).\QQQ This shows the last part of \eqref{eq:lsc0}(v). \EEE

We now proceed with the compactness for the deformations.  We start by introducing sets for a suitable formulation of the Dirichlet boundary conditions: choose an open set $V \supset \Omega$ such that $V$ and $V \setminus \overline{\Omega}$ are Lipschitz sets and $V \cap \partial \Omega = \partial_D \Omega$. Our goal is to apply Theorem \ref{prop:rigidity} in the version of Corollary \ref{cor: rig-cor} for $U := \Omega$ and \BBB $U_D := V \setminus {\Omega}$. \EEE To this end, we introduce the functions $\hat{y}_\delta$ by
\begin{align}\label{eq: haty}
\hat{y}_\delta = \begin{cases} {\rm id} + \delta (u_\delta - u_0)  & \text{ on $U= \Omega$},\\
{\rm id}& \text{ on $U_D =V\setminus \Omega$}\,.
\end{cases} 
\end{align} 
Note that $\hat{y}_\delta$ are  Sobolev functions \OOO when \EEE restricted to $V \setminus \QQQ\overline{E_\delta}$ since $V \cap \partial \Omega = \partial_D\Omega$ and  $\mathrm{tr}(\hat{y}_\delta)  =\mathrm{tr}(y^\delta_0 - \delta u_0) = {\rm id}$ on $\QQQ\partial_D\Omega\EEE
$,  by the fact that \EEE  $\mathcal{F}_\delta(u_\delta,E_\delta)<+\infty$.  Then, by the triangle inequality, \eqref{eq: nonlinear energy}, \EEE and \OOO the fact that  \EEE $ \mathcal{F}_\delta(u_\delta, E_\delta) \le\QQQ M\EEE$, we get that
\begin{align}\label{eq: endibound}
\int_{V \setminus \overline{E}} \dist^2\big(\nabla \hat{y}_{\delta\EEE} (x), SO(d)\big) \, {\rm d}x \le C'\delta^2
\end{align}
for a constant $C'>0$ \QQQ depending on $M$ and \EEE also on $u_0$. We want to \EEE apply Theorem \ref{prop:rigidity} on $(\hat{y}_\delta,E_\delta)$. To this end,  \EEE  in view of \OOO the fact that \EEE  $\gamma_\delta \to 0$ as $\delta \to 0$, \BBB see \eqref{eq:CdeltaRate-appli}, \EEE and  the definition of $\kappa_\delta$ in \eqref{eq: kappa-pro}, by a suitable diagonal argument we can find \OOO a sequence \EEE $\OOO(\EEE\eta_\delta\OOO)\EEE_{\OOO \delta>0\EEE}$ with $\eta_\delta \to 0$ and  \BBB smooth \EEE sets $\BBB \tilde{\Omega}_\delta \EEE \subset \subset V$ such that,  \OOO as $\delta\to 0$, \EEE
\begin{align}\label{eq: dia-prep}
{\rm (i)} & \ \  C_{\eta_\delta}    \kappa_\delta^{-2}  \gamma_\delta^{-2d/q}  \to 0, \quad \quad  {\rm (ii)} \ \  \BBB \sup\nolimits_{\delta >0} \EEE C_{\eta_\delta}    \delta^{1/3} <+\infty \,,
\end{align}
\begin{align}\label{eq: dia-prep2}
 {\rm (i)}&  \ \  \mathcal{L}^d(V \setminus \tilde{\Omega}_\delta) \to 0, \quad \quad  {\rm (ii)} \ \  \sup\nolimits_{\delta\OOO>0\EEE} \mathcal{H}^{d-1}(\partial \tilde{\Omega}_\delta)<+\infty\,,
\end{align}
\BBB where  $C_{\eta_\delta}$ is the constant in \eqref{eq: main rigitity}. \EEE 
We then apply  \EEE Theorem \ref{prop:rigidity} for  $\eta_\delta$ \BBB and $\gamma_\delta$,  for $V$ in place of $\Omega$, and for \EEE $\tilde{\Omega}_\delta$ \QQQ in place of $\tilde \Omega$\EEE. We use the notation $\mathcal{F}_{\rm surf}^{\varphi,\gamma_\delta,q}$ introduced in \eqref{eq: surface energy}. \QQQ Since $E_\delta\subset \Omega$, $V\cap \partial \Omega=\partial_D\Omega$ and  $\overline{E_\delta}\cap \partial_D\Omega=\emptyset$, we have $E_\delta\in \mathcal{A}_{\rm{reg}}(V)$ and $\mathcal{F}_{\rm surf}^{\varphi,\gamma_\delta,q}(E_\delta, V)= \mathcal{F}_{\rm surf}^{\varphi,\gamma_\delta,q}(E_\delta) \le C'$ \EEE for every $\delta>0$. \EEE Now, by applying \eqref{eq: partition}--\eqref{eq: main rigitity} and using that \EEE $\gamma_\delta \to 0$, $\eta_\delta \to 0$ as $\delta\to 0$, we get that  there  exist sets $\OOO( \EEE E^*_{\delta}\OOO)\EEE_{\OOO \delta>0\EEE}$ with \EEE $E_\delta  \subset E^*_{\delta} \subset \BBB V\EEE$, $\partial E^*_{\delta}\cap \AAA V\EEE$ is a union of finitely many regular submanifolds \EEE for every $\delta>0$,  \EEE and 
\begin{align}\label{eq: partition-appli}
{\rm (i)}  \ \  \lim_{\delta \to 0} \mathcal{L}^d(E^*_{\delta}\setminus E_\delta) = 0\,,  \quad \quad \quad   {\rm (ii)}  \ \liminf_{\delta \to 0 }   \int_{\BBB \partial E^*_{\delta} \cap V} \varphi(\nu_{E^*_{\delta}}) \, {\rm d}\mathcal{H}^{d-1}  \leq  \liminf_{\delta \to 0 } \mathcal{F}_{\rm surf}^{\varphi,\gamma_\delta,q}(E_\delta)\,,
\end{align}
such that for the finitely many connected components of $\tilde{\Omega}_\delta  \setminus E^*_{\delta}$, denoted by  $(\OOO \tilde{\Omega}\EEE^{\eta_{\OOO \delta\EEE},\gamma_\delta}_j)_j$, there exist  corresponding rotations $(R^{\eta_{\OOO \delta\EEE},\gamma_\delta}_j)_j \subset SO(d)$  such that  by \eqref{eq: endibound} \EEE
\begin{align}\label{eq: main rigitity-appli}
\begin{split}
{\rm (i)} & \ \  \sum\nolimits_j \int_{\OOO \tilde{\Omega}\EEE^{\eta_{\OOO \delta\EEE},\gamma_\delta}_j}\big|{\rm sym}\big((R^{\eta_{\OOO \delta\EEE},\gamma_\delta}_j)^T \nabla \hat{y}_\delta-\mathrm{Id}\big)\big|^2\,\mathrm{d}x \leq C_0C'\delta^2,
\\
{\rm (ii)} & \ \   \sum\nolimits_{j}\int_{\OOO \tilde{\Omega}\EEE^{\eta_{\OOO\delta\EEE},\gamma_\delta}_j} \big|(R^{\eta_{\OOO\delta\EEE},\gamma_\delta}_j)^T \nabla \hat{y}_\delta-\mathrm{Id}\big|^2\,\mathrm{d}x \leq C'\EEE C_{\eta_\delta}  \gamma_\delta^{-2d/q} \delta^2\,.
\end{split}
\end{align}
In fact, for \eqref{eq: main rigitity-appli}(i) we used that $C_{\eta_\delta} \AAA\gamma_\delta^{-5d/q}\delta^2\EEE = (\delta^{-q/3d} \gamma_\delta)^{-5d/q}\OOO\cdot \EEE C_{\eta_\delta}\delta^{1/3}  \to 0$ by \eqref{eq:CdeltaRate-appli} and \eqref{eq: dia-prep}(ii).  \EEE   
In view of Corollary \ref{cor: rig-cor} \BBB and \eqref{eq: haty}, \OOO we can choose  \EEE $R^{\eta_{\OOO\delta\EEE},\gamma_\delta}_j = {\rm Id}$ whenever we have $\mathcal{L}^d(U_D \cap \OOO \tilde{\Omega}\EEE^{\eta_{\OOO\delta\EEE},\gamma_\delta}_j)>0$. We denote the union of the components with this property by $\Omega_\delta^{\rm good}$.  Note that 
\begin{align}\label{eq: for bdy}
\Omega_\delta^{\rm good} \supset (U_D \cap \tilde{\Omega}_\delta  )\setminus E^*_\delta\,.
\end{align} 
We introduce the   \OOO  mappings  \EEE $\OOO( \EEE v_\delta\OOO)\EEE_{\OOO \delta>0\EEE} \in GSBD^2(V)$ by  
\begin{align}\label{eq: v}
v_\delta = \begin{cases} u_\delta & \text{ on $\Omega_\delta^{\rm good} \cap  \OOO \Omega \EEE$},\\
 u_0 & \text{ on $\Omega_\delta^{\rm good} \cap   \BBB (V\setminus {\Omega})\EEE$},\\
  0 & \text{ on $E^*_\delta \cup (V \setminus \tilde{\Omega}_\delta)$},\\
 \frac{1}{\delta} e_1 & \text{ on $\tilde{\Omega}_\delta \setminus (\Omega_\delta^{\rm good} \cup E^*_\delta)$}\,,
 \end{cases}
 \end{align}
 where $e_1$ denotes the first coordinate vector, see Figure~\ref{fig:defvdelta} for the different regions in the definition of $v_\delta$\EEE. By \eqref{eq: haty}, \eqref{eq: main rigitity-appli}, \eqref{eq: v}, the definition of $\Omega_\delta^{\rm good}$ and the triangle inequality, we find \OOO  for all $\delta>0$ that 
 \begin{align}\label{eq: main rigitity-appli2}
{\rm (i)}  \ \  \Vert e(v_\delta) \Vert^2_{L^2(V)}   \le C',\quad \quad \quad {\rm (ii)}  \ \   \Vert \nabla v_\delta \Vert^2_{L^2(V)} \le \OOO C' \EEE C_{\eta_\delta}  \gamma_\delta^{-2d/q}\,, 
\end{align}  
where $C'$ depends \OOO additionally  \EEE on $u_0$.
\begin{figure}
\begin{tikzpicture}

\draw[densely dotted] plot [smooth cycle] coordinates {(2,.5)(5,.5)(6,2.25)(2.75,2.5)};

\draw(6.5,2.3) node {$V\setminus \Omega$};

\draw[fill=white] plot [smooth cycle] coordinates {(0,0)(2,.25)(4,0)(3.5,2)(2,2.3)(1,1.8)(1,3.6)(0,3.2)(.25,1.5)};

\begin{scope}

\clip plot [smooth cycle] coordinates {(.5,.5)(2,.75)(3.5,.5)(5,1)(5.5,1.9)(4,2.3)(3,1.9)(2,1.8)(.9,1.3)(.8,3.1)(.3,2.9)(.5,1.5)};
\clip plot [smooth cycle] coordinates {(4.5,0)(2.8,.5)(2,1)(1.8,2.8)(3,5)(10,10)(10,0)};

\draw[fill=gray!30!white](-5,0) rectangle (10,5);

\end{scope}

\draw[ultra thick] plot [smooth] coordinates {(4.06,.32)(3.93,1)(3.5,2)(3,2.25)(2.5,2.305)};


\begin{scope}

\draw[clip] plot [smooth cycle] coordinates {(0,0)(2,.25)(4,0)(3.5,2)(2,2.3)(1,1.8)(1,3.6)(0,3.2)(.25,1.5)};

\draw[line width=1mm,gray] plot [smooth] coordinates {(4.5,0)(2.8,.5)(2,1)(1.8,2.8)};

\draw[line width=1mm,gray] plot [smooth] coordinates {(0,1)(1.5,.75)(2,1)};
\draw[line width=1mm,gray] plot [smooth] coordinates {(1.4,-.2)(1.5,.75)};

\draw[line width=1mm,gray] plot [smooth] coordinates {(0,2.2)(.5,2.8)(2,2.6)};

\end{scope}

\draw[dashed] plot [smooth cycle] coordinates {(.5,.5)(2,.75)(3.5,.5)(5,1)(5.5,1.9)(4,2.3)(3,1.9)(2,1.8)(.9,1.3)(.8,3.1)(.3,2.9)(.5,1.5)};

\draw plot [smooth cycle] coordinates {(0,0)(2,.25)(4,0)(3.5,2)(2,2.3)(1,1.8)(1,3.6)(0,3.2)(.25,1.5)};

\draw(3,1.3) node {$\tilde{\Omega}_\delta^\mathrm{good}$};

\draw(-.125,3.5) node {$\Omega$};
\draw(4.3,1.3) node {$\partial_D \Omega$};

\draw(1.45,1.15) node {$\tilde{\Omega}_\delta$};
\draw(1.4,-.125) node {$E^*_\delta$};

\end{tikzpicture}
\caption{The sets relevant for the definition of $v_\delta$: the thick curve indicates the set $\partial_D\Omega$ and  $E_\delta^*$ is depicted in gray. The region delimited by the dashed curve is $\tilde{\Omega}_\delta$. The region enclosed by the dotted curve is $V\setminus \Omega$. The set $\tilde{\Omega}_\delta^\mathrm{good}$ is depicted in light gray.}
\label{fig:defvdelta}
\end{figure}
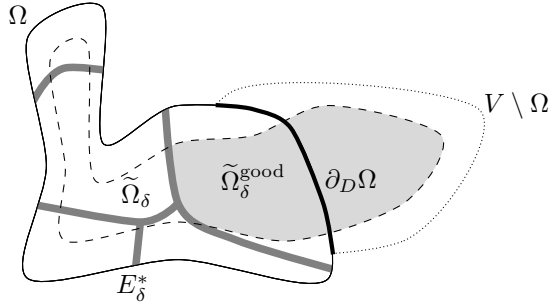
As $J_{v_\delta} \subset \EEE (\partial E^*_\delta \cap V) \cup \partial \tilde{\Omega}_\delta$, \eqref{eq: dia-prep2}(ii), \eqref{eq: partition-appli}, \eqref{eq: main rigitity-appli2}, and the fact that $\min_{\mathbb{S}^{d-1}}\varphi>0$ imply that
\begin{equation*}
\sup\nolimits_{\delta>0} \big(   \Vert e(v_\delta) \Vert^2_{L^2(V)}   + \mathcal{H}^{d-1}(J_{v_\delta}) \big) <+\infty\,.
\end{equation*}
By a compactness result in $GSBD^2$, see Theorem \ref{th: GSDBcompactness}, letting
\begin{align}\label{eq: omega+}
\omega_u := \lbrace x\in V\colon |v_\delta(x)| \to \infty  \text{ as } \delta\to 0\rbrace\,, 
\end{align}
we get that $\omega_u$ is a set of finite perimeter, and we find $v \in GSBD^2(V)$ with $v = 0$ on $\omega_u$ such that \OOO(again up to a subsequence, not relabeled) \EEE $v_\delta$ converges in measure  to $v$ on $V \setminus \omega_u$. (In the language of \cite[Subsection 3.4]{Crismale} we say that $v_\delta \to v$ weakly in $GSBD^2_\infty(V)$.)   Moreover, we note that $v = 0$ \BBB a.e.\   \EEE on $E$ which follows from the convergence in measure, \OOO the fact that  \EEE $\chi_{E_\delta} \to \chi_E$, \eqref{eq: partition-appli}(i), and \eqref{eq: v}.  Thus, $v=0$ \BBB a.e.\ \EEE     on $E\cup\omega_u$.  \BBB We also find that \EEE 
\begin{align}\label{eq: v out}
v = u_0 \quad \text{\AAA a.e.\ \EEE on $U_D =V\setminus \Omega$}
\end{align}
by \eqref{eq: dia-prep2}(i), \eqref{eq: partition-appli}(i), \eqref{eq: for bdy}, \eqref{eq: v}, and the fact that $E\subset \Omega$. \AAA(Here and in the following, set inclusions will be intended in the measure-theoretical sense, i.e., up to sets of $\mathcal{L}^d$-measure zero.) \EEE Therefore, \EEE  we get    $\omega_u \subset U=\Omega$.  We denote the restriction of $v$ to $\Omega$ by $u$, and note that \OOO then \EEE $u = 0$ on $E\cup\omega_u$.  We \OOO also  \EEE observe \OOO that  \EEE 
\begin{align}\label{eq: good cover}
\mathcal{L}^d\big( \Omega \setminus (\omega_u \cup \Omega_\delta^{\rm good} \cup E_\delta ) \big) \to 0 \quad \text{ as $\delta\to 0$}\,.
\end{align}
\BBB In fact,  $\mathcal{L}^d(\Omega \setminus \tilde{\Omega}_\delta) \to 0$  by \eqref{eq: dia-prep2}(i), $ \mathcal{L}^d(E^*_{\delta}\setminus E_\delta) \to 0$ by \eqref{eq: partition-appli}(i), and $\mathcal{L}^d((\tilde{\Omega}_\delta \setminus (\Omega_\delta^{\rm good} \cup E^*_\delta)) \setminus \omega_u) \to 0$ by  \eqref{eq: v} and  \eqref{eq: omega+}.  \EEE
 
 We \OOO now  \EEE show  properties  \eqref{eq:lsc0}. First \OOO of all, \EEE \eqref{eq:lsc0}(iv) follows directly from \eqref{eq: partition-appli}(ii). Since $\mathcal{F}_\delta(u_\delta,E_\delta)<+\infty$ \OOO for all $\delta>0$, \EEE  we have $u_\delta = \OOO\chi_{\Omega \setminus E_\delta} \EEE u_\delta$. Then, using \eqref{eq: v} as well as \eqref{eq: good cover}, we get \OOO that  \EEE  $\mathcal{L}^d( \OOO(\Omega\setminus \omega_u)\cap\EEE\lbrace v_\delta \neq u_\delta  \rbrace ) \to 0$ as $\delta \to 0$ and thus $u_\delta \to v=u$ in measure on  $\Omega \setminus \omega_u$. This shows   \eqref{eq:lsc0}(i). To see \eqref{eq:lsc0}(iii), we again use \OOO that  \EEE $\mathcal{L}^d(\OOO(\Omega\setminus \omega_u)\cap\EEE\lbrace v_\delta \neq u_\delta  \rbrace) \to 0$ as $\delta \to 0$,  as well as \eqref{eq: dia-prep}(i) and \eqref{eq: main rigitity-appli2}(ii)   to calculate
\begin{align*}
\limsup_{\delta \to 0 }
\mathcal{L}^d\big(\lbrace |\nabla u_\delta|>\kappa_\delta \rbrace \setminus \omega_u \big) 
& \le \limsup_{\delta \to 0 }
\mathcal{L}^d\big(\lbrace |\nabla v_\delta|>\kappa_\delta \rbrace \setminus \omega_u \big)\OOO\le  \limsup_{\delta \to 0 } \kappa_\delta^{-2} \int_{V}|\nabla v_\delta|^2\, dx \EEE \\
& \le \OOO C'\limsup_{\delta \to 0} \, \EEE C_{\eta_\delta}\kappa_\delta^{-2}  \gamma_\delta^{-2d/q} =0\,.
\end{align*} 
It therefore remains  \BBB to define the sets $(\omega_u^\delta)_{\delta>0} \subset \mathfrak{M}(\Omega)$ and \QQQ to prove  \eqref{eq:lsc0}(ii),(v)\EEE. Let $V' \subset \subset V$. Since $\chi_{E_\delta} \to \chi_{E}$ in $L^1(\Omega)$, \eqref{eq: partition-appli}\OOO(i) \EEE also implies \OOO that \EEE $\chi_{E^*_\delta} \to \chi_{E}$ in $L^1(V)$. Moreover, for $\delta>0$ small depending  \AAA also \EEE on $V'$\EEE, we have $v_\delta =0$ on $E^*_\delta$ and $v_\delta|_{V'\setminus \overline{E^*_\delta}} \in H^1(V'\setminus \overline{E^*_\delta};\R^d)$, see \eqref{eq: dia-prep2}(i) and  \eqref{eq: v}. This along with the fact that  $\OOO( \EEE v_\delta\OOO)_{\delta>0}\EEE$ converges weakly to  \OOO $v$ \EEE in $GSBD^2_\infty(V)$, \OOO means that  \EEE we can apply \cite[Theorem 5.1]{Crismale} on the set $V'$ for $\OOO( \EEE v_\delta\OOO)_{\delta>0}\EEE$ and $\OOO( \EEE E^{*}_\delta\OOO)_{\delta>0}\EEE$ to find
\begin{equation}\label{1405191303}
\OOO \chi_{V' \setminus (E^*_\delta \cup \omega_u)} \EEE e(v_\delta) \rightharpoonup \OOO \chi_{V' \setminus (E \cup \omega_u)} \EEE e(v)  \quad\text{  weakly  in }L^2(V';\R^{d\times d}_{\rm sym})\,,
\end{equation}
as well as 
\begin{equation}\label{1405191304}
\int_{J_v \cap E^0 \cap V'} 2\varphi(\nu_v)\,{\rm d} \mathcal{H}^{d-1}+\int_{\OOO\partial^* E \EEE \cap V'} \varphi(\nu_E) \, {\rm d}\mathcal{H}^{d-1} \leq \liminf_{\delta \to 0} \int_{\partial E^*_\delta\EEE\cap V'} \varphi(\nu_{E^*_\delta}) \, {\rm d}\mathcal{H}^{d-1}\,,
\end{equation}
where  $E^0$ denotes the set of points with  density zero for $E$ \OOO and $\nu_v$ is \BBB a \OOO measure-theoretical unit normal to  $J_v$. \EEE Define $\omega_u^\delta := \omega_u \cup (\OOO (\Omega\cap \EEE\tilde{\Omega}_\delta) \setminus (\Omega_\delta^{\rm good} \cup E_\delta^*))$. Note that $\mathcal{L}^d(\OOO\omega^{\delta}_{u}\BBB \triangle \EEE  \omega_u\EEE) \to 0$ by \eqref{eq: partition-appli}(i) and \EEE \eqref{eq: good cover}, \QQQ which finishes the verification of \eqref{eq:lsc0}(v), and that $$\mathcal{H}^{d-1}(\partial^*\omega_u) + \sup_{\delta>0}\mathcal{H}^{d-1}(\partial^*\omega_u^\delta) \le C_M$$ by \eqref{eq: dia-prep2}(ii),  \eqref{eq: partition-appli}(ii), and Theorem~\ref{th: GSDBcompactness}, for a $C_M>0$  depending on $M:=\sup\nolimits_{\delta>0} \mathcal{F}_\delta(u_\delta, E_\delta)$. \EEE 
\EEE As  $e(v_\delta) = 0$ on  $\Omega \setminus \Omega_\delta^{\rm good}$ and $u_\delta = v_\delta$ on $\Omega\cap\Omega_\delta^{\rm good}$, see \eqref{eq: v}, by  recalling \eqref{eq: dia-prep2}(i), for $\delta>0$ small enough  we get \OOO  a.e.\ on $V$ that \EEE
$${\OOO\chi_{(V' \cap\Omega) \setminus (E^*_\delta \cup \omega_u)} \EEE e(v_\delta) = \OOO \chi_{(V' \cap \Omega_\delta^{\rm good}\cap \Omega) \setminus (E^*_\delta \cup \omega_u)} \EEE e(u_\delta) = \OOO\chi_{(V' \cap \Omega) \setminus (E^*_\delta \cup \omega^\delta_u)}\, \EEE e(u_\delta).}$$
By using \eqref{1405191303} and recalling \OOO that by definition \EEE $u=v$ on $\Omega$, we obtain \eqref{eq:lsc0}(ii) as $V' \subset \subset V$ was arbitrary. This concludes the proof \BBB of  \eqref{eq:lsc0}(ii).  \EEE  For later purposes, we also directly discuss the implications of the estimate \eqref{1405191304} in the subsequent remark. 
\end{proof}

\begin{remark}\label{remark: lower bound}
{\normalfont
In the setting of the previous result, we also have 
\begin{align}\label{eq:lsc}
 \int_{J_u \setminus \partial^*E} 2\varphi(\nu_u)\,{\rm d}\mathcal{H}^{d-1} + \int_{\OOO \partial^* E\EEE\cap\Omega} \varphi(\nu_E) \,{\rm d}\mathcal{H}^{d-1} + F^{\rm bd\OOO r \EEE y}(u,E) \leq \liminf_{\delta \to 0} F_{\rm surf}^{\varphi,\gamma_\delta,q}(E_\delta)\,,
\end{align}
where $F^{\rm bd\OOO r \EEE y}$ is defined in \eqref{eq: bdaypart} and  $F_{\rm surf}^{\varphi,\gamma_\delta,q}$ in \eqref{eq: surface energy}. Indeed, note that $J_u \cap E^0 = J_u \setminus \partial^*E$ since $u = 0$ on $E$. Then, by the fact that \EEE $u=v$ on $\Omega$, $E \subset \Omega$, \BBB $V \cap \partial \Omega =\partial_D\Omega$, and \eqref{eq: v out}, \EEE  we observe 
\begin{align*}
\int\limits_{J_u \setminus \partial^*E} 2\varphi(\nu_u)\,{\rm d}\mathcal{H}^{d-1} + \hspace{-0.15cm}\int\limits_{\OOO\partial^* E\EEE\cap\Omega} \varphi(\nu_E) \,{\rm d}\mathcal{H}^{d-1} + F^{\rm bd\OOO r \EEE y}(u,E)  = \hspace{-0.15cm}\int\limits_{J_v \cap E^0}   \hspace{-0.15cm} 2\varphi(\nu_v)\, {\rm d}\EEE \mathcal{H}^{d-1} + \hspace{-0.15cm} \int\limits_{\OOO \partial^* E\cap \EEE V} \hspace{-0.15cm} \varphi(\nu_E) \, {\rm d}\EEE \mathcal{H}^{d-1}\,.
\end{align*}
Then, in view of \eqref{eq: partition-appli}(ii) and  \eqref{1405191304} for a sequence $\OOO( \EEE V_n\OOO)\EEE_{\OOO n\in \mathbb{N}\EEE} \subset \subset V$ with $\mathcal{L}^d(V \setminus V_n) \to 0$ \OOO as $n\to \infty$ \EEE  we get \eqref{eq:lsc}. }
\end{remark}

\QQQ
\begin{remark}\label{general_boundary_conditions}
{\normalfont
A closer inspection of the previous proof reveals that the compactness result in Proposition \ref{prop: compi1} remains valid, even if we impose \textit{thickened boundary conditions} for the sequence $(u_\delta)_{\delta>0}$, for instance in the following way. 

As before, let $u_0 \in W^{1,\infty}(\R^d;\R^d)$, $d=2,3$. Similarly to the argument in the previous proof, we introduce an open set $V\supset \Omega$ such that $V$ and $V \setminus \overline{\Omega}$ are Lipschitz sets and $V \cap \partial \Omega = \partial_D \Omega$. For $\delta >0$ define $y_0^\delta := {\rm id} + \delta u_{0,\delta}$, where $(u_{0,\delta})_{\delta>0}\subset W^{1,\infty}(\R^d;\R^d)$ is such that $u_{0,\delta}\to u_0$ locally uniformly in $\R^d$ as $\delta\to 0$. Consider also a sequence of open Lipschitz sets $(V_\delta)_{\delta>0}$, with $V_\delta\subset\subset V\setminus\overline{\Omega}$ such that $\chi_{V_\delta}\to \chi_{V\setminus \Omega}$ locally uniformly. We let again $F_\delta(y,E)$ be defined by \eqref{eq: F functional} if $E \in\mathcal{A}_{\rm reg}(\Omega)$, $\overline{E}\cap\partial_D\Omega=\emptyset$, $y|_{V\setminus \overline E} \in H^1( V \setminus \overline E;\R^d)$, $y|_E={\rm id}$, and now $y|_{V_\delta} =y^\delta_0|_{V_\delta}$, and $F_\delta(y,E) = +\infty$, otherwise. Then the conclusion of Proposition \ref{prop: compi1} still holds. 
}
\end{remark}
\EEE

\begin{proof}[Proof of Proposition \ref{prop: compi2}]
Consider  $(u_\delta, h_\delta)_{\delta\OOO>0\EEE}$ with $\QQQ K:=\EEE\sup_{\delta\OOO>0\EEE} \mathcal{G}_\delta(u_\delta, h_\delta) < +\infty$. First, by \BBB this energy bound, \EEE \eqref{eq: Gfunctional}, and a standard compactness argument, we find $h \in BV(\omega;  [0,M]  )$  such that $h_\delta \to h$ in $L^1(\omega)$, up to a subsequence (not relabeled). For the compactness of $(u_\delta)_{\delta\OOO>0\EEE}$, we proceed as in the proof of Proposition~\ref{prop: compi1}, applied for  $V:= \omega \times (-2,M+1)$, i.e., $U_D:= \omega \times (-2,-1\BBB ] \EEE$, and $E_\delta := \Omega\setminus \overline{\Omega_{h_\delta}}$. The only point  to prove is that $\omega_u$ given in \eqref{eq: omega+}   satisfies  $\mathcal{L}^d(\omega_u) = 0$. In fact, then \eqref{eq:lsc0XXXX} for a limit $u \in GSBD^2(\Omega)$ follows from \eqref{eq:lsc0}. Eventually, since    $u_\delta = \OOO\chi_{\Omega_{h_\delta}} \EEE u_\delta $  \EEE and $u_\delta = u_0$ on $\omega \times (-1,0\OOO]\EEE$, by \OOO the fact that \EEE $\mathcal{G}_\delta(u_\delta, h_\delta) < +\infty$ (see \eqref{eq: Gfunctional}),  \eqref{eq:lsc0XXXX}(i) shows  $u=\OOO \chi_{\Omega_h} \EEE u$ and  $u = u_0$ on $\omega\times (-1,0\OOO ]\EEE$.

Let us now check that $\mathcal{L}^d(\omega_u) =  0$. To this end, we apply Corollary \ref{cor: rig-cor} once again, now in the version for graphs, see Corollary~\ref{cor: graphi}. We denote the corresponding set \BBB $E'_{\eta_\delta,\gamma_\delta}$ by $E'_\delta$ for simplicity and we  let $h'_\delta \colon \omega \to \R$  be \EEE such that $\Omega_{h'_\delta} = \Omega \setminus \BBB \overline{E'_\delta}$. We note that \BBB $E'_\delta \supset E_\delta^*$ and thus \EEE $h'_\delta \leq \EEE h_\delta$. This \BBB along with \EEE \eqref{eq: graphiii}(i) implies $h'_\delta \to h$ in $L^1(\omega)$ \OOO since $\eta_\delta, \EEE\gamma_\delta\to 0$. \EEE In view of \eqref{eq: main rigitity-appli2},  \EEE \eqref{eq: graphiii}(ii) applied for $\varphi\equiv 1$\EEE, and \OOO the fact that \EEE $E'_\delta \supset E_\delta^*$, we get $\NNN v\EEE_\delta|_{\Omega_{h'_\delta}} \in H^1(\Omega_{h'_\delta};\R^d)$ and \EEE
$$\sup_{\delta\OOO>0\EEE} \AAA \Big( \EEE \int_{\Omega_{h'_\delta}^+} |e(\NNN v\EEE_\delta)|^2 \, {\rm d}x + \int_{\omega} \sqrt{1+|\nabla h'_{\delta}|^2}  \, {\rm d}x'\AAA \Big)\EEE < +  \infty\,.$$   
Therefore, by \cite[Theorem 2.5]{Crismale} we find \OOO that \EEE $u = \OOO \chi_{\Omega_h} \EEE u \in GSBD^2(\Omega)$ \OOO is \EEE such that $\OOO \chi_{\Omega_{h'_\delta}} \EEE \NNN v\EEE_\delta \to u$ in measure. (Indeed, $u$ coincides with the limiting function identified above.) As $\NNN v\EEE_\delta = 0$ on $E_\delta$ and $\mathcal{L}^d(E'_{\delta} \setminus E_\delta) \to 0$ by \eqref{eq: graphiii}(i),  we conclude that  $\omega_u$ defined in \eqref{eq: omega+} satisfies  $\mathcal{L}^d(\omega_u) = 0$. 
\end{proof}

\subsection{Derivation of effective linearized limits by $\Gamma$-convergence}\label{sec: compre2}

We start \BBB with \EEE two results on the linearization of nonlinear elastic energies which are by now classical, see e.g.~\cite{alicandro.dalmaso.lazzaroni.palombaro, Braides-Solci-Vitali:07, DalMasoNegriPercivale:02, MFMK, Friedrich-Schmidt:15, NegriToader:2013, Schmidt:08, Schmidt:2009}. For completeness, however, we include short proofs, in particular due to the fact that our setting, involving \OOO varying \EEE sets $(E_\delta)_{\delta\OOO>0\EEE}$, is slightly different \BBB compared to \EEE the above mentioned works\EEE.  Recall the quadratic form  $\mathcal{Q}$ defined in \eqref{eq: quadratic form}.  \EEE 
\begin{lemma}\label{prop:liminfel}
Let $(u_\delta)_{\delta\OOO>0\EEE} \subset GSBD^2(\Omega)$,  $u\in GSBD^2(\Omega)$, and let $(\Theta_\delta)_{\delta\OOO>0\EEE}, \Theta \in \M(\Omega)$ \OOO be  \EEE such that $\OOO \chi_{\Theta_\delta} \EEE e(u_\delta) \rightharpoonup \OOO \chi_{\Theta} \EEE e(u)$ weakly in $L^2(\Omega,\R^{d\times d}_{\rm sym})$, $
\mathcal{L}^d(\lbrace |\nabla u_\delta|>\kappa_\delta \rbrace  \cap \Theta)  \to 0$, and $u = 0$ on $\Omega \setminus \Theta$, where     $\kappa_\delta$ is defined  in \eqref{eq: kappa-pro}. Then, 
 \begin{align*}
\liminf_{\delta \to 0} \frac{1}{\delta^2} \int_\Omega W( {\rm Id} + \delta \nabla u_\delta) \, {\rm d}x  \geq \frac{1}{2}\int_{\Omega} \mathcal{Q}(e(u)) \,\mathrm{d}x\,.
\end{align*}
\end{lemma}

\begin{proof} 
We \BBB define \EEE $\vartheta_\delta \in L^\infty(\Omega)$ by $\vartheta_\delta(x) = \chi_{[0,\kappa_\delta]}(|\nabla u_\delta(x)|)$, and note that $
\mathcal{L}^d(\lbrace |\nabla u_\delta|>\kappa_\delta \rbrace  \cap \Theta)  \to 0$ implies       $\vartheta_\delta \to 1$ boundedly in measure on $\Theta$, \OOO as $\delta\to 0$. \EEE By the regularity \OOO and the structural hypotheses \EEE of $W$  we get   $W({\rm Id} + F) = \frac{1}{2}\mathcal{Q}( {\rm sym}(F)) + \Phi(F)$, \BBB where \EEE $\Phi\colon\R^{d \times d}\to  \R $ is a function  satisfying $|\Phi(F)|  \le C|F|^3$  for all $F \in \R^{d\times d}$ with $|F| \le 1$. Then, \EEE \OOO the fact that \eqref{eq: kappa-pro} implies  $\delta\kappa_\delta\to 0$ and hence $0<\delta\kappa_\delta\leq 1$ for $\delta>0$ sufficiently small, together with  the fact that \EEE $W \ge 0$, imply that \EEE
\begin{align*}
\liminf_{\delta \to 0} \frac{1}{\delta^2}\int_{\Omega} W({\rm Id} + \delta \nabla u_\delta)\OOO  \, {\rm d}x & \ge \liminf_{\delta\to 0}  \frac{1}{\delta^2}\int_{\Theta} \vartheta_\delta W({\rm Id} + \delta \nabla u_\delta) \, {\rm d}x
\\
&=  \liminf_{\delta\to 0} \int_{\Theta} \vartheta_\delta \Big( \frac{1}{2}\mathcal{Q}(e(u_\delta)) + \frac{1}{\delta^2} \Phi(\delta \nabla u_\delta)  \Big) \, {\rm d}x 
\\
&\ge  \liminf_{\delta\to 0} \Big(\OOO\frac{1}{2}\EEE\int_{\Theta}  \vartheta_\delta   Q(\OOO \chi_{\Theta_\delta} \EEE e(u_\delta)) \,{\rm d}x \EEE - C\int_{\Theta}   \vartheta_\delta \delta|\nabla u_\delta|^3\Big)\,. \EEE
\end{align*}
The second term converges to zero \BBB since \EEE  $\vartheta_\delta\OOO\delta \EEE |\nabla u_\delta|^3$ is uniformly controlled \AAA from above \EEE by $\OOO\delta\EEE\kappa_\delta^3$, where  $\OOO\delta\EEE\kappa_\delta^3\to 0$ by \eqref{eq: kappa-pro}.   As $\OOO\chi_{\Theta_\delta} \EEE e(u_\delta) \rightharpoonup \OOO \chi_{\Theta} \EEE e(u)$ weakly in $L^2(\Omega,\R^{d\times d}_{\rm sym})$, by the convexity of $\mathcal{Q}$, and the fact that $\vartheta_\delta$ converges to $1$ boundedly in measure on $\Theta$, we conclude that 
\begin{align*}
\liminf_{\delta \to 0} \frac{1}{\delta^2}\int_{\Omega} W(\OOO \mathrm{Id}+\delta\nabla u_\delta\EEE)\, {\rm d}x \ge  \OOO\frac{1}{2}\EEE\int_{\Theta} \mathcal{Q}(e(u))\, {\rm d}x =  \int_{\Omega} \frac{1}{2}\mathcal{Q}(e(u))\, {\rm d}x\,,
\end{align*}
where the last step follows from the fact that $u = 0$ on $\Omega \setminus \Theta$. This concludes the proof. 
\end{proof}

\begin{lemma}\label{lemma: Taylor}
Let $(\Theta_\delta)_{\delta\OOO>0\EEE}$  be a sequence of  open subsets of $\Omega$  and let $\OOO( \EEE u_\delta\OOO)\EEE_{\OOO\delta>0\EEE} \in H^1\EEE (\Theta_\delta;\R^d)$ be  such that 
\begin{align}\label{eq: dense-in-appli}
\Vert \nabla u_\delta \Vert_{L^\infty(\Theta_\delta)}   \le \delta^{-1/4}\,.
 \end{align}
 Then, as $\delta \to 0$, we have for $y_\delta := {\rm id} + \delta u_\delta$ that 
$$ \lim_{\delta \to 0} \Big|\frac{1}{\delta^2} \int_{\Theta_\delta} W(\nabla y_\delta) \, {\rm d}x - \OOO\frac{1}{2} \EEE \int_{\Theta_\delta} \mathcal{Q}( e(u_\delta) ) \, {\rm d}x \Big| = 0\,. $$\\[-20pt]  
 \end{lemma}
\begin{proof}
As in the previous proof, we use that $W({\rm Id} + F) =  \frac{1}{2}\mathcal{Q}( {\rm sym}(F)) + \Phi(F)$ with $|\Phi(F)|\le C|F|^3$ for $|F| \le 1$. Then, for $y_\delta = {\rm id} + \delta u_\delta$, we compute
\begin{align*} 
\frac{1}{\delta^2}\int_{\Theta_\delta} W(\nabla y_\delta) \, {\rm d}x & = \frac{1}{\delta^2} \int_{\Theta_\delta} W({\rm Id} +  \delta  \nabla u_\delta)\, {\rm d}x  =   \int_{\Theta_\delta}  \Big( \frac{1}{2} \mathcal{Q}(e(u_\delta)) + \frac{1}{\delta^2} \Phi(\delta \nabla u_\delta)  \Big)\,{\rm d}x  \\ & = \frac{1}{2}  \int_{\Theta_\delta}   \mathcal{Q}(e(u_\delta)) \, {\rm d}x  +    \int_{\Theta_\delta}    {\rm O}\big( \delta|\nabla u_\delta|^3 \big)\,.
\end{align*}
The result now follows by taking limits  and using \eqref{eq: dense-in-appli}. 
\end{proof}

We now proceed with the $\Gamma$-convergence results. The proofs essentially rely on the above preparations, the estimates in Subsection \ref{sec: compre}, and  the results in the linearized setting obtained in \cite[Section 2]{Crismale}.  We start with Theorem \ref{theorem:convergence}.

\begin{proof}[Proof of Theorem \ref{theorem:convergence}]
We first address the lower bound and afterwards the upper bound.\\
\begin{step}{1}(Lower bound) 
Suppose that $(u_\delta,E_\delta) \overset{\tau}{\to} (u,E)$, i.e., there exist a set of finite perimeter  $\omega_u \in \M(\Omega)$ such that  $\chi_{E_\delta} \to \chi_E \text{ in }L^1(\Omega)$, $u_\delta \to u$ in measure on $\Omega \setminus \omega_u$, and $u\equiv 0$ on $E \cup \omega_u$.  Without restriction, we can assume that $\sup_{\delta\OOO>0\EEE} \mathcal{F}_\delta(u_\delta, \EEE E_\delta)< +\infty$. In view of Proposition \ref{prop: compi1}, this yields $u = \OOO \chi_{\Omega \setminus E} \EEE u \in GSBD^2(\Omega)$, $\mathcal{H}^{d-1}(\partial^* E)<+\infty$, and that \eqref{eq:lsc0} holds. Therefore, \OOO we obtain \EEE $\mathcal{F}_0(u,E)<+\infty$.  Now, the lower bound for the surface energy follows directly from  Remark~\ref{remark: lower bound}. For the elastic part, we use \eqref{eq:lsc0}(ii),(iii) and apply Lemma \ref{prop:liminfel} \AAA for an  arbitrary $\Omega' \subset \subset \Omega$, \EEE for  $\Theta_\delta = \Omega' \EEE \setminus (E_\delta^* \cup \omega^\delta_u) $ and $\Theta =\Omega' \EEE \setminus (E \cup \omega_u)$.  
\end{step}\\[2pt]
\begin{step}{2}(Recovery sequence) 
By \cite[Theorem 2.2]{Crismale}, for each $E\in\M(\Omega)$ with $\mathcal{H}^{d-1}(\partial^* E)<+\infty$ and each  $u=\OOO\chi_{\Omega \setminus E} \EEE u \in GSBD^2(\Omega)$, there exists a sequence of \AAA sets  $\OOO( \EEE E_\delta\OOO)\EEE_{\OOO\delta>0\EEE}$ with $\QQQ E_\delta\subset\subset\Omega$, $\partial E_\delta
\in C^{\infty}$, \EEE $\chi_{E_\delta} \to \chi_E$ in $L^1(\Omega)$ and a sequence  $\AAA(u_\delta)_{\delta>0}\EEE$ with $u_\delta|_{\Omega\setminus \overline{E_\delta}} \in H^1(\Omega \setminus \overline{E_\delta};\R^d)$, $u_\delta|_{E_\delta} = 0$, and ${\rm tr}(u_\delta) = {\rm tr}(u_0)$ on $\QQQ \partial_D \Omega\EEE
$ such that $u_\delta \to u$ in $L^0(\Omega;\R^d)$ and
\begin{align}\label{eq: limsupVito1}
\lim_{\delta \to 0} \Big(\OOO\frac{1}{2} \EEE \int_{\Omega \setminus \OOO\overline {E_\delta}\EEE}  \mathcal{Q}(e(u_\delta))\, {\rm d}x + \int_{\QQQ \partial E_\delta\EEE 
}  \varphi(\nu_{E_\delta}) \, {\rm d}\mathcal{H}^{d-1}  \Big)  = \mathcal{F}_0(u,E)\,.
\end{align}
Strictly speaking, \cite[Theorem 2.2]{Crismale} only ensures that $\partial E_\delta$ is Lipschitz, \BBB see \cite[(2.2)]{Crismale}, \EEE but in the proof it is shown that \QQQ $E_\delta$ can be chosen compactly contained in $\Omega$ and $\partial E_\delta$ \EEE of class $C^\infty$, see \cite[Proposition 5.4]{Crismale}. By a density argument \NNN and the fact that $\Omega$ is Lipschitz, \EEE without relabeling of functions and sets, \EEE it is not restrictive to further assume that each $u_\delta$ is Lipschitz on $\Omega \setminus \overline{E_\delta}$. By a diagonal argument we may further suppose without restriction that 
$${ \Vert \nabla u_\delta \Vert_{L^\infty(\Omega \setminus \OOO\overline{E_\delta}\EEE)}   \le \delta^{-1/4} ,\quad \quad  \quad \int_{\QQQ\partial E_\delta\EEE
} |\bm{A}_\delta|^q \, {\rm d}\mathcal{H}^{d-1} \le \BBB \gamma^{-1/2}_\delta \,, \EEE}
$$
where $\bm{A}_\delta$ denotes the second fundamental form associated to $\QQQ \partial E_\delta\AAA
$. \BBB Here, we use that $\gamma_\delta^{-1/2} \to \infty$, see \eqref{eq:CdeltaRate-appli}. \EEE Then, in view of \eqref{eq: F functional} and  \eqref{eq: limsupVito1},  by applying Lemma \ref{lemma: Taylor} for $\Theta_\delta = \Omega \setminus \OOO \overline{E_\delta}\EEE$ \BBB and by  using again that $\gamma_\delta \to 0$ as $\delta \to 0$, \EEE we conclude that $\BBB \lim_{\delta \to 0} \mathcal{F}_\delta(u_\delta,E_\delta) = \lim_{\delta \to 0} {F}_\delta(y_\delta,E_\delta) \EEE = \mathcal{F}_0(u,E)$, \BBB where $y_\delta = {\rm id} + \delta u_\delta$. \EEE  This concludes the construction of recovery sequences. Eventually, \cite[Theorem 2.2]{Crismale} also shows that a volume constraint can be incorporated, which yields Remark~\ref{rem:voids}(i). 
\end{step}
\end{proof}

We now proceed with the proof of Theorem \ref{theorem:convergence2}. To this end, we recall the notion of $\sigma_{\rm sym}^2$-convergence introduced in \cite[Section 4]{Crismale}, in a slightly simplified version.    In the following, we use the notation $A \tilde{\subset} B$ if $\mathcal{H}^{d-1}(A \setminus B) = 0$ and $A \tilde{=} B$ if  $A \tilde{\subset} B$ and  $B \tilde{\subset} A$.

\begin{definition}[$\sigma^2_{\rm sym}$-convergence]\label{def:spsconv}
Let $ U  \subset \R^d$  be  open,   $U' \supset U$  be  open  with $\mathcal{L}^d(U' \setminus U)>0$. Consider a sequence $(\Gamma_n)_{n\in \mathbb{N}}  \subset  \overline{U}\cap U'$  with $\sup_{n\in \N} \mathcal{H}^{d-1}(\Gamma_n) <+\infty$. We suppose that for each $C>0$ the sets   
\begin{align}\label{eq: XCn}
{\mathcal{X}_{C,n} := \Big\{  v \in GSBD^2(U')\colon \text{\BBB $v = 0$ \EEE in $U' \setminus U$,\, $\Vert e(v) \Vert_{L^2(U')} \le C$,\,  $\BBB J_{v}  \EEE \tilde{\subset} \Gamma_n$} \Big\}}    
\end{align}
are equi-precompact  in $L^0(U';\R^d)$, in the sense that every sequence $(v_n)_{n\in \mathbb{N}\EEE}$ with $v_n  \in \mathcal{X}_{C,n}$ \EEE admits a convergent subsequence in $L^0(U';\R^d)$. Then, we say that $(\Gamma_n)_{n\in \mathbb{N}\EEE}$ $\sigma_{\rm sym}^2$-converges to $\Gamma$ satisfying  $\Gamma \subset \overline{U} \cap U'$ and  $\mathcal{H}^{d-1}(\Gamma) < +\infty$, if there holds:  

(i)  for any $C>0$ and any sequence $\OOO( \EEE v_n\OOO)\EEE_{n\in \mathbb{N}\EEE}$ with $v_n \in  \EEE \mathcal{X}_{C,n}$, \EEE if a subsequence $(v_{n_k})_{k\in \mathbb{N}\EEE}$ converges in measure to $v \in GSBD^2(U')$, then $J_v  \tilde{\subset} \Gamma$. 

(ii) there exists a function $v \in GSBD^2(U')$  and a sequence  $\BBB (v_n)_{n\in \mathbb{N}\EEE}$ with $v_n \in  \mathcal{X}_{C,n}$ \EEE for some $C>0$ such that $v_n \to v$ in measure on $U'$ and  $J_v  \tilde{=} \Gamma$.
\end{definition}

Our definition is simplified compared to  \cite[Section 4]{Crismale} as we assume a compactness property for the sets in \eqref{eq: XCn}. Indeed, all involved sequences converge in measure on $U'$, and therefore we can neglect the set $G_\infty$ appearing in \cite[Definition 4.1]{Crismale}, which is related to the set where a sequence $(v_n)_{n\in \mathbb{N}\EEE}$ as in (i) may converge to infinity. \BBB In a similar fashion, the space $GSBD^2_\infty$ introduced in  \cite[Subsection 3.4]{Crismale} is not needed. \EEE Note that imposing boundary conditions in \eqref{eq: XCn} is fundamental for compactness, by  excluding nonzero constant functions. We refer to  \cite[Section~4]{Crismale} for a more general discussion on this notion and mention here only the fundamental compactness result, see \cite[Theorem 4.2]{Crismale}.

\begin{theorem}[Compactness of $\sigma^2_{\rm sym}$-convergence]\label{thm:compSps}
Let $U \subset \R^d$  be  open,  $U' \supset U$  be open  with $\mathcal{L}^d(U' \setminus U)>0$.  Then, every sequence $(\Gamma_n)_{n\in \mathbb{N}\EEE} \subset \overline{U}\cap U'$ satisfying the assumptions in Definition~\ref{def:spsconv} has a $\sigma_{\rm sym}^2$-convergent subsequence (not relabeled) with limit  $\Gamma$  satisfying the inequality $\mathcal{H}^{d-1}(\Gamma) \leq \liminf_{n\to\infty} \mathcal{H}^{d-1}(\Gamma_n)$.  
\end{theorem}

Moreover, the following lower semicontinuity result can be shown.

\begin{lemma}[Lower semicontinuity  of surfaces]\label{lemma: vertical sets}
Let $\Omega = \omega \times (-1,M+1)$. Let $(D_\delta)_{\delta\OOO>0\EEE}$ be a sequence of Lipschitz sets such that $\Gamma_\delta := \partial D_\delta\cap \Omega$  are $\sigma_{\rm sym}^2$-converging to $\Gamma$ in the sense of Definition \ref{def:spsconv} with respect to the sets    $U = \omega \times (-\frac{1}{2},M) $  and $U' = \Omega$. \BBB Suppose that \EEE there exists \OOO a function  \EEE $h \in BV(\omega;[0,M])$ such that $\mathcal{L}^d((\Omega \setminus D_\delta) \triangle \Omega_h) \to 0$ as $\delta\to 0$. \EEE   Then,  we have 
\begin{align*}
\mathcal{H}^{d-1}(\partial^* \Omega_h \cap \Omega) + 2 \mathcal{H}^{d-1}(\Gamma \cap \Omega_h^1) \leq \liminf_{\delta\to 0}   \mathcal{H}^{d-1}(\Gamma_\delta)\,.
\end{align*}
\end{lemma}

\begin{proof}
For the proof we refer to   \cite[Subsection 6.1]{Crismale}, in particular \BBB to \EEE \cite[(6.4), (6.6)]{Crismale}. Note that there the proof was only performed  in the case that $\partial D_\delta \cap \Omega$ are graphs, but this assumption is not needed since the argument relies on the lower semicontinuity result in \cite[Theorem 5.1]{Crismale}.   
\end{proof}

We are now in \OOO the \EEE position to give the proof of Theorem \ref{theorem:convergence2}.

\begin{proof}[Proof of Theorem \ref{theorem:convergence2}]
We first address the lower bound and afterwards the upper bound.\\
\begin{step}{1}(Lower bound) 
Suppose that $u_\delta \to u$ in $L^0(\Omega;\R^d)$ and that $h_\delta \to h$ in $L^1(\omega)$. Without restriction, we can assume that $\sup_{\delta\OOO>0\EEE} \mathcal{G}_\delta(u_\delta, \EEE h_\delta)< +\infty$. By Proposition~\ref{prop: compi2} this implies \OOO that  \EEE $h \in BV(\omega;[0,M])$, $u = \OOO\chi_{\Omega_h} \EEE u \in GSBD^2(\Omega)$, as well as $u = u_0$ on $\omega \times (-1,0\OOO]\EEE$.  Therefore, $\mathcal{G}_0(u,h)<+\infty$. Moreover, \eqref{eq:lsc0XXXX} holds.  The lower bound for the elastic energy follows by \eqref{eq:lsc0XXXX}(ii),(iii) and by Lemma \ref{prop:liminfel} applied for  $\Theta_\delta :=\Omega' \EEE \setminus (E_\delta^* \cup \omega_u^\delta)$ and $\Theta :=\Omega'$ for arbitrary $\Omega' \subset \subset \Omega$. \EEE   Therefore, it remains to prove that  
\begin{equation*}
\mathcal{H}^{d-1}(\partial^* \Omega_h \cap \Omega ) + 2 \mathcal{H}^{d-1}(J_u'  \cap \Omega^1_h)  \leq \liminf_{\delta\to 0} \Big( \mathcal{H}^{d-1}\big( \partial \Omega_{h_\delta} \EEE \cap \Omega\big)     + \gamma_\delta\int_{\partial \Omega_{h_\delta} \EEE \cap \Omega} |\bm{A}_\delta|^q \, {\rm d}\mathcal{H}^{d-1}\Big)\,,
\end{equation*}
where  $\bm{A}_\delta$ denotes the second fundamental form corresponding to $\partial \Omega_{h_\delta} \BBB \cap \Omega\EEE$, and   $J_u'$ is defined in \eqref{eq: Ju'}. To this end, we define 
\begin{equation*}
\Gamma_\delta:=  \partial E_{\delta}^*\cap\Omega\,,
\end{equation*}
where \BBB $(E_\delta^*)_{\delta>0}$ \EEE are given in \eqref{eq:lsc0XXXX}, and note that $\sup_{\delta\OOO>0\EEE} \mathcal{H}^{d-1}(\Gamma_\delta)<+\infty$ by \eqref{eq:lsc0XXXX}(iv) \OOO(up to a subsequence, not relabeled). \EEE   We let $U = \omega \times (-\frac{1}{2},M) $  and $U' = \Omega = \omega \times (-1,M+1)$. By \cite[Theorem~2.5]{Crismale} and Corollary \ref{cor: graphi} for $\varphi \equiv 1$, we now  observe that the sets given in \eqref{eq: XCn} are equi-precompact in $L^0(U';\R^d)$. In fact, given $v_\delta \in \mathcal{X}_{C,\delta}$, we define  $w_\delta := \BBB \chi_{\AAA \Omega\setminus \overline{E_\delta'\EEE}} \EEE v_\delta$, where $\partial E_\delta' \cap \Omega$ is  \OOO the  \EEE graph of a function, see Corollary \ref{cor: graphi}. By \eqref{eq: graphiii}(ii) and  the fact that  $\OOO\sup_{\delta>0}\EEE\Vert e(w_\delta) \Vert_{L^2(\Omega)}<+\infty\EEE$, we can apply \cite[Theorem 2.5]{Crismale} to find that $\OOO( \EEE w_\delta\OOO)\EEE_{\OOO \delta>0\EEE}$ converges in measure on $\Omega$ to some $w \in L^0(\Omega;\R^d)$. By \eqref{eq: graphiii}(i) we conclude $v_\delta \to w$ in measure, as well. Therefore, as $(\mathcal{X}_{C,\delta})_{\delta\OOO>0\EEE}$ are equi-precompact, we can   apply Theorem~\ref{thm:compSps}   to  deduce that $(\Gamma_\delta)_{\delta\OOO>0\EEE}$ $\sigma_{\rm sym}^2$-converges  (up to a subsequence)  to some $\Gamma\subset \overline{U} \cap U'$. By combining \eqref{eq:lsc0XXXX}(iv) and Lemma \ref{lemma: vertical sets} for $D_\delta = E_\delta^*$ (note that indeed  $\OOO\mathcal{L}^{d}\EEE((\Omega   \setminus E^*_\delta)  \triangle \EEE \Omega_h) \to 0$ as $\delta \to 0$ by Proposition \ref{prop: compi2}) we get 
\begin{equation}\label{2304191211XXX}
\mathcal{H}^{d-1}(\partial^* \Omega_h \cap \Omega) + 2 \mathcal{H}^{d-1}( \Gamma \cap \Omega_h^1) \leq  \liminf_{\delta\to 0} \Big( \mathcal{H}^{d-1}\big(\partial \Omega_{h_\delta} \EEE  \cap \Omega\big)     + \gamma_\delta\int_{\partial \Omega_{h_\delta} \EEE  \cap \Omega } |\bm{A}_\delta|^q \, {\rm d}\mathcal{H}^{d-1}\Big)\,.
\end{equation}
Thus, to conclude the proof, it remains to check that $J_u'\OOO \cap \Omega_h^1\EEE\subset \Gamma \cap \Omega_h^1$ up to an $\mathcal{H}^{d-1}$-negligible set. To this end, we follow  \cite[Subsection 6.1]{Crismale}: consider the sequence of  \OOO mappings  \EEE  $v_\delta := \psi\OOO\chi_{\Omega \setminus E^*_\delta} \EEE u_\delta$, where $\psi \in C^\infty(\Omega)$ with $\psi = 1$ in a neighborhood of $\Omega^+  = \Omega\cap\lbrace x_d >0\rbrace  $ and $\psi = 0$ on $\omega  \times (-1,-\frac{1}{2})$. Moreover, for $t>0$, we let $v'_\delta(x):= \OOO\chi_{\Omega \setminus E^*_\delta}(x) \EEE v_\delta(x', x_d -t)$,  extended   by zero  in $\omega{\times}(-1, -1+t)$. Defining $v'(x) = \psi\OOO\chi_{\Omega_h}(x) \EEE u(x',x_d-t)$, we observe that $v'_\delta$ converge  to $v'$ in measure on $U'$ since $u_\delta \to u$ in measure on $U'$, see \eqref{eq:lsc0XXXX}(i). \EEE We also observe that  $v'_\delta = 0$ on $U'\setminus U  = \omega \times ((-1,-\frac{1}{2}]\cup[M,M+1)) $. Thus, applying Definition~\ref{def:spsconv}(i) on the sequence $(v'_\delta)_{\delta\OOO>0\EEE}$,  which clearly satisfies $J_{v'_\delta} \BBB \tilde{\subset} \EEE \Gamma_\delta$,   we obtain $J_{v'} \BBB \tilde{\subset} \EEE \Gamma$. This shows 
\begin{equation*}
(J_u+t e_d) \cap \Omega_h^1 = J_{v'} \cap \Omega_h^1  \subset \Gamma \cap \Omega_h^1  \,.
\end{equation*}
Since $t \ge 0$ was arbitrary, recalling   the definition of $J_u' = \lbrace (x',x_d + t)\colon \, x \in J_u, \, t \ge 0 \rbrace$, see \eqref{eq: Ju'},   we indeed find $J_u'  \cap \Omega^1_h\subset \Gamma \cap \Omega_h^1$. In view of \eqref{2304191211XXX}, this concludes the proof of the lower bound. 
\end{step}\\[3pt]
\begin{step}{2}(Recovery sequence) 
By applying \cite[Theorem 2.4]{Crismale}, for each $h \in BV(\omega;[0,M])$ and for each $u = \OOO\chi_{\Omega_h} \EEE u \in GSBD^2(\Omega)$ with $u=u_0$ on $\omega \times (-1,0\OOO]\EEE$  there exists a sequence $(h_\delta)_{\delta\OOO>0\EEE} \subset C^\infty \EEE (\omega;[0,M])$ and \OOO mappings  \EEE $\OOO(u_\delta)\EEE_{\OOO \delta>0\EEE}$ with $u_\delta|_{\Omega_{h_\delta}} \in \BBB H^1\EEE  (\Omega_{h_\delta};\R^d)$, $u_\delta= 0$ on $\Omega \setminus \AAA\overline{\Omega_{h_\delta}}\EEE$, and $u_\delta = u_0$ on $\omega \times (-1,0\OOO]\EEE$     such that $h_\delta \to h$ in $L^1(\omega)$,  $u_\delta \to u$ in $L^0(\Omega;\R^d)$, and
\begin{align}\label{eq: limsupVito2}
\lim_{\delta \to 0} \Big(\OOO\frac{1}{2}\EEE\int_{\BBB \Omega_{h_\delta}^+ \EEE } \mathcal{Q}(e(u_\delta))\, {\rm d}x +  \mathcal{H}^{d-1}(\partial \Omega_{h_\delta} \cap \Omega)  \Big)  = \mathcal{G}_0(u,h)\,.
\end{align}
Strictly speaking, \cite[Theorem 2.4]{Crismale} only ensures that $h_{\OOO\delta\EEE}$ is a $C^1$-function, but in the proof recovery sequences are constructed for profiles of regularity $C^\infty$, see \cite[Lemma 6.4]{Crismale}. Moreover, by a density argument we can assume that each $u_\delta$ is Lipschitz on $\Omega_{h_\delta}^+$. By a diagonal argument we may further suppose without restriction that 
$${ \Vert \nabla u_\delta \Vert_{L^\infty(\Omega_{h_\delta}^+)}   \le \delta^{-1/4} ,\quad \quad  \quad \int_{\partial \Omega_{h_\delta} \cap \Omega} |\bm{A}_\delta|^q \, {\rm d}\mathcal{H}^{d-1} \le \BBB \gamma_\delta^{-1/2} \,, \EEE}
$$
where we use that $\gamma_\delta^{-1/2} \to + \infty$ as $\delta \to 0$. \EEE By \eqref{eq: Gfunctional}, \eqref{eq: limsupVito2}, the fact that \EEE $\gamma_\delta \to 0$, and by applying Lemma \ref{lemma: Taylor} for $\Theta_\delta := \Omega^+_{h_\delta}$, we conclude that $\lim_{\delta \to 0} \mathcal{G}_\delta(\OOO u\EEE_\delta,h_\delta) = \mathcal{G}_0(u,h)$.  Eventually, \cite[Remark 6.8]{Crismale} also shows that a volume constraint  on the film can be taken into account, as mentioned in Remark~\ref{rem: vol2}. 
\end{step}
\end{proof}

We close with the short proofs of Corollaries \ref{cor: mean1} and \ref{cor: mean2}. 

\begin{proof}
In view of the above proofs, we observe that replacing $|\bm{A}_\delta|^2$ by $|\bm{H}_\delta|^2$ does  not affect the $\Gamma$-limit, but is only relevant for the compactness results in Propositions \ref{prop: compi1} and \ref{prop: compi2}, \OOO  respectively. \EEE To proceed as above, in particular in order to obtain \eqref{eq:lsc0}(iv) and \eqref{eq:lsc0XXXX}(iv),  it suffices to check that, under the assumptions given in Corollaries \ref{cor: mean1} and \ref{cor: mean2}, it holds $$\liminf_{\delta \to 0} \gamma_\delta\int_{\partial E_\delta \cap \Omega} |\bm{A}_\delta|^2 \, {\rm d}\mathcal{H}^{d-1} \le \liminf_{\delta \to 0}  \gamma_\delta\int_{\partial E_\delta \cap \Omega} |\bm{H}_\delta|^2 \, {\rm d}\mathcal{H}^{d-1}\,. $$
We refer to the cases (a) and (b) discussed  in Remark~\ref{remark: mean}. 
\end{proof}

\section*{Acknowledgements} 
This work was supported by the DFG project FR 4083/1-1 and by the Deutsche Forschungsgemeinschaft (DFG, German Research Foundation) under Germany's Excellence Strategy EXC 2044 -390685587, Mathematics M\"unster: Dynamics--Geometry--Structure.

\appendix

\EEE

 \section{Some auxiliary lemmata}\label{sec: appi}

  

 \subsection{Two elementary lemmata on \AAA planar \EEE curves}\label{appendix_a}
\begin{lemma}\label{1_dim_curv_est} 
Let $q\geq 1$. For every  closed, planar $C^2$-curve \EEE $\ggamma$ it holds that 
\begin{equation*}
\int_{\ggamma}|\kappa_{\ggamma}|^q\, \mathrm{d}\mathcal{H}^1\geq (\mathrm{diam}\ggamma)^{1-q}\,,
\end{equation*}
 where  $\kappa_{\ggamma}$ denotes the curvature of the curve. 
\end{lemma}
\begin{proof}
Let $\ggamma=(\ggamma_1,\ggamma_2)\colon[0,L_{\ggamma}]\mapsto\mathbb{R}^2$ be an arc-length parametrization of $\ggamma$, where $L_{\ggamma}$ denotes the length of the curve. Without restriction, after a possible translation, we assume  that $\ggamma(0)=\ggamma(L_{\ggamma})=0$. Let also $s_0\in [0,L_{\ggamma}]$ be such that $|\ggamma(s_0)|=\|\BBB \ggamma \EEE \|_{L^{\infty}}$. Since $|\dot {\ggamma}|\equiv 1$ in this parametrization, by integration by parts and H\"older's inequality, we get
\begin{align*}
L_{\ggamma}&=\int_{0}^{L_{\ggamma}}|\dot{\ggamma}|^2\,\mathrm{d}s=-\int_{0}^{L_{\ggamma}}\ggamma\cdot\ddot{\ggamma}\, \mathrm{d}s\leq \|\ggamma\|_{L^\infty}\int_{0}^{L_{\ggamma}}|\ddot{\ggamma}|\,\mathrm{d}s\\
&=|\ggamma(s_0)-\ggamma(0)|\int_{0}^{L_{\ggamma}}|\kappa_{\ggamma}|\, \mathrm{d}s
\leq \mathrm{diam}\ggamma\cdot L_{\ggamma}^{1-1/q}\left(\int_{0}^{L_{\ggamma}}|\kappa_{\ggamma}|^q\, \mathrm{d}s\right)^{1/q}\,.
\end{align*}
This, along with the obvious fact that $L_{\ggamma}\geq \mathrm{diam}\ggamma$ for every closed curve $\ggamma$\AAA, \EEE concludes the proof. 
\end{proof}

We proceed with the proof of Lemma   \ref{lemma: curve graph}.

\begin{proof}[Proof of Lemma \ref{lemma: curve graph}]
Clearly, $\partial E \cap Q_{8\rho}$ can be written as a finite union of pairwise disjoint curves $(\ggamma_i)_{i=1}^N$. We denote by $(\ggamma_i)_{i=1}^M$ the subset of those curves intersecting $Q_{3\rho}$. It suffices to establish the desired properties for one curve only, denoted by $\ggamma$ for simplicity. Additionally, we show that
\begin{align}\label{eq: length curve}
\mathcal{H}^1(\ggamma \cap Q_{8\rho}) \ge \rho\,.
\end{align}
\BBB The \EEE latter, along with \OOO the assumption that \EEE  $\mathcal{H}^1(\partial E\cap Q_{8\rho})\leq \Lambda\rho$, shows that $M \le \Lambda$.

Without restriction, we let ${\ggamma\OOO=(\ggamma_1,\ggamma_2)\EEE}\colon [0,L_{\ggamma}]\mapsto\mathbb{R}^2$ be an arc-length parametrization of $\ggamma$, where $L_{\ggamma}$ denotes the length of the curve. Let $L := {\ggamma}(0) + \R \dot{\ggamma}(0)$. Without restriction, up to an isometry we suppose that $L = \R (1,0)$, \OOO  i.e., $\ggamma(0)=0$, $\dot{\ggamma}(0)=(1,0)$,  \EEE for notational convenience. As $\int_{\partial E\cap Q_{8\rho}}|\bm{A}|\, {\rm d}\mathcal{H}^1\leq \varepsilon $, we get   \EEE that  
\begin{align*}
\text{$\OOO|\EEE\dot{\ggamma}(s) -\dot{\ggamma}(0)\OOO| \EEE = \OOO\Big|\EEE\int_0^s \ddot{\ggamma}(t) \, {\rm d}t\OOO\Big|\leq \int_{0}^{L_{\ggamma}}|\ddot\gamma(t)|\, {\rm d}t\EEE=\int_{\ggamma}|\kappa_{\ggamma}|\, {\rm d}\mathcal{H}^1\EEE\le \eps \le \eps_0$  \quad for all $s \in [0,L_{\ggamma}]$.}
\end{align*}
Thus,  provided that $\eps_0$ is chosen sufficiently small, we get $\dot{\ggamma}_1 \ge 1/2$ and $|\dot{\ggamma}_2| \le \eps$.  Consequently, $\ggamma$ is the graph of a \BBB regular \EEE function $u \colon \AAA \overline{U}\EEE \to L^\perp$ for \AAA an open \EEE segment $ U \subset L$ containing $\ggamma(0) \BBB = 0 \EEE$ satisfying $u(\ggamma(0)) = u'(\ggamma(0)) = 0$, more precisely $u(x) = \ggamma_2( \ggamma_1^{-1}(x))e_2$. This implies \EEE  $u' = \dot{\ggamma_2}/\dot{\ggamma}_1e_2$ and thus \EEE $\Vert u' \Vert_\infty \le 2\eps$. \EEE 

Then, \AAA switching back to a general line $L$ in $\R^2$, \EEE the fundamental theorem of calculus along with the fact that $u(\ggamma(0)) = 0$ and that \OOO  $\mathcal{H}^1(U)=\mathrm{diam}(U)\le 8\sqrt{2}\rho$ \EEE yields $\Vert u \Vert_\infty \le \OOO\Vert u' \Vert_\infty\EEE\mathcal{H}^1(U)\le C_1\eps\rho$. It remains to show \eqref{eq: length curve}. In fact, by $\gamma\cap Q_{3\rho}\neq \emptyset$ and $\Vert u \Vert_\infty \le C_1\eps\rho$, provided that $\eps_0\OOO>0\EEE$ is small enough, we \AAA have \EEE that $L \cap Q_{4\rho} \neq \emptyset$. \EEE Therefore, $\mathcal{H}^1(L \cap Q_{6\rho}) \ge \rho$,  \EEE which along with $\Vert u \Vert_\infty \le C_1\eps\rho$  \OOO implies \EEE the estimate. \end{proof}

\subsection{Lemma\OOO ta \EEE on good and bad planes}\label{appendix_b}

In this subsection, we give  the proofs of Lemmata~\ref{lemma:thetagood}--\ref{lemma: bad plane}. Let $\OOO 0<\EEE\theta < 1/\sqrt{3}$.  \EEE Without restriction, let $Q_\rho$ be the cube centered at $0$, and let $L$ be a plane with normal $\AAA\nu_L:\EEE=\nu=(\nu_1,\nu_2,\nu_3) \in \mathbb{S}^2\EEE$ such that   $(L)_{3\eta\rho} \cap Q_\rho \neq \emptyset$, see \eqref{eq: thick-def}. Before we start with the proofs, we observe the following elementary property: suppose that there exists $k \in \lbrace 1,2,3 \rbrace$  such that  $|\nu_j| \le \theta$ for both $j \neq k$. Then,  we get
\begin{align}\label{eq: goodprep}
|(x-y) \cdot e_k| \le 18\theta\rho \quad \quad \text{for all $x,y \in L\cap Q_{3\rho}$.} 
\end{align} 
Indeed, suppose without restriction \OOO(up to an appropriate reflection if necessary) \EEE that  $\nu_1> \theta$ and $|\nu_2|,|\nu_3|\leq \theta$. Thus, we have $\nu_1 \ge \sqrt{1-2\theta^2}$, and an elementary computation yields
\begin{equation*}
|e_1 - \nu|^2 \le 2\theta^2 + (1- \sqrt{1-2\theta^2})^2 \le 4\theta^2
\end{equation*}
\BBB as \EEE $0<\theta \le 1/\sqrt{2}$. Then, there exists $R_\nu \in SO(3)$ with $R_\nu \nu = e_1$ such that   $|R_\nu - {\rm Id}|^2 = 3|e_1 - \nu|^2 \le 12\theta^2$, i.e., $|R_\nu - {\rm Id}| \le 2\sqrt{3}\theta$. 
 We fix two  arbitrary points $x,y \in L\cap Q_{3\rho}$, and observe that $(x-y) \cdot \nu = 0$. Therefore,   we  compute
\begin{align*}
|(x-y) \cdot e_1| = |(x-y) \cdot R_\nu \nu| \le |(x-y) \cdot \nu| + |x-y||R_\nu - {\rm Id}| \le 2\sqrt{3}\theta |x-y| \le 18\theta \rho\,,
\end{align*}
where in the last step we used \OOO that \EEE $x,y \in Q_{3\rho}$, and therefore $|x-y| \le 3\sqrt{3}\rho$.

\begin{proof}[Proof of Lemma \ref{lemma:thetagood}]
The main step of the proof consists in showing the following statement:
There exists $\theta \in (0,1/\sqrt{3})$ small enough and a constant $C_\theta>0$ such that for any $Q_\rho$  and any $\theta$-good plane $L$ \OOO  for $Q_\rho$  \EEE the following holds: given a function  $v\NNN\in L^\infty(\AAA V \EEE;L^\perp)\EEE$   for some bounded domain $\MMM L\cap Q_\rho \subset \AAA V \EEE \subset L$ \AAA and \EEE $\Vert v \Vert_{L^\infty(\AAA V \EEE)} \le 3\eta\rho$, for all $\rho \le r \le \OOO(1+ 6\OOO\eta)\EEE\rho$ we get that
\begin{align}\label{eq: achivi}
\mathcal{H}^2\big( \omega_v^r  \triangle (L \cap Q_\rho)   \big) \le  \MMM  C_\theta \EEE \eta\rho^2  \,,
\end{align}
where $\omega^r_v := \Pi_L\AAA\big(\EEE\mathrm{graph}(v)\cap {Q_r}\EEE \AAA\big)\EEE$ and $\Pi_L$ denotes the orthogonal projection onto the plane $L$.\\[4pt]
\begin{step}{1}(Reduction to \eqref{eq: achivi})
In fact, once \eqref{eq: achivi} has been  \OOO shown, \EEE the statement can be derived as follows: \OOO \eqref{ineq:area_of_domains_of_parametrization} \EEE is immediate from \eqref{eq: achivi}. For \eqref{eq: goodtheta1}, observe that ${(\partial^{-}S_L)}_{\rm int} := \partial^- S_L \cap {\rm int}(Q_{(1+ 6\eta)\rho})\EEE$ can be expressed  as the graph of the \OOO  constant \EEE function $z \colon L\to L^\perp$ given by $z \equiv 3\eta\rho\nu$, i.e., ${(\partial^{-}S_L)}_{\rm int}  = {\rm graph}(z) \cap {\rm int}(Q_{(1+ 6\eta)\rho})=\omega^{(1+6\eta)\rho}_z+3\eta\rho\nu$\EEE. Then, \AAA by \eqref{eq: achivi} \EEE  we obtain
\begin{align*}
\mathcal{H}^2({(\partial^{-}S_L)}_{\rm int}) & = \mathcal{H}^2(\omega_z^{(1+ 6\eta)\rho}) \le \mathcal{H}^2\big( \omega_z^{(1+ 6\EEE\eta)\rho}  \triangle (L \cap Q_\rho)   \big)  + \mathcal{H}^2(L \cap Q_\rho)\le  \OOO\mathcal{H}^2(L \cap Q_\rho)+C_\theta \eta\rho^2\,.
\end{align*}  
As the normal vector is \OOO constant equal to  \EEE $\pm \nu$ \OOO  both on  \EEE ${(\partial^{-}S_L)\EEE}_{\rm int}$ and $L \cap Q_\rho$, we also get 
\OOO
\begin{equation*}
\mathcal{H}^2_\varphi({(\partial^{-}S_L)\EEE}_{\rm int}) \le  \mathcal{H}^2_\varphi(L \cap Q_\rho)+C_\theta\OOO\varphi_{\mathrm{max}} \EEE \eta\rho^2 ,
\end{equation*}
\EEE where we recall the notation in \eqref{eq: ani-per}. Consequently, to conclude the argument, we need to check that 
\begin{align}\label{eq: goodtheta1XXX}
\mathcal{H}^2\big( \partial^- S_L \setminus {(\partial^{-}S_L)}_{\rm int} \big) \le C_\theta \eta\rho^2\,.  
\end{align}
Then, \eqref{eq: goodtheta1} indeed follows. 
To this end, note that \OOO the set  \EEE $ \partial^- S_L \setminus {(\partial^{-}S_L)}_{\rm int}$ consists of the six (possibly empty) sets  $(L)_{3\eta\rho} \AAA\cap \partial Q_{(1+ 6\eta)\rho}\EEE \cap \lbrace \pm x_k = \frac{1}{2}(1+ 6\EEE\eta)\rho \rbrace$, for $k \in \lbrace 1,2,3\rbrace$. We derive the estimate only for one of these sets. Without restriction let \EEE $W:= (L)_{3\eta\rho}\cap \AAA\partial Q_{(1+ 6\eta)\rho}\EEE \cap \lbrace  x_3 = \frac{1}{2}(1+ 6\EEE\eta)\rho \rbrace$ and suppose \OOO  that  \EEE $W \neq \emptyset$. First, provided that $\eta_0$ is chosen small with respect to $\theta$, we note that this set is nonempty only if  \eqref{eq: assu-bad plane} for $k=3$ does not hold. Therefore, as $L$ is a $\theta$-good plane, we   necessarily  \OOO have  \EEE  $|\nu_1| \ge \theta$ or $|\nu_2| \ge \theta$ \OOO  and thus \EEE $|\nu_3| \le \sqrt{1-\theta^2}$.\\[-7pt]

Let \BBB $W_t:= W \cap \lbrace x\colon (x-x_0) \cdot \nu = t \rbrace$ for some arbitrary $x_0 \in L$. \EEE Note that $\mathcal{H}^1(W_t) = 0$ for $|t|>3\eta\rho$ and $\mathcal{H}^1(W_t) \le \sqrt{2}(1+ 6\EEE\eta)\rho$.    Then, by the coarea formula (see \cite{maggi2012sets}, formula (18.25), applied \EEE with slicing direction $\nu$ in place of $e_n$ and $e_3$ as unit normal to the surface $W$) \EEE we get
\begin{align*}
\sqrt{1-(\nu\cdot e_3)^2}\ \mathcal{H}^2(W)&=\int_{W}\sqrt{1-(\nu\cdot e_3)^2}\, \mathrm{d}\mathcal{H}^2=\int_{\mathbb{R}}\mathcal{H}^1(W_t)\, {\rm d}t\leq 6\eta\rho \cdot \sqrt{2}(1+ 6\EEE\eta)\rho \le C\eta\rho^2\,.
\end{align*}
By using \OOO the  fact that  \EEE $\sqrt{1-(\nu\cdot e_3)^2} \ge \theta$ and by repeating the estimate for all six sets, we indeed get \eqref{eq: goodtheta1XXX}. A similar argument shows that
\begin{align}\label{eq: simila}
\mathcal{H}^2\big( L  \cap  (\overline{Q_{(1+12\eta)\rho}} \setminus Q_\rho) \big) \le C_\theta\eta\rho^2\,. 
\end{align}
We omit the details.
\end{step}\\
\begin{step}{2}(Proof of \eqref{eq: achivi}, preparations)
Let us now show \eqref{eq: achivi}. For convenience, we extend $v$ to a function $w$ defined on $L \cap \overline{Q_{(1+12\eta)\rho}}$ satisfying $\Vert w \Vert_\infty \le 3\eta\rho$. It suffices to show \OOO  that  \EEE for all $\rho \le r \le (1+ 6\EEE\eta)\rho$
\begin{align}\label{eq: achivi22}
\mathcal{H}^2\big( \omega_w^r  \triangle (L \cap Q_\rho)   \big) \le C_\theta \eta\rho^2, 
\end{align}
as then the statement readily follows from the fact that 
\begin{equation*}
\omega_v^r\triangle (L \cap Q_\rho)\subset \big(\omega_w^r\triangle (L \cap Q_\rho)\big)\cup \big((L \cap \overline{Q_{(1+12\eta)\rho}})\setminus {V}\EEE\big)
\end{equation*}
and that  \EEE by \eqref{eq: simila} \MMM  and $L\cap Q_\rho \subset {V}\EEE$ \EEE we have 
$$\mathcal{H}^2\big((L \cap \overline{Q_{(1+12\eta)\rho}}) \setminus {V}\EEE \big) \le  \MMM C_\theta \EEE \eta\rho^2 \,.$$ 
We start with the observation that, in view of $\Vert  w \Vert_{\infty} \le 3\eta\rho$,   for all $\rho \le r \le (1+6\EEE\eta)\rho$    it holds that 
\begin{equation}\label{set_inclusions}
L\cap \overline{Q_{(1-6\eta)\rho}}\subset \overline{\omega^r_w} \EEE \subset L\cap \overline{Q_{(1+12\eta)\rho}}\,.
\end{equation}
\NNN Indeed, to see the left inclusion, for each $x \in L\cap \BBB \overline{Q_{(1-6\eta)\rho}}\EEE$ and every $i \in \lbrace 1,2,3\rbrace$ we  estimate
\begin{equation*}
\left|(x+w(x))\cdot e_i\right|\leq |x\cdot e_i|+\|w\|_{\infty}\leq  \frac{(1-6\eta)\rho}{2}+3\eta\rho = \frac{\rho}{2}\AAA\leq \frac{r}{2}\EEE\,.
\end{equation*}
To see  the right inclusion, for every $\rho\leq r\leq (1+6\eta)\rho$ and  $x\in \overline{\omega^r_w}$,  we estimate  for $i \in \lbrace 1,2,3\rbrace$ 
$$|x_i|\leq |(x+\AAA w\EEE(x)) \cdot e_i|+\|\AAA w\EEE\|_{L^\infty}\leq \frac{r}{2}+3\eta\rho\leq \frac{(1+12\eta)\rho}{2}\,.$$
\EEE
For notational convenience, we let   $\Sigma^r_w:=\omega^r_w \triangle  (L\cap Q_\rho)$. We treat the two possible cases in the definition of $\theta$-good planes separately. 
\end{step}\\
\begin{step}{3}(Proof of \eqref{eq: achivi22}, Case (1))  Let $L$ be a $\theta$-good plane belonging to \BBB Case~(1) in \EEE   \OOO Definition~\ref{theta_good_definition}. \EEE Without restriction we suppose that  ${\rm argmin}_{i=1,2,3}|\nu_i| = 3$. This implies that $|\nu_3|\leq 1/\sqrt 3$ and    $|\nu_1|,|\nu_2|\geq \theta$. For $t \in \R$ and $\rho \le r \le (1+ 6\EEE\eta)\rho$, we introduce the sets 
\begin{align*}
Q_{r}^t:={Q_{r}}\cap \lbrace x_3 =t\rbrace,\ \ \ \omega^{\OOO r, \EEE t}:=\BBB {\overline{\omega^r_w}} \EEE \cap \lbrace x_3 =t\rbrace\ \ \mathrm{and}\ \ L^{t}:=(L\cap \overline{Q_\rho})\cap \lbrace x_3 =t\rbrace\,.
\end{align*}
By \OOO the fact that  \EEE $|\nu_3|\leq 1/\sqrt 3$, the second inclusion in \eqref{set_inclusions},  and by the coarea formula      we have for  $\BBB \rho_\eta\EEE := (1+12\eta)\rho$ that 
\begin{align}\label{eq: firstiX}
\hspace{-1em}\sqrt{\frac{2}{3}} \mathcal{H}^2(\Sigma^r_w)& \le \int_{\Sigma^r_w}\hspace{-0.2em}\chi_{{Q_{\rho_\eta}}}\sqrt{1-(\nu\cdot e_3)^2}\, \mathrm{d}\mathcal{H}^2=\int_{\mathbb{R}}\hspace{-0.2em}\mathcal{H}^1(\Sigma^r_w\cap Q_{\rho_\eta}^t)\, {\rm d}t \le \EEE \int_{-\rho_\eta/2}^{\rho_\eta/2}\hspace{-0.2em}\mathcal{H}^1(\omega^{\OOO r, \EEE t}\triangle L^t)\,{\rm d}t\,,
\end{align}
\BBB where we use that $\nu$ is a unit normal to $\Sigma^r_w$. \EEE 
We now proceed with estimating $\mathcal{H}^1(\omega^{\OOO r, \EEE t}\triangle L^t)$ for $|t| \le \rho_\eta/2$. To this end, fixing some $z \in L$, we first introduce a parametrization of the \OOO one dimensional \EEE sets $\omega^{\OOO r, \EEE t}$ and $L^t$. First, for $\OOO s \EEE \in \R$   we introduce 
\begin{equation}\label{eq bbb} 
X^t(\OOO s\EEE) := \left(\OOO s\EEE, -\frac{\nu_1}{\nu_2}\OOO s\EEE+b_{t,\nu},t\right),\ \mathrm{where}\ \  b_{t,\nu}:=\frac{z\cdot\nu-t\nu_3}{\nu_2},\,
\end{equation}
and we observe that $L \cap \lbrace x_3=t\rbrace = \lbrace X^t(\BBB s \EEE )\colon \BBB s \EEE \in \R\rbrace$ since $X^t(\OOO s \EEE )\cdot \nu = z\cdot \nu$.  Thus, for \EEE  $|t| \le \rho/2$  it holds that 
\begin{equation*}
L^{t}=\left\{X^t( \BBB s \EEE )\colon \BBB s \EEE \in I^t_L\right\}\,,\ \text{where } I^{t}_L:=\left[ -\frac{\rho}{2},\frac{\rho}{2}\right]\cap\left[ -\frac{|\nu_2|}{|\nu_1|}\frac{\rho}{2}+\frac{\nu_2}{\nu_1}b_{t,\nu},\frac{|\nu_2|}{|\nu_1|}\frac{\rho}{2}+\frac{\nu_2}{\nu_1}b_{t,\nu}\right]
\end{equation*}
and $L^t = \emptyset$ for $|t|>\rho/2$. In a similar fashion, we obtain $\omega^{\OOO r, \EEE t} = \emptyset$ for $|t|>\rho_\eta/2$ and for $|t| \le \rho_\eta/2$ we get
$\omega^{\OOO r, \EEE t} =\{X^t(\OOO s\EEE) \colon \OOO s \EEE \in {I}^{r,t}_\omega \EEE \}$, where 
\begin{equation*}
I^{r,t}_\omega := \left\{ \OOO s  \EEE \colon \, |\OOO s\EEE+w_1(X^{\AAA t \EEE}(\OOO s\EEE))|\leq \frac{r}{2},\ \left|\left(-\frac{\nu_1}{\nu_2}{\OOO s \EEE }+b_{\AAA t \EEE,\nu}\right)+w_2(X^{\AAA t \EEE}(\OOO s\EEE))\right|\leq\frac{r}{2},\  \AAA\left|t+w_3(X^{\AAA t \EEE}(\AAA s))\right|\leq\frac{r}{2}\EEE\right\}
\end{equation*}
where $w_k$ denotes the $k$-th component of $w$.  Here, we have again used the second inclusion in \EEE \eqref{set_inclusions}. By the area formula we get \EEE 
\begin{equation*}
\mathcal{H}^1(\omega^{r,t}\triangle L^{t}) \le \sqrt{1 +\Big(\frac{\nu_1}{\nu_2}\Big)^2}\,\mathcal{H}^1(I^{r,t}_\omega \EEE \triangle I^{t}_L) \text{\ \ for \ all\ } |t|\le \rho_\eta/2\,.
\end{equation*}
Then, \EEE from \eqref{eq: firstiX} and the fact that   $|\nu_1|\leq 1, |\nu_2|\geq \theta$ we derive 
\begin{align}\label{eq: firstiXX}
 \mathcal{H}^2(\Sigma^r_w)&\le \sqrt{\frac{3}{2}} \int_{-\frac{\rho_\eta}{2}}^{\frac{\rho_\eta}{2}}\mathcal{H}^1(\BBB \omega^{r,t} \EEE \triangle L^t)\,{\rm d}t \le \frac{\sqrt{{3}} }{\theta} \int_{-\frac{\rho_\eta}{2}}^{\frac{\rho_\eta}{2}}\mathcal{H}^1(I^{r,t}_\omega \EEE \triangle I^{t}_L)\,{\rm d}t\,.
\end{align}
A careful inspection of the definition of $I^{r,t}_\omega$ and $I^t_L$ implies that \EEE
\begin{align}\label{ineq:ItildawithoutI}
\mathcal{H}^1( I^{r,t}_\omega\setminus I^t_L)\leq \begin{cases} (r - \rho) + 2\|w\|_\infty + \frac{|\nu_2|}{|\nu_1|}\big( (r-\rho) +  2\|w\|_\infty   \big),&\text{if } |t|\leq \frac{\rho}{2}\,;\\
r + 2\Vert w \Vert_\infty,&\text{if } \text{$\frac{\rho}{2}<|t| \le \frac{\rho_\eta}{2}$\,,}
\end{cases}
\end{align}
\MMM as well as \EEE
\begin{align}\label{ineq:IwithoutItilda}
\mathcal{H}^1( I^t_L\setminus I^{r,t}_\omega) \leq \begin{cases} 2\|w\|_\infty + \frac{|\nu_2|}{|\nu_1|}   2\|w\|_\infty  ,  & \text{if } |t|\leq \frac{\rho}{2}-\|w\|_\infty\,;\\
\rho,  & 
\text{\OOO if } \ \OOO\text{$\frac{\rho}{2}-\|w\|_{\infty}<|t|\leq \frac{\BBB \rho}{2}$.\EEE}
\end{cases}
\end{align}
Combining \eqref{ineq:ItildawithoutI}--\eqref{ineq:IwithoutItilda} and \MMM using \EEE the fact that \EEE  $|\nu_1|\geq \theta$, \MMM $|\nu_2|\leq1$,  \EEE  and $r\leq \rho_\eta$ \EEE   we get 
\begin{align*}
\int_{-\frac{\rho_\eta}{2}}^{\frac{\rho_\eta}{2}}\mathcal{H}^1(\OOO I^{r,t}_\omega\triangle I^t_L\EEE)\,{\rm d}t  & \le  (2\rho - 2\Vert w\Vert_\infty)\Big( (\rho_\eta - \rho) + 2\|w\|_\infty + \frac{1}{\theta}\big( (\rho_\eta-\rho) +  2\|w\|_\infty  \big)\Big)    \\ & \ \ \  + (\rho_\eta-\rho) (\rho_\eta + 2\Vert w \Vert_\infty) + \BBB 2\Vert w \Vert_\infty \EEE \rho\,,
\end{align*}
and, since  $\|w\|_\infty \leq \OOO 3\EEE\eta\rho$ and $\rho_\eta =  (1+12\eta)\rho$, we conclude by recalling  \eqref{eq: firstiXX} that
$$
\mathcal{H}^2\big(\omega^r_w \triangle (L\cap Q_\rho)\big) =  \mathcal{H}^2(\Sigma^r_w)  \leq C_\theta \eta\rho^2 $$
for a constant $C_\theta>0$. This concludes the proof of \eqref{eq: achivi22} in \BBB Case~(1). \EEE
\end{step}\\[3pt]
\begin{step}{4}(Proof of \eqref{eq: achivi22}, Case (2)) 
 Let now $L$ be a $\theta$-good plane \OOO for $Q_{\rho}$ \EEE belonging to Case~(2) in  \OOO Definition \ref{theta_good_definition}, \EEE  i.e., \OOO there exists $k\in \{1,2,3\}$ such that \EEE  $|\nu_k| \geq \theta$ and  
\begin{align}\label{eq: wdh}
\mathrm{dist}\big(L \cap Q_{3\rho}, \lbrace x_k = - \rho/2 \rbrace \cup \lbrace x_k = \rho/2 \rbrace \big) \geq 20\theta\rho\,.
\end{align}
Without restriction, we suppose that $k=3$ and that $|\nu_1|,|\nu_2|< \theta$ as otherwise Case~(1) of the definition applies. 
We start by observing that  \eqref{set_inclusions} yields the estimate
\begin{align}\label{ineq:set_incl}
\begin{split}
\mathcal{H}^2\big(\omega^r_w\triangle (L\cap Q_\rho)\big)&\leq \mathcal{H}^2\big((L\cap \OOO\overline{Q_{(1+12\eta)\rho}}\EEE)\setminus (L\cap Q_\rho)\big)+ \mathcal{H}^2\big((L\cap Q_\rho)\setminus (L\cap \OOO\overline{Q_{(1-6\eta)\rho}})\big)
\\
&= \mathcal{H}^2\big(L\cap \OOO{\overline{Q_{(1+12\eta)\rho}}}\EEE\big)-\mathcal{H}^2\big(L\cap \OOO\overline{Q_{(1-6\eta)\rho}}\EEE\big)\,.
\end{split}
\end{align}
for $\rho \le r \le (1+ 6\EEE\eta)\rho$.  In view of \eqref{eq: wdh}  \BBB and the fact that  $\dist(L,Q_\rho) \le 3\eta \rho$  by definition, \EEE \eqref{eq: goodprep} implies for $\eta_{\AAA 0\EEE}$ small with respect to $\theta$ that 
\begin{align}\label{eq: graph-loc}
L \cap \big( [-r/2,r/2]^2  \times \R \big) \subset \overline{Q_r} \quad \text{for all $(1- 6\eta)\rho \le \EEE r \le 3\rho$}\,.  
\end{align}
Let us denote by $h_L\colon \R^2 \to \R$ the affine function with ${\rm graph}(h_L) = L$. Observe that $\nabla h_L \equiv (- \nu_1/\nu_3,-\nu_2/\nu_3)$ and therefore
\begin{align}\label{eq: constdervi}
\sqrt{1 + |\nabla h_L|^2} =  1/|\nu_3| \,.
\end{align} 
Now,  by the area formula, \BBB \eqref{eq: graph-loc}, \EEE \OOO and \eqref{eq: constdervi}  \EEE we find 
\begin{align*}
\mathcal{H}^2(L\cap \OOO\overline{Q_{(1+12\eta)\rho}}\EEE) =  \int_{ (-\rho/2-6\eta\OOO\rho\EEE,\ \rho/2+6\eta\OOO\rho\EEE)^2  }\sqrt{1+|\nabla h_L|^2}\, {\rm d}\mathcal{H}^2= \OOO\frac{(1 + 12\eta)^2}{|\nu_3|}\EEE\rho^2\,,
\end{align*}
and in a similar fashion 
\begin{align*}
\mathcal{H}^2(L\cap \OOO\overline{Q_{(1-6\eta)\rho}})&= \OOO\frac{(1 -6\eta)^2}{|\nu_3|}\EEE\rho^2\,.
\end{align*}
Combining the previous two  \OOO equalities  with \eqref{ineq:set_incl}, \EEE we conclude
$$\mathcal{H}^2\big(\omega^r_w\triangle (L\cap Q_\rho)\big) \le \big( (1 + 12\eta)^2\rho^2 - (1-6\eta)^2\rho^2 \big) /|\nu_3| \le C_\theta\eta\rho^2 $$
for a constant $C_\theta >0$ depending \OOO only  \EEE on $\theta$, where in the last step we used that $|\nu_3|\geq \sqrt{1-2\theta^2}$. This concludes the proof of \eqref{eq: achivi22}.  
\end{step}
\end{proof}

\begin{proof}[Proof of Lemma \ref{lemma: bad plane}]
Let $L$ be $\theta$-bad plane \BBB for $Q_\rho$. \EEE Let $k \in \lbrace 1,2,3\rbrace$ be  such that \OOO $|\nu_k|\ge \theta$ and  \EEE $|\nu_j| \le \theta$ for $j \neq k$. Since  \eqref{eq: assu-bad plane} does not hold, we get $\mathrm{dist}\big(L \cap Q_{3\rho},  \lbrace x_k = \pm\rho/2\rbrace    \big) < 20\theta\rho$,  where $\pm$ is a placeholder for $+$ or $-$.  Thus, we find $x_0 \in L \cap Q_{3\rho}$ such that $ |(x_0 \pm \rho/2) \cdot e_k|<20\theta\rho$. This along with \eqref{eq: goodprep} shows that $ |(x \pm \rho/2) \cdot e_k| < 38\theta\rho$ for all $x \in L \cap Q_{3\rho}$. For $\theta \BBB < \EEE 1/152$, we obtain the statement. 
\end{proof}

\subsection{\EEE Rigidity estimate on cubic sets}\label{appendix_c}

\AAA Here, we \EEE give the proof of Proposition  \ref{lemma:chain}. Recall the notation introduced in \eqref{eq: r-cubic set}.

\begin{proof}[Proof of Proposition  \ref{lemma:chain}] 
\QQQ We give the argument in detail for \eqref{eq: RRR}$\rm(i)$, and only sketch the proof for \eqref{eq: RRR}$\rm(ii)$, which can be derived along similar lines. \EEE For convenience, we drop the index $r$ and simply write $Q$ for cubes $Q \in \mathcal{Q}_r$. Let us fix $Q,Q' \in \mathcal{Q}_r(\OOO U\EEE)$ with $\mathcal{H}^{d-1}(\partial Q \cap \partial Q') >0$. By applying  \cite[Theorem 3.1] {friesecke2002theorem} for $y$ and $\mathrm{int}(Q)$ or ${\rm int}({Q \cup Q'})$\AAA, \EEE respectively,  there exist $R_Q, R_{Q,Q'} \in SO(d)\EEE$  such that
\begin{align}\label{ineq:Qi}
\int_{Q} |\nabla y-R_Q|^2\,\mathrm{d}x &\leq C\int_{Q} \mathrm{dist}^2(\nabla y, SO(d))\,\mathrm{d}x\,,\\
\int_{Q\cup Q'} |\nabla y-{R}_{Q,Q'}|^2\,\mathrm{d}x &\leq C\int_{Q\cup Q'} \mathrm{dist}^2(\nabla y, SO(d))\,\mathrm{d}x \label{ineq:QiQip1}
\end{align}
for  an absolute  constant $C>0$. Then,  \EEE due to \eqref{ineq:Qi} and \eqref{ineq:QiQip1}, we have
\begin{align*}
r^d|R_Q -{R}_{Q,Q'}|^2 = \int_{Q} |R_Q-{R}_{Q,Q'}|^2\,\mathrm{d}x &\leq 2\left( \int_{Q} |R_Q-\nabla y|^2\,\mathrm{d}x + \int_{Q\cup Q'} |{R}_{Q,Q'}-\nabla y|^2\,\mathrm{d}x\right)\\&\leq C\int_{Q\cup Q'} \mathrm{dist}^2(\nabla y,SO(d))\,\mathrm{d}x\,.
\end{align*} 
The same argument can be repeated with $Q'$ in place of $Q$ for a corresponding $R_{Q'} \in SO(d)$  \EEE to obtain an estimate on $|R_{Q'} -{R}_{Q,Q'}|^2$. Then, we obtain
\begin{align}\label{ineq:RiRip1}
r^d|R_Q-R_{Q'}|^2 \leq C\int_{Q\cup Q'} \mathrm{dist}^2(\nabla y,SO(d))\,\mathrm{d}x\,.
\end{align}
Based on this, we  \EEE compare $R_Q$ and $R_{Q'}$ for arbitrary $Q,Q' \in \mathcal{Q}_r(\OOO U\EEE)$, $Q \neq Q'$.  \EEE  We show that
\begin{align}\label{ineq:RiRj}
r^d\max_{Q,Q'}|R_Q-R_{Q'}|^2 \leq  C N  \EEE \int_{(\OOO U\EEE)^r} \mathrm{dist}^2(\nabla y,SO(d))\,\mathrm{d}x\,,
\end{align}
where for notational convenience we have set $N:=  \#  \mathcal{Q}_r(\OOO U\EEE) $.  \EEE 
To this end, we consider  $Q,Q'\in \mathcal{Q}_r(\OOO U\EEE)$, \OOO $Q\neq Q',$ \EEE and let $\{Q_0,\ldots,Q_M\} \subset \mathcal{Q}_r(\OOO U\EEE)$ be a \BBB simple \EEE path, i.e., $Q_0 =Q$, $Q_M=Q'$, $Q_i \neq Q_j$ for all $i\neq j$\AAA, \EEE and  \EEE  $\mathcal{H}^{d-1}(\partial Q_i \cap \partial Q_{i+1}) >0$ \OOO for all $i=0,\dots,M-1$. \EEE Here, we \BBB use \EEE that $(\OOO U\EEE)^r$ is connected. Clearly, we have $M\leq N$.  \EEE Then, due to  \eqref{ineq:RiRip1} \BBB and \OOO the Cauchy-Schwarz  \EEE  inequality,  \EEE  we obtain
\begin{align*}
r^d |R_Q-R_{Q'}|^2 &= r^d\big|\sum\nolimits_{i=0}^{M-1}(R_{Q_{i+1}}-R_{Q_i})\big|^2 \leq r^d M \sum\nolimits_{i=0}^{M-1} |R_{Q_{i+1}}-R_{Q_i}|^2\\& \leq  C N \sum\nolimits_{i=0}^{M-1} \int_{Q_i \cup Q_{i+1}}\mathrm{dist}^2(\nabla y,SO(d))\,\mathrm{d}x \leq C \BBB N  \EEE \int_{(\OOO U\EEE)^r}\mathrm{dist}^2(\nabla y,SO(d))\,\mathrm{d}x\,.
\end{align*}
As  \OOO the \EEE  choice of  the cubes $Q, Q'\in \mathcal{Q}_r(\OOO U)$ \EEE was  arbitrary, we indeed  \EEE get \eqref{ineq:RiRj}. We are now in the position to prove the statement  for $R=R_{Q^*}\in SO(d)$ for some arbitrary $Q^* \in \mathcal{Q}_r(\OOO U\EEE)$\EEE. Indeed, by using \eqref{ineq:Qi} and \eqref{ineq:RiRj} we have \EEE
\begin{align*}
\int_{(\OOO U\EEE)^r} |\nabla y -R|^2\,\mathrm{d}x &= \hspace{-0.2cm} \sum_{Q \in \mathcal{Q}_r(\OOO U\EEE)}\int_{Q} |\nabla y-R|^2\,\mathrm{d}x  \leq 2\hspace{-0.15cm}\sum_{Q \in \mathcal{Q}_r(\OOO U\EEE)}\left(\int_{Q} |\nabla y-R_Q|^2 \,\mathrm{d}x + r^d\max_{Q,Q'} |R_Q-R_{Q'}|^2\right)\\ 
&\leq 2C\sum\nolimits_{Q \in \mathcal{Q}_r(\OOO U\EEE)}\int_{Q} \mathrm{dist}^2(\nabla y,SO(d)) \,\mathrm{d}x +2N r^d\max_{Q,Q'} |R_Q-R_{Q'}|^2 \\
&\leq C\int_{(\OOO U\EEE)^r} \mathrm{dist}^2(\nabla y, SO(d)) \,\mathrm{d}x  + C N^2  \EEE \int_{(\OOO U\EEE)^r}\mathrm{dist}^2(\nabla y,SO(d)) \,\mathrm{d}x  \\
&\leq C N^2  \EEE   \int_{(\OOO U\EEE)^r}\mathrm{dist}^2(\nabla y,SO(d)) \,\mathrm{d}x \,.
\end{align*}
In view of $N = \# \mathcal{Q}_r(\OOO U\EEE)$, this concludes the proof of \eqref{eq: RRR}. It remains to observe that one can choose $R = {\rm Id}$ if there exists $Q \in \mathcal{Q}_r(\OOO U\EEE)$ with   $\mathcal{L}^d(Q \cap\lbrace \nabla y = {\rm Id}\rbrace) \ge cr^d$. Indeed, by \eqref{ineq:Qi} one gets $\mathcal{L}^d(Q \cap\lbrace \nabla y = {\rm Id}\rbrace) |R_Q-{\rm Id}|^2 \le C\int_{Q} \mathrm{dist}^2(\nabla y, SO(d))\,\mathrm{d}x$ and therefore \eqref{ineq:Qi} holds for ${\rm Id}$ in place of $R_Q$, for $C$ also depending on $c$. This, along with the fact that $R=R_{Q^*}\in SO(d)$ can be chosen for an arbitrary $Q^* \in \mathcal{Q}_r(\OOO U\EEE)$, concludes the proof \QQQ of \eqref{eq: RRR}$\rm{(i)}$. 

The proof of \eqref{eq: RRR}$\rm{(ii)}$ follows analogously, as a direct consequence of the following version of the Poincar{\' e} inequality on the cubic set $(U)^r$: 

\textit{In the setting of Proposition \ref{lemma:chain}, there exists an absolute constant $C>0$ (independent of $U$ and $r$) such that for every $v\in H^1((U)^r;\R^d)$ there exists a vector $b_v\in \R^d$ such that}
\begin{equation}\label{eq: Poincare_on_cubic_sets}
\textit{$r^{-2}\int_{(U)^r}|v(x)-b_v|^2\leq C(\#\mathcal{Q}_r(U))^2\int_{(U)^r}|\nabla v|^2\,\mathrm{d}x\,.$}
\end{equation}
Once \eqref{eq: Poincare_on_cubic_sets} is established, its application for  $v(x):=y(x)-Rx$ along with \eqref{eq: RRR}$\rm{(i)}$ implies \eqref{eq: RRR}$\rm{(ii)}$. For the proof of \eqref{eq: Poincare_on_cubic_sets}, note that for every $Q\in \mathcal{Q}_r$, Poincar\'e's  inequality in $Q$  gives a vector $b_Q\in \R^d$ for which
\begin{equation*}
r^{-2}\int_{Q}|v-b_Q|^2\,\mathrm{d}x\leq C\int_{Q}|\nabla v|^2\,\mathrm{d}x\,,
\end{equation*}
where $C>0$ is an absolute constant. The proof can then be performed following the same steps as the proof of \eqref{eq: RRR}$\rm{(i)}$ above, with the obvious adaptations.
\end{proof}

\subsection{Generalized special functions of bounded deformation}\label{sec: GSBD}

Let $U\subset \R^d$ be open. A function $v\in L^1(U;\R^d)$ belongs to the space of \emph{functions of bounded deformation},  denoted by $BD(U)$,   if 
 the distribution  $\mathrm{E}v := \frac{1}{2}( \mathrm{D}v + (\mathrm{D}v)^{\mathrm{T}} )$ is a bounded $\R^{d\times d}_{\rm sym}$-valued Radon measure on $U$,  
where $\mathrm{D}v=(\mathrm{D}_1 v,\dots, \mathrm{D}_d v)$ is the distributional differential.
For $v\in BD(U)$,  the jump set  $J_v$ is countably $\mathcal{H}^{d-1}$-rectifiable (in the sense of \cite[Definition~2.57]{Ambrosio-Fusco-Pallara:2000}) and  it holds  that $\mathrm{E}v=\mathrm{E}^a v+ \mathrm{E}^c v + \mathrm{E}^j v$, 
where $\mathrm{E}^a v$ is absolutely continuous with respect to $\mathcal{L}^d$, $\mathrm{E}^c v$ is singular with respect to $\mathcal{L}^d$ and such that $|\mathrm{E}^c v|(B)=0$ if $\mathcal{H}^{d-1}(B)<\infty$, while $\mathrm{E}^j v$ is concentrated on $J_v$. The density of $\mathrm{E}^a v$ with respect to $\mathcal{L}^d$ is denoted by $e(v)$. The space $SBD(U)$ is the subspace of all functions $v\in BD(U)$ such that $\mathrm{E}^c v=0$.

We now come to the definition of the space of \emph{generalized functions of bounded deformation}  $GBD(U)$ and of  \emph{generalized special functions of bounded deformation}  $GSBD(U)\subset GBD(U)$. These spaces have been introduced and investigated in \cite{DalMaso:13}. We first state the definition, see  \cite[Definitions~4.1 and 4.2]{DalMaso:13}.

\begin{definition}
Let $U\subset \R^d$ be a  bounded open set,  and let  $v\colon U\to \R^d$ be measurable. We introduce the notation  
$${
\Pi^\xi:=\{y\in \R^d\colon y\cdot \xi=0\},\ \ B^\xi_y:=\{t\in \R\colon y+t\xi \in B\} \ \ \ \text{ for any \AAA $B\subset \R^d$, \KKK $\xi \in \mathbb{S}^{d-1}$, $y\in \Pi^\xi$}\EEE\,,}
$$
and for every  $t\in B^\xi_y$ we let
\begin{equation*}
v^\xi_y(t):=v(y+t\xi),\qquad \widehat{v}^\xi_y(t):=v^\xi_y(t)\cdot \xi\,.
\end{equation*}
Then, $v\in GBD(U)$ \AAA iff \EEE there exists a nonnegative bounded Radon measure $\lambda_v$ on $U$ such that $\widehat{v}^\xi_y \in BV_{\mathrm{loc}}(U^\xi_y)$ for $\mathcal{H}^{d-1}$-a.e.\ $y\in \Pi^\xi$, and for every Borel set $B\subset U$ 
\begin{equation*}
\int_{\Pi^\xi} \Big(\big|\mathrm{D} {\widehat{v}}_y^\xi\big|\big(B^\xi_y\setminus J^1_{{\widehat{v}}^\xi_y}\big)+ \mathcal{H}^0\big(B^\xi_y\cap J^1_{{\widehat{v}}^\xi_y}\big)\Big) \, {\rm d}\mathcal{H}^{d-1}(y)\leq \lambda_v(B)\,,
\end{equation*}
where
$J^1_{{\widehat{\AAA v\EEE}}^\xi_y}:=\left\{t\in J_{{\widehat{\AAA v\EEE}}^\xi_y} : |[{\KKK\widehat{v}}\EEE_y^\xi]|(t) \geq 1\right\}$.
Moreover, \AAA $v$ \EEE belongs to $GSBD(U)$ \AAA iff \EEE $v\in GBD(U)$ and $\widehat{v}^\xi_y \in SBV_{\mathrm{loc}}(U^\xi_y)$ for 
every
$\xi \in \mathbb{S}^{d-1}$ and for $\mathcal{H}^{d-1}$-a.e.\ $y\in \Pi^\xi$.
\end{definition}
\EEE

Every $v\in GBD(U)$ has an \emph{approximate symmetric gradient} $e(v)\in L^1(U;\R^{d\times d}_{\rm sym})$ and an \emph{approximate jump set} $J_v$ which is still 
\BBB countably $\mathcal{H}^{d-1}$-rectifiable \EEE
(cf.~\cite[Theorem~9.1,  Theorem~6.2]{DalMaso:13}).   The notation for $e(v)$ and $J_v$, which is the same as that one in the $SBD$ case, is consistent: in fact, if $v$  lies in  $SBD(U)$, the objects coincide, up to  negligible sets of points with respect to $\mathcal{L}^d$ and $\mathcal{H}^{d-1}$, respectively. The subspace $GSBD^2(U)$ is  given by  
\begin{equation*}
GSBD^2(U):=\{v\in GSBD(U)\colon e(v)\in L^2(U;\R^{d\times d}_{\rm sym}),\, \mathcal{H}^{d-1}(J_v)<\infty\}.
\end{equation*}
If $U$ has Lipschitz boundary, for each $v\in GBD(U)$ the traces on $\partial U$ are well defined, see~\cite[Theorem~5.5]{DalMaso:13},   in the sense that for $\mathcal{H}^{d-1}$-a.e.\ $x \in  \partial U  $ there exists ${\rm tr}(v)(x) \in \R^d$ such that 
\begin{align*}
\lim_{\eps \to 0}\eps^{-d} \mathcal{L}^d\big(U \cap B_\eps(x)\cap \{|v-{\rm tr}(v)(x)|>\varrho\}\big)=0 \quad \text{ for all $\varrho >0$}.
\end{align*}
 We close this short subsection with a compactness result in $GSBD^2(U)$, see \cite[Theorem 1.1]{Crismale-Cham}.
 \begin{theorem}[$GSBD^2$ compactness]\label{th: GSDBcompactness}
 Let  $U \subset \R^d$  be an open, bounded set,  and let $(u_n)_{\AAA n\in \mathbb{N}\EEE} \subset  GSBD^2(U)$ be a sequence satisfying
$$ \sup\nolimits_{n\in \N} \big( \Vert e(u_n) \Vert_{L^2(U)} + \mathcal{H}^{d-1}(J_{u_n})\big) < + \infty.$$
Then, there exists a subsequence, still denoted by $(u_n)_{\AAA n\in \mathbb{N}\EEE}$, such that the set  $\omega_u \EEE := \lbrace x\in U\colon \, |u_n(x)| \to \infty \rbrace$ has finite perimeter, and  there exists  $u \in GSBD^2(U)$  with $u= 0$ on $\omega_u$ such that 
\begin{align*}
{\rm (i)} & \ \ u_n \to u \  \ \ \ \text{ in } L^0(U \setminus \omega_u; \R^d), \notag \\ 
{\rm (ii)} & \ \ e(u_n) \rightharpoonup e(u) \ \ \ \text{ weakly  in } L^2(U \setminus \omega_u; \R^{d \times d}_{\rm sym}),\notag \\
{\rm (iii)} & \ \ \liminf_{n \to \infty} \mathcal{H}^{d-1}(J_{u_n}) \ge \mathcal{H}^{d-1}(J_u \cup  (\partial^* \omega_u \cap U)  ).
\end{align*}
In the language of \cite[Subsection 3.4]{Crismale}, we say that $u_n \to u$ weakly in $GSBD^2_\infty(U)$.
\end{theorem}

\typeout{References}

\end{document}